\documentclass{amsart}

 \usepackage[linktocpage=true]{hyperref}

\usepackage{microtype}

\usepackage{amsmath}

\usepackage{amssymb}
\usepackage{mathrsfs}
\usepackage[shortalphabetic]{amsrefs}
\usepackage{dirtytalk}
\usepackage{graphicx}
\usepackage{enumitem}
\usepackage[cal=boondoxo,scr=boondoxo]{mathalpha}

\newtheorem{thm}{Theorem}[section]

\newtheorem{lem}[thm]{Lemma}
\newtheorem{prop}[thm]{Proposition}

\newtheorem{rmk}[thm]{Remark}

\DeclareMathOperator{\odd}{odd}

\DeclareMathOperator{\K}{K}

\DeclareMathOperator{\G}{G}

\DeclareMathOperator{\Ku}{Ku}
\DeclareMathOperator{\der}{d}
\DeclareMathOperator{\mq}{\mathfrak{q}} 
\DeclareMathOperator{\ma}{\mathfrak{a}} 
 \DeclareMathOperator{\cq}{c}

 \renewcommand{\footnoterule}{\kern -3pt \hrule width 0.4\columnwidth \kern 2.6pt} 

\begin{document}

\title[Distribution of square-free palindromes]{Distribution of square-free palindromes\textsuperscript{1}
}

\author{Aleksandr Tuxanidy}

\address{Independent Scholar}

\email{aleksandr.tuxanidy@gmail.com}

\begin{abstract}
An exponent of distribution $1/16$ is established for square-free palindromes.
The main input is an upper bound for the number of palindromes, in arithmetic
progressions to large moduli, divisible by large squares. 
Our argument combines a simplifying reformulation with exponential-sum estimates, recent work on $6$-almost-prime palindromes, and the large sieve with square moduli of Baier-Zhao.

	\end{abstract}

\maketitle
\setcounter{footnote}{1}

\footnotetext[1]{{\em In memory of Juana Garc\'{i}a Castillo }}

\begingroup
\setlength{\parskip}{0pt}
\setlength{\parindent}{1em} 
\tableofcontents
\endgroup

\newpage

\section*{Notation}\label{sec:notation}

$\mathscr{P}_b(x)$ is the set of all $b$-palindromes at most $x$.

$\mathscr{P}_b^*(x)$ is the set of all numbers in $\mathscr{P}_b(x)$ coprime to $b^3-b$.

$\Pi_b(L)$ is the set of all $b$-palindromes in the interval $[b^L, b^{L+1})$.

$d_j(n)$ is the $j$-th digit of $n$ in its $b$-adic expansion $n = \sum_{j\geq 0}d_j(n)b^j$.

$X = b^L$ is a large variable. 

For $n \in [0, bX)$, $\rho(n) = X\sum_{j\geq 0}d_j(n)b^{-j}$ reverses the first $L+1$ digits of $n$.

$\odd(n)$ denotes the odd part of an integer $n$.

$|\mathcal{A}| = \#\mathcal{A}$ denotes the cardinality of a finite set $\mathcal{A}$.

$\sum_{n \sim N} = \sum_{N/2 < n \leq N}$. 

$\sum_{a(q)}$ is a sum over all residue classes $a$ modulo $q$.

$\sum_{a(q)}^*$ is a sum over all invertible residue classes $a$ modulo $q$.

$e(\alpha) = e^{2\pi i \alpha}$ is the complex exponential and $e_q(\alpha) = e(\alpha/q)$. 

$\G^*(a;q) = \frac{1}{\sqrt{q}}\sum_{n(q)}^*e_q\left(an^2\right)$ is a Gauss sum with coprimality constraint.

$\K_2(c,d;q) = \frac{1}{\sqrt{q}}\sum_{n(q)}^*e_q\left(c\overline{n}^2 + dn\right)$ is the quadratic Kloosterman sum. 

$\mu$ is the M\"{o}bius function.

$\tau(n) = \sum_{d \mid n }1$ is the divisor function.

$\varphi$ is Euler's totient function.

$(m,n)$ is the greatest common divisor (GCD) of $m,n$ (unless stated otherwise).

$\|f\|_1$, $\|f\|_\infty$ are the $L^1$-norms and $L^{\infty}$-norms of a function $f : \mathbb{R} \to \mathbb{C}$.

The Fourier transform of $f$ is normalized as $\widehat{f}(\xi) = \int_{\mathbb{R}}f(t)e(-\xi t)\der t$.

Given a statement $\mathcal{S}$, $\mathbf{1}_{\mathcal{S}}$ equals $1$ if $\mathcal{S}$ is true and is zero otherwise. 

$X \ll Y$, $Y \gg X$ and $X = O(Y)$ all signify $|X| \leq C|Y|$ for some constant $C > 0$. 

$X \asymp Y$ means $X \ll Y \ll X$. 

Dependence of implied constants on parameters is denoted by a subscript.

\section{Introduction}

In the year 1752, Euler highlighted in a letter to Goldbach several of the initial primes of the form $n^2+1$ (Conrad \cite{Conrad}).
In 1912, Landau \cite{Landau1912} singled out the question of whether there are infinitely many such primes, as one of his four ``unattackable'' problems about primes.
More than a century later, the problem remains open. Why?

Landau’s question underscores the limits to our understanding of sparse sets and the scarcity of general methods for detecting primes within them.
Nevertheless, a number of breakthroughs show that sparsity can be compatible with rich prime patterns, provided one can exploit additional structure.
For instance, Bourgain \cite{Bourgain(2015)} and Swaenepoel \cite{Swaenepoel(2020)} obtained asymptotic formulas for primes with a positive proportion of prescribed digits.
Friedlander-Iwaniec \cite{Friedlander-Iwaniec(1998)} proved that the sparse sequence $m^2+n^4$ contains infinitely many primes. Maynard \cite{Maynard(2019)} proved there are infinitely many primes whose decimal expansion has a prescribed missing digit, say $7$.
Their arguments suggest that the right notion is not merely ``thinness'', but the interplay between sparsity and structure. This theme will appear later in the present work.

When a problem about primes in a sparse set is difficult, our understanding can be furthered by considering simplifying approximations to the primes, such as almost-primes (with a bounded number of prime factors) or square-free numbers (not divisible by squares larger than $1$).
Even in this ``easier'' setting, genuinely new ideas are often required.
Such ideas can accumulate into a toolkit that, in favourable situations, brings the original prime-problem within reach. We point out two of the several interesting open problems on square-free numbers in sparse sets:

\begin{enumerate}[label=(\Roman*)]
	\item Is there an expected amount of $L$-bit square-free integers with $50\%$ of their base-$2$ digits prescribed arbitrarily - both in their values and in positions (indices; except for obvious exceptions)? Dietmann-Elsholtz-Shparlinski \cite{Dietmann-Elsholtz-Shparlinski} showed this holds with $ < 40\%$ in place of $50\%$.
	
	\item Let $c>1$ be non-integral. Does the Piatetski--Shapiro sequence $(\bigl\lfloor n^c \bigr\rfloor)$ contain square-free values with the expected order of magnitude (along $n$) for all $1<c<2$? Baker {\em et al.} \cite{Baker_et_al.(2013)} established the result for $1<c<149/87$.
\end{enumerate}

Another major development of the last few decades has occurred in digital number theory, the study of arithmetic sequences defined by constraints on digits. Foundational questions in this direction go back to Gelfond \cite{Gelfond(1968)}.
Early progress combining digital structure with sieve methods includes work of Fouvry-Mauduit \cite{Fouvry-Mauduit(1996)}.
A breakthrough was achieved by Mauduit-Rivat \cite{Mauduit-Rivat(2010)}, resolving Gelfond's problem on the sum of digits of primes. Alongside the earlier works \cite{Bourgain(2015), Maynard(2019), Swaenepoel(2020)} already mentioned, see also the recent works of Drmota-M\"ullner-Spiegelhofer \cite{Drmota-Mullner-Spiegelhofer(2025)} on primes as sums of Fibonacci numbers, Drmota-Rivat \cite{Drmota-Rivat(2025)} on digital functions along squares of primes, and Swaenepoel \cite{Swaenepoel(2025)} on squares with prescribed digits.

Let $b > 1$ be an integer. Our main objects of interest are $b$-palindromes, namely natural numbers whose $b$-adic expansion reads the same forwards and backwards.
Let $\mathscr{P}_b(x)$ be the set of all $b$-palindromes up to $x \geq 1$.
Note $|\mathscr{P}_b(x)| \asymp_b \sqrt{x}$, so $\mathscr{P}_b(x)$ is as sparse as the set of squares up to $x$. 

Palindromes have been intertwined with human culture for over two millennia\footnote{Reversible wordplay and verses designed to be read in more than one direction were already common in antiquity; cf.\ Martial \cite{Martial(1993)} and the discussion of related traditions in Kwapisz \cite{Kwapisz(2019)}. For early inscriptional evidence, the ROTAS/SATOR word-square is attested at Pompeii (pre-62~CE); see Encyclopedia Britannica \cite{Britannica-SatorSquare}.}. In recent decades, they have attracted attention in mathematics
\cite{Banks(2015),Banks-Hart-Sakata(2004), Banks-Shparlinski, Banks-Shparlinski(2006), Carrillo-Santana(2025), Chourasiya-Johnston(2025), Cilleruelo-Luca-Baxter(2018), Cilleruelo-Luca-Shparlinski(2009), Col(2009), Irving(2014), Johnston-Kerr, Luca-Togbe(2008),Rajasekaran-Shallit-Smith(2018), T-Panario}. Although it remains unknown if there are infinitely many $b$-palindromic primes, there has been progress on both the additive and multiplicative fronts.

On the additive side, Banks \cite{Banks(2015)} initiated the study of representing integers as sums of palindromes, and Cilleruelo-Luca-Baxter \cite{Cilleruelo-Luca-Baxter(2018)} proved that every natural number is a sum of at most three $b$-palindromes when $b\geq 5$; see also Rajasekaran-Shallit-Smith \cite{Rajasekaran-Shallit-Smith(2018)} for the remaining small bases. On the multiplicative side, Col \cite{Col(2009)} proved there exists an integer $k_b$, depending only on $b$, such that infinitely many $b$-palindromes $n$ have at most $k_b$ prime factors. This was the first result of its kind. More recently, it was shown in \cite{T-Panario} that one may take $k_b=6$ for every base $b$, and that $b$-palindromes have level of distribution $1/5$, uniformly in $b$.

Johnston-Kerr \cite{Johnston-Kerr} recently derived asymptotics, under a standard coprimality constraint, for the number of square-free $b$-palindromes up to $x$, thus answering a question of Banks-Shparlinski \cite{Banks-Shparlinski}. Their argument adapts the setup of Cilleruelo-Luca-Shparlinski \cite{Cilleruelo-Luca-Shparlinski(2009)} and then applies van der Corput-type methods to control certain critical ranges of square-divisor sizes. For readers unfamiliar with the van der Corput method (including the $q$-van der Corput variant) see for instance Iwaniec-Kowalski \cite{Iwaniec-Kowalski}.

Closely related to palindromes is the reversal function, which reverses the base-$b$ digits of an integer; palindromes are precisely its fixed points. Recent work of Bhowmik-Suzuki \cite{Bhowmik-Suzuki} and Dartyge-Rivat-Swaenepoel \cite{Dartyge-Rivat-Swaenepoel(2025)} independently established Siegel-Walfisz-type results for reversals of primes, while Dartyge-Rivat-Swaenepoel \cite{Dartyge-Rivat-Swaenepoel(2025)} also proved a Bombieri-Vinogradov-type theorem and derived sieve-theoretic applications.


\subsection{Main results}\label{subsec:Main results}
We study quantitative square-divisibility questions for palindromes in arithmetic progressions to large moduli. Our first main result, Theorem~\ref{thm: square divisors}, provides upper bounds of this type. Combining it with inputs from \cite{T-Panario} and the large sieve of Baier-Zhao \cite{Baier-Zhao} for square moduli, we deduce the Bombieri-Vinogradov-type Theorem~\ref{thm:equidistribution of square-free palindromes} for square-free palindromes, with exponent of distribution $1/16$.

For an integer $L\geq 0$, let $\Pi_b(L)$ be the set of all $b$-palindromes in the interval $[b^L, b^{L+1})$. Clearly $|\Pi_b(L)| \asymp_b b^{L/2}$. For integers $q,a$ with $q \geq 1$, define the set
\begin{equation}\label{palindromes in an arithmetic progression}
\Pi_b(L,q,a) = \left\{\ell \in \Pi_b(L) \ : \ \ell \equiv a (q), \ (\ell,b)=1\right\}.
\end{equation}
By the bound of Banks-Shparlinski \cite{Banks-Shparlinski, Banks-Shparlinski(2006)}, 
\begin{equation}\label{pals in arith progr. Banks-Shparlinski}
\max_{a \in \mathbb{Z}}|\Pi_b(L,q,a)| \ll_b \dfrac{|\Pi_b(L)|}{\sqrt{q}} + 1.
\end{equation}
To the best of our knowledge, this remains the strongest general bound available in this setting.

In this article we use the notation $n \sim N$ to signify $N/2 < n \leq N$.

\begin{thm}\label{thm: square divisors}
	Let $L,q\geq 1$ be integers with $(q,b)=1$ and let $N\geq 1$. Then
	\begin{align}
	&\max_{(a,q)=1}\sum_{n\sim N}\sum_{\substack{\ell \in \Pi_b(L,q,a) \\ n^2 \mid \ell  }} 1  \nonumber\\ 
	&\ll_{b,\epsilon} 
	\dfrac{|\Pi_b(L)|^{1 + \epsilon}}{\sqrt{q}}\left(\left(\dfrac{q^{1/2}}{N}\right)^{1/5} + \left(\dfrac{q^{1/3}}{N}\right)^{3/16}  +  \dfrac{N^{3/5}q^{7/10} + (Nq)^{5/8} + q}{|\Pi_b(L)|}  \right) \label{square divisors of Pi(L,q,a)}
	\end{align}
for any $\epsilon > 0$. In particular,
\begin{equation}\label{square divisors of Pi}
\sum_{n\sim N}\sum_{\substack{\ell \in \Pi_b(L) \\ (\ell,b)=1\\ n^2\mid \ell}}1 \ll_{b,\epsilon} \dfrac{|\Pi_b(L)|^{1 + \epsilon}}{N^{3/16}}
\end{equation}
for any $\epsilon > 0$.
\end{thm}


 Theorem \ref{thm: square divisors} follows from Propositions \ref{prop: main term for smooth S}, \ref{prop: N large for squares n^2 with n of size N}, \ref{prop:first estimate for E_>(r)}, relying heavily on exponential-sum estimates. Some of our methodology is detailed in the upcoming subsections. 
 
 The case when $q=1$ above and $Q=1$ below was considered recently by Johnston-Kerr \cite{Johnston-Kerr} who showed that, for any $N\geq 1$, the left hand side of (\ref{square divisors of Pi}) is $\ll_{b,\epsilon}|\Pi_b(L)|^{1 + \epsilon}/N^{\epsilon_0}$ for some $\epsilon_0 > 0$ (implicit from their work after modifications to Proposition 10.1 in \cite{T-Panario}).

In the following, $\mathscr{P}_b^*(x)$ denotes the set of all $b$-palindromes up to $x$ and coprime to $b^3-b$. By Lemma 9.1 in \cite{T-Panario},  $|\mathscr{P}_b^*(x)| \asymp_b \sqrt{x}$. The next theorem shows that square-free palindromes are equidistributed in arithmetic progressions on average over moduli up to level $x^{1/16-\epsilon}$, with the expected main term. Here $6/\pi^2$ is the natural density of square-free integers, while the factors $\mathfrak{S}(m_b)$ and $\mathfrak{S}(q)$ account for the local corrections forced by the base and the modulus.

\enlargethispage{\baselineskip}
\begin{thm}\label{thm:equidistribution of square-free palindromes}
For a natural number $n$, define
\begin{equation}\label{def of G_b}
\mathfrak{S}(n) = \prod_{p \mid n}\left(1 - \dfrac{1}{p^2}\right)^{-1}
\end{equation}
with the convention $\mathfrak{S}(1)=1$. Set $m_b = b^3-b$.

For any $x \geq 1$, $\epsilon > 0$ and $1 \leq Q \leq x^{1/16-\epsilon}$, 
\begin{equation}\label{equidistribution for square-free pals}
\sum_{\substack{q \leq Q \\ (q,m_b)=1}}\sup_{\substack{(a,q)=1 \\ y\leq x}}\left|\sum_{\substack{n\in \mathscr{P}_b^*(y) \\ n\equiv a (q)}}\mu^2(n) - \dfrac{6\mathfrak{S}\left(m_b\right)\mathfrak{S}\left(q\right)|\mathscr{P}_b^*(y)|}{\pi^2q}\right| 
\ll_{b,\epsilon} \sqrt{x}\exp\left(-\sigma \sqrt{\log x}\right)
\end{equation}
for some $\sigma > 0$ depending only on $b,\epsilon$. 
\end{thm}

Theorem \ref{thm:equidistribution of square-free palindromes} follows from Theorem \ref{thm: square divisors} and Proposition \ref{prop:equidistribution involving squares}. The latter relies on the machinery for $6$-almost-primes in \cite{T-Panario} as well as on the estimate of Baier-Zhao \cite{Baier-Zhao} (see also Baker \cite{Baker(2017)}) for the large sieve with square moduli.

Theorem \ref{thm: square divisors} and Proposition \ref{prop:equidistribution involving squares}, although powered by rather different tools and arguments, converge to the exponent $1/16$ of Theorem \ref{thm:equidistribution of square-free palindromes} as follows. 
First it can be shown that Theorem \ref{thm:equidistribution of square-free palindromes} holds if we have
\begin{align}\label{E(Q,D) in intro}
\sum_{\substack{q \sim Q \\ (q,m_b)=1}}\sup_{\substack{(a,q)=1 \\ y \leq x}}\sum_{\substack{d \sim D \\ (d,m_b q)=1}}\left|\sum_{n \in \mathscr{P}_b^*(y)}\left(\mathbf{1}_{d^2 \mid n\equiv a (q)} - \dfrac{1}{qd^2}\right)\right|  \ll_{b,\epsilon} \sqrt{x}\exp\left(-\sigma' \sqrt{\log x}\right)
\end{align}
for any $Q \leq x^{1/16-\epsilon}$ and $D \leq \sqrt{x}$. This follows from $\mu^2(n) = \sum_{d^2 \mid n}\mu(d)$, an Euler product expansion of 
$\sum_{\substack{(d,m_b q)=1}}\mu(d)/d^2$ 
and dyadic decompositions.

\begin{itemize}
	\item Theorem \ref{thm: square divisors} implies that if $DQ \gg x^{1/4-\epsilon}$, then (\ref{E(Q,D) in intro}) holds if $Q \ll x^{1/16-\epsilon}$.
	
	\item Proposition \ref{prop:equidistribution involving squares} implies that if $DQ \ll x^{1/4-\epsilon}$, then (\ref{E(Q,D) in intro}) holds if $Q \ll x^{1/15-\epsilon}$.
\end{itemize}

\subsection{Structural ingredients}\label{subsec:two structural ingredients}

The bounds in Theorem \ref{thm: square divisors} rely in part on structural properties of palindromes. 

Rather than work with palindromes directly, it is advantageous to consider instead quasi-palindromes at some high level $\lambda$. We say that a natural number $n$ is a {\em $b$-quasi-palindrome at level $\lambda \in \mathbb{N}$} if
$$
d_j(n)=d_{\lfloor \log_b n\rfloor-j}(n)
\qquad \forall \ \ 0\leq j<\lambda,
$$
where $d_j(n)$ denotes the $j$-th digit of $n$ in its $b$-adic expansion $n=\sum_j d_j(n)b^j$. Clearly, palindromes are also quasi-palindromes. The point is that in an interval $[b^L,b^{L+1})$, quasi-palindromes form a structured union of short arithmetic progressions (SAPs). This is used to reduce square-divisibility questions, for palindromes in an arithmetic progression AP, to analogous questions for numbers in SAPs $\cap $ AP. There is a loss in the expected main term when we cover the set of palindromes by that of quasi-palindromes at level $\lambda$. Nevertheless, the loss is acceptable provided $\lambda$ is not too small. We clarify next what is meant by all this. 

In what follows, we use the notation $\mq,\ma$ in place of the modulus $q$ and residue class $a$ from Theorem \ref{thm: square divisors}; the latter $q$ will be used instead to denote an integral power $b^\lambda$ arising from the quasi-palindromic majorant below. By the assumptions in Theorem \ref{thm: square divisors}, $(\mq,\ma b)=1$. We also set $X = b^L$. 

Let $ 1 <  \lambda < L/2$ be a large integer to be chosen later. Recalling the definition of $\Pi_b(L,\mq,\ma)$ in (\ref{palindromes in an arithmetic progression}) and removing palindromic digit-constraints around the middle,
$$
\sum_{n \sim N} \sum_{\substack{\ell \in \Pi_b(L,\mq,\ma) \\ n^2 \mid \ell}}1 
\leq 
\sum_{\substack{n \sim N \\ (n,b\mq)=1}} \sum_{\substack{X \leq \ell < bX \\ d_j(\ell) = d_{L-j}(\ell) \  \forall  \  0 \leq j < \lambda\\ (\ell,b)=1\\n^2 \mid \ell \\ \ell \equiv \ma (\mq) }}1.
$$
Setting $q = b^\lambda$ and using the $b$-adic representation of numbers, we can rewrite the above as 
\begin{equation}\label{decomposition_involving_quasi-palindromes}
\sum_{n \sim N} \sum_{\substack{\ell \in \Pi_b(L,\mq,\ma) \\ n^2 \mid \ell}}1 
\leq
\sum_{\substack{n \sim N \\ (n,q\mq)=1}} \sum_{a \in \mathcal{A}_\lambda} \sum_{\substack{0 \leq m < bX/q^2 \\ a + qm \equiv 0 (n^2) \\ a + qm \equiv \ma (\mq)}}1,
\end{equation}
where
\begin{equation}\label{A_lambda_again}
\mathcal{A}_\lambda = \left\{\sum_{0 \leq j < \lambda}\left(b^j + b^{L-j}\right)k_j \ : \ 0 \leq k_0, \ldots, k_{\lambda - 1} < b, \ (k_0,b)=1\right\}.
\end{equation}
In this manner, we turn the square-divisibility problem concerning palindromes in an AP, to an analogous one involving structured SAPs.

After a pass to harmonic analysis via Poisson summation and an appeal to B\'ezout's identity, we encounter (similarly as in Johnston-Kerr \cite{Johnston-Kerr}, but with a different entry-point) exponential sums comprised of both arithmetic harmonics (involving modular inverses) and archimedean harmonics (with monomial-type phases). 
An important feature of the argument is that palindromic structures yield an effective $A$-type process for $2$-dimensional sums involving the reversal function, lowering arithmetic and archimedean conductors simultaneously. We detail this in the following Subsection \ref{subsec:A-type_process}. This quickly brings the conductors into a range where the $B$-process becomes effective. Particularly, this allows us to treat most regimes (of square-divisor sizes) in unison and, more importantly, permits a derivation of the bounds in Theorem \ref{thm: square divisors}. Although there may be other significant uses for the reversal sums, these will be employed primarily as a conductor-lowering machine.

\subsection{A-type process}\label{subsec:A-type_process}

We explain a concept that we have found helpful in deriving bounds. It concerns a kind of ``two-fold'' reduction in the total\footnote{We use the term {\em total conductor} to signify the product $q(t + 1)$, where $q$ is the modulus (or arithmetic conductor) in the arithmetic harmonic, and $t$ is the oscillating parameter (or archimedean conductor) of the archimedean harmonic.} conductor of the harmonics, whereby both arithmetic and archimedean conductors are reduced simultaneously and by the same amount (yielding a quadratic drop in the total conductor). In effect, it has features of both the classical $A$-process and the $q$-van der Corput method. This is due to properties of the reversal function.

To explain it, let us first consider the reversal function 
$$\rho : [0, b^{L+1}) \cap \mathbb{Z} \to [0, b^{L+1}) \cap \mathbb{Z}$$ defined by 
$$
\rho(n) = \sum_{0\leq \ell \leq L}d_\ell(n)b^{L-\ell},
$$
where $d_\ell(n)$ denotes the $\ell$-th digit in the $b$-adic expansion of $n = \sum_{0 \leq \ell \leq L}d_\ell(n)b^\ell$. Clearly if $d_0(n) \neq 0$, then $\rho(n) \asymp b^L$. Note that if there's no carry in the $b$-adic addition of two integers $m,n \in [0, b^{L+1})$, then $\rho(m + n) = \rho(m) + \rho(n)$. Moreover if $b^\ell\mid n$ for some integer $ 0\leq \ell \leq L$, then $\rho(n) < b^{L+1-\ell}$. For convenience, set $X = b^L$. These few basic properties of $\rho$ have a useful implication for sums of the following form
$$
\dfrac{1}{q}\sum_{\substack{1 \leq m \leq q \\ (m,q)=1}}\alpha(m) \sum_{\substack{n \sim N \\ (n,q)=1}}e_q(h m\overline{n})e\left(q\dfrac{\rho(m)}{X}\dfrac{N}{n}\right)
$$
with $q \leq X$ a large integral power of $b$, $(h,q)=1$ and $1$-bounded complex numbers $\alpha(m)$. We have chosen this sum for the purposes of illustration, but other types in higher dimensions and with more complicated phases, different conductors, weights, or other objects such as Kloosterman-type sums (see Section \ref{sec:BAB-process}) replacing $e_q(hm\overline{n})$, etc., are possible.

Note the total conductor of the harmonics is $q^2$ (arithmetic $\times$ archimedean) and so a $B$-process on the $n$-sum should yield essentially a bound of size $q$ (square root of the total conductor). This is only adequate if $q$ is sufficiently small relative to $N$. On the other hand, if it is too large, we can proceed as follows. Let $b \leq s \mid q$ be an integral power of $b$. We may express each $1 \leq m \leq q$, $(m,q)=1$, uniquely as $m = c + qd/s$ for some $0 \leq c < q/s$, $(b,c)=1$, and $0 \leq d < s$. Using the additivity of $\rho$ (in the absence of carry) moving the $d$-sum to the inside, taking absolute values and bounding the $c$-sum, the above is
$$
\leq \max_{c}\sum_{\substack{ n \sim N \\ (n,s)=1}}\left|\dfrac{1}{s}\sum_{0\leq d < s}\alpha\left(c + \dfrac{qd}{s}\right)e_s(hd\overline{n}) e\left(q\dfrac{\rho(qd/s)}{X}\dfrac{N}{n}\right)\right|
$$
in absolute value. Since each $qd/s$ is divisible by $q/s$, a power of $b$, then $\rho(qd/s) \ll sX/q$. Hence the archimedean conductor (for the phase on the far-right) is $\ll s$. Taking also into account the modulus $s$ of the arithmetic component, this gives a total conductor of size $s^2$ (at worst). Thus with the input of a sum of size $s$ (which determines the gain from the diagonal portion after a subsequent Cauchy-Schwarz) the original total conductor, $q^2$, is reduced by a factor of $q^2/s^2$. We make use of this type of $A$-process in both Sections \ref{sec:BAB-process} and $\ref{sec:N small relative to q}$. It would be interesting to investigate 
$A,B$-sequences involving this $A$-type process, for broader classes of harmonic sums involving the reversal function. We do not pursue this here.

\enlargethispage{\baselineskip}
\section{Paper structure}\label{subsec:Organization}

We develop preparatory lemmas in Section \ref{sec:Preparatory_lemmas} and translate the counting problem into exponential-sum estimates in Proposition \ref{prop: main term for smooth S}. Here we pass from palindromes to quasi-palindromes, isolate the expected main term and treat the case when the archimedean conductors are small. The sums involved in the latter case are purely arithmetic and handled via completion together with correlation estimates for quadratic Kloosterman $\K_2$ sums. In Section \ref{sec:BAB-process} we study the range in which 
$N$ is large relative to the conductors, using a 
BAB-type argument; here
A denotes the
A-type process (discussed earlier) adapted for $\K_2$ sums in place of $e_q(hm\overline{n})$, etc.. In Section \ref{sec:N small relative to q} we treat the complementary range, where the conductors are large relative to 
$N$. Here conductors are lowered at the outset via the A-type process. Theorem \ref{thm: square divisors} is proved in Section \ref{sec:Proof of Theorem {thm:square divisors}} by combining the estimates from Propositions \ref{prop: main term for smooth S}, \ref{prop: N large for squares n^2 with n of size N}, \ref{prop:first estimate for E_>(r)}. In Section \ref{sec:medium square divisors} we prove Proposition \ref{prop:equidistribution involving squares}, which handles the small and medium ranges of square divisors by combining inputs from \cite{T-Panario} with the large sieve for square moduli of Baier-Zhao \cite{Baier-Zhao}. Theorem \ref{thm:equidistribution of square-free palindromes} is then proved in Section \ref{sec:Proof of Theorem equidistribution of square-free palindromes} by combining Theorem \ref{thm: square divisors} and Proposition \ref{prop:equidistribution involving squares}. We include in Appendix \ref{sec: large squares} an ``elementary''-type argument for very large squares, not used in the proofs of the main results.

\section{Preparatory lemmas}\label{sec:Preparatory_lemmas}

We gather several tools needed in the following sections. These are either well-known in the literature or follow directly from known facts.

\begin{lem}[Poisson summation]\label{lem:poisson summation}
Let $\eta: \mathbb{R} \to \mathbb{C}$ be a smooth compactly supported function such that, for some $\delta \geq 0$, $\|\eta^{(j)}\|_{\infty} \ll_{\delta,j} X^{\delta j}$ for each $j\geq 0$. Then for any $N \gg 1$ and integers $a,q$ with $q \geq 1$ and $N,q \leq X^{O(1)}$, 
$$
\sum_{n \equiv a (q)}\eta(n/N) = \dfrac{N\widehat{\eta}(0)}{q} + \dfrac{N}{q}\sum_{1 \leq |h| \leq H}\widehat{\eta}\left(\dfrac{Nh}{q}\right)e_q(ah) + O_{A,\epsilon, \delta}\left(X^{-A}\right)
$$
for any $A,\epsilon > 0$ and any $H \geq  qX^{\epsilon + \delta}/N$. 
\end{lem}

\begin{proof}
Writing the sum on the left as
$$
\sum_{m}\eta\left(\dfrac{a + mq}{N}\right)
$$
and applying the usual Poisson summation formula to this sum, one derives	
$$
\sum_{n \equiv a (q)}\eta(n/N) = \dfrac{N}{q}\sum_{h}\widehat{\eta}\left(\dfrac{Nh}{q}\right)e_q(ah).
$$
The claim now follows after isolating the term with $h=0$ and noticing that $\widehat{\eta}(Nh/q) \ll_{\delta,A} X^{\delta A}|Nh/q|^{-A}$ for any $|h| > H$ and any $A > 0$. The last follows from several integrations by parts on the defining integral of $\widehat{\eta}$.   
\end{proof}	

\begin{lem}[Coprime sums]\label{lem:coprime sums}
With the same assumptions of Lemma \ref{lem:poisson summation},
$$
\sum_{(n,q)=1}\eta(n/N) = \dfrac{\varphi(q)\widehat{\eta}(0) N}{q} + O_\delta\left(X^{2\delta}\tau(q)\right). 
$$
\end{lem}

\begin{proof}
By the M\"{o}bius inversion formula $\mathbf{1}_{(n,q)=1} = \sum_{d\mid (n,q)}\mu(d)$, the sum on the left equals
$$
\sum_{d\mid q}\mu(d)\sum_{n\equiv 0 (d)} \eta(n/N).
$$
The result now follows from Lemma \ref{lem:poisson summation}, the fact $\sum_{d\mid q}\mu(d)/d = \varphi(q)/q$ and the bound $\widehat{\eta}(t) \ll_\delta 1$. 	
\end{proof}	

\begin{lem}[B\'{e}zout's identity]\label{lem:Bezout}
For any coprime integers $m,n \geq 1$,
$$
\dfrac{1}{mn} \equiv \dfrac{\overline{m}}{n} + \dfrac{\overline{n}}{m} \pmod{1}.
$$
\end{lem}

\begin{proof}
This is a consequence of
the Chinese remainder theorem.	
\end{proof}	

\begin{lem}[Sums of GCDs]\label{lem:sums of GCDs}
	Let $q \geq 1$ be an integer and let $N\geq 1$. Then
	$$
	\sum_{1 \leq n \leq N}(n,q) \leq N \tau(q). 
	$$	
\end{lem}

\begin{proof}
	Follows from
	$
	(n,q) \leq \sum_{d \mid (q,n)}d 
	$ and a switch in the order of summation. 
\end{proof}

\begin{lem}[Properties of $\rho$]\label{lem:Reversal-type function}
	For an integer $L \geq 0$, define the function
	$$\rho : [0, b^{L+1}) \cap \mathbb{Z} \to [0, b^{L+1}) \cap \mathbb{Z}$$
	by
	\begin{equation}\label{def of rho in lemma}
	\rho(n) = \sum_{0 \leq \ell \leq L}d_\ell(n)b^{L-\ell},
	\end{equation}
	where $d_\ell(n)$ denotes the $\ell$-th digit of $n$ in its $b$-adic expansion
	$$
	n = \sum_{0 \leq \ell \leq L} d_\ell(n) b^\ell.
	$$
	Then $\rho$ satisfies the following properties:
	
	(A) $\rho$ is bijective.
	
	(B) For any integers $n \in [0, b^{L+1})$ and $0 \leq \ell \leq L$, if $n \leq b^\ell$, then
	$$
	b^{L-\ell} \mid \rho(n).
	$$
	
	(C) For any integers $n \in [0, b^{L+1})$ and $0 \leq \ell \leq L$, if $b^\ell \mid n$, then
	$$
	\rho(n) < b^{L+1-\ell}.
	$$
	
	(D) For any integers $m,n \in [0, b^{L+1})$ with no carry in their $b$-adic addition,
	$$
	\rho(m + n) = \rho(m) + \rho(n).
	$$
	
	(E) If $0 \leq n < b^{L+1}$ and $b \nmid n$, then $\rho(n) \geq b^{L}$. 
\end{lem}

\begin{proof}
	Follows directly from the definition of $\rho$ and the uniqueness of the $b$-adic representation of integers in the interval $[0, b^{L+1})$. 	
\end{proof}

In what follows we define, for integers $c,d,q$ with $q \geq 1$, 
\begin{align}\label{def of Gauss type sum}
\G^*(c;q) &= \dfrac{1}{\sqrt{q}}\sum_{n(q)}^*e_q\left(cn^2\right),\\
\K_2(c,d;q) &= \dfrac{1}{\sqrt{q}}\sum_{n(q)}^*e_q\left(c\overline{n}^2 + dn\right)\label{def of Kloosterman-type S}.
\end{align}

\begin{lem}[Multiplicative property]\label{lem:multiplicative property}
For any integers $c,d,q,r$ with $q,r \geq 1$ and $(q,r) = 1$,
$$
\K_2(c,d;qr) = \K_2(c\overline{r},d\overline{r};q)\K_2(c\overline{q},d\overline{q};r).
$$	
\end{lem}

\begin{proof}
	Follows from the Chinese remainder theorem.
\end{proof}

In the following, $\odd(n)$ denotes the odd part of an integer $n$; that is, if $n = 2^\ell r$ with $r$ odd, then $\odd(n) = r$.

\begin{lem}[Coprime Gauss sums]\label{lem: Gauss sum variant}
	If $q = 2^\ell r$ is an integer with $r$ odd and $(a,q)=1$, then
	\begin{equation}\label{eqn:gauss sum bound}
	\G^*(a;q) \ll \mathbf{1}_{\ell \leq 3} \mu(r) \tau(q).
	\end{equation}
	For arbitrary $a$,
	$$
	\G^*(a;q) \ll \mu\left(\odd\left(\dfrac{q}{(a,q)}\right)\right)\mathbf{1}_{16 \nmid q/(a,q)} \tau(q)\sqrt{(a,q)}.
	$$
\end{lem}	

\begin{proof}
	For the first inequality, see for instance Lemma 3 of Heath-Brown and Tolev \cite{Heath-Brown-Tolev} and the definition in Equation (7) there. The second inequality follows from the first.	
\end{proof}	

In the following, we use the notation $e_q(n^{-1}) := e_q(\overline{n})$. Thus for any positive integer $k$, $e_q(n^{-k})  := e_q(\overline{n}^k)$. 

\begin{lem}[Sums with binomial Laurent phases]\label{lem: Shparlinski bound}
	For any (possibly negative)	integers $c,d,k,\ell,q$ with $q \geq 1$, $k\neq \ell$ and $k,\ell \neq 0$,
	$$
	\sum_{\substack{n(q)}}^*e_q\left(cn^k + d n^\ell\right)\ll_{k,\ell, \epsilon} q^\epsilon \sqrt{q(c,d,q)}
	$$
	for any $\epsilon > 0$. In particular,
	$$
	\K_2(c,d;q) \ll_\epsilon q^\epsilon \sqrt{(c,d,q)}
	$$
	for any $\epsilon > 0$. 
\end{lem}	

\begin{proof}
	See Theorem 1 of Shparlinski \cite{Shparlinski}. 
\end{proof}

\begin{lem}[Completion of sums with squared inverses]\label{lem:completion of squared inverses}
With the same assumptions of \textnormal{Lemma \ref{lem:poisson summation}},  
\begin{align*}
&\sum_{(n,q)=1 }\eta(n/N)e_q\left(a\overline{n}^2\right)\\
 &= \dfrac{N\mathcal{\G^*}(a;q)\widehat{\eta}(0)}{\sqrt{q}} + \dfrac{N}{\sqrt{q}}\sum_{1 \leq |h| \leq H}\widehat{\eta}\left(\dfrac{Nh}{q}\right)  \K_2(a,h;q)
+O_{A,\epsilon,\delta}\left(X^{-A}\right)\\
&\ll 
 N\tau(q)\mu^2\left(\odd\left(\dfrac{q}{(a,q)}\right)\right)\mathbf{1}_{16 \nmid q/(a,q)} \sqrt{\dfrac{(a,q)}{q}} + \sqrt{q}X^{2\epsilon + \delta}
\end{align*}
for any $A,\epsilon > 0$ and any $H \geq qX^{\epsilon + \delta}/N$.
\end{lem}

\begin{proof}
	Follows directly from Lemmas \ref{lem:poisson summation}, \ref{lem: Gauss sum variant}, \ref{lem: Shparlinski bound}, \ref{lem:sums of GCDs}.
\end{proof}

\begin{lem}[Twisted Poisson summation]\label{lem:twisted Poisson summation}
Let $\eta: \mathbb{R} \to \mathbb{C}$ be a smooth compactly supported function and let $\phi: \mathbb{R} \to \mathbb{R}$ be smooth. Then for any reals $N,T \neq 0$ and integers $a,q$ with $q \geq 1$,
$$
\sum_{n\equiv a (q)}\eta(n/N)e\left(T\phi(n/N)\right) = \dfrac{N}{q\sqrt{T}}\sum_{h} e_q(ah) \mathfrak{J}_{\eta,\phi}\left(\dfrac{Nh}{q}; T\right),
$$
where
$$
\mathfrak{J}_{\eta,\phi}(\xi; T) = \sqrt{T}\int_{\mathbb{R}}\eta(x) e\left(T\phi(x) - \xi x\right)dx.
$$	
\end{lem}

\begin{proof}
Letting $f(x) = \eta(x) e(T\phi(x))$, one observes that the sum on the left equals
$$
\sum_{n\equiv a (q)} f(n/N) = \sum_{m}f\left(\dfrac{a + qm}{N}\right) 
$$
and applies the usual Poisson summation formula to the $m$-sum.	
\end{proof}	

\begin{lem}[Oscillatory integrals]\label{lem:Huxley stat phase integrals}
Let $c < d$ be real numbers. Suppose $f,g : \mathbb{R} \to \mathbb{R}$ are smooth functions  such that, for some $\Theta_f, \Omega_f, \Omega_g$ with $ \Omega_f \gg d-c$, 
$$
f^{(j)}(x) \ll \dfrac{\Theta_f}{\Omega_f^j}, \hspace{2em} g^{(k)}(x) \ll \dfrac{1}{\Omega_g^k}
$$
for $j=2,3$ and $k = 0,1,2$. Assume $g(c) = g(d)=0$ and let
$$
\mathfrak{I} = \int_{c}^d g(x) e(f(x))dx.
$$
Then we have the following two facts.

(A) Suppose $f',f''$ do not vanish on $[c,d]$. Let $\Lambda = \min_{[c,d]}|f'(x)|$. Then
$$
\mathfrak{J} \ll \dfrac{\Theta_f}{\Omega_f^2 \Lambda^3}\left(1 + \dfrac{\Omega_f}{\Omega_g} + \dfrac{\Omega_f^2}{\Omega_g^2}\dfrac{\Lambda}{\Theta_f/\Omega_f}\right).
$$

(B) Suppose $f$ changes sign from negative to positive at the unique point $x_0 \in (c,d)$. Let $\kappa = \min(d-x_0, x_0 - c)$. Further assume $f''(x) \gg \Theta_f/\Omega_f^2$ and $f^{(4)}(x) \ll \Theta_f/\Omega_f^4$. Then
$$
\mathfrak{I} = \dfrac{g(x_0)e(f(x_0) + 1/8)}{\sqrt{f''(x_0)}} + O\left(\dfrac{\Omega_f^4}{\Theta_f^2 \kappa^3} + \dfrac{\Omega_f}{\Theta_f^{3/2}} + \dfrac{\Omega_f^3}{\Theta_f^{3/2} \Omega_g^2}\right).
$$
\end{lem}

\begin{proof}
	This is a special case of Huxley \cite{Huxley} phrased similarly as in Munshi \cite{Munshi}.
\end{proof}	

\begin{lem}[Integrals with squared inverses]\label{lem:integrals with squared inverses}
Let $\eta: \mathbb{R} \to \mathbb{R}$ be a smooth function compactly supported on some interval $[c,d]$ with $c > 0$ and $1 \ll  c < d \ll 1$ such that, for some fixed $\delta > 0$,  $\|\eta^{(j)}\|_\infty \ll_j X^{\delta j}$ for each $j\geq 0$. Let $T \gg 1$ positive and $h\in \mathbb{R}$. Then 
\begin{align}\label{stat for squared inverses}
\int_{\mathbb{R}}\eta(x)e\left(Tx^{-2} + hx\right)dx = 
\dfrac{\widetilde{\eta}(h/T) }{\sqrt{T}}  e\left(\dfrac{3}{2^{2/3}}T^{1/3}h^{2/3}\right)
+ O\left(\dfrac{X^{2\delta}}{\max\left(T^{3/2}, |h|^{3/2}\right)}\right) ,
\end{align}
where
\begin{equation}\label{def of tilde{eta}}
\widetilde{\eta}(t) =  \dfrac{e(1/8)2^{2/3}\mathbf{1}_{t > 0}t^{-2/3}\eta\left(2^{1/3}t^{-1/3}\right)}{\sqrt{6}}.  
\end{equation}
\end{lem}

\begin{proof}
Write $Tx^{-2} + hx = T\phi(x)$, where $\phi(x) = x^{-2} + h'x$ with $h' = h/T$. Note
$$
\phi'(x) = -2x^{-3} + h'
$$
and
$$
\phi^{(j)}(x) = (-1)^j (j+1)! x^{-j - 2}
$$
for $j\geq 2$. Clearly, for $x$ in the support of $\eta$, we have $\phi^{(j)}(x) \asymp_j 1$ for each $j\geq 2$. Note that if $h' \geq K$ or $h' \leq k$ for some sufficiently large $0 < K \ll 1$ and small $0 < k \ll 1$, respectively (with each depending only on $c,d \asymp 1$) then, for $x$ in the support of $\eta$, we have $\phi'(x) \gg \max(1, |h'|)$. We may apply Lemma \ref{lem:Huxley stat phase integrals} (A) with $f(x) = T\phi(x)$, $\Theta_f = T$, $\Omega_f = 1$, $\Lambda = T\max(1, |h'|) = \max(T,|h|)$, $\Omega_g = X^{-\delta}$. This gives (for this case of $h' = h/T$)
$$
\int_{\mathbb{R}}\eta(x)e(Tx^{-2} + hx)dx \ll \dfrac{X^{2\delta}}{\max(T^2, h^2)}.
$$
If we further impose that $k$ and $K$ are sufficiently small and large, respectively, so that also $\widetilde{\eta}(h/T = h') = 0$, it then follows that (\ref{stat for squared inverses}) holds in our case of $h'$. 

Consider now the case when $k < h' < K$ with $1 \ll k < K \ll 1$ positive as above. In this case, $h \asymp T$. Note $\phi$ is stationary (i.e., $\phi'(x)=0$) at the unique point 
$$
x_0 = \left(\dfrac{2}{h'}\right)^{1/3} \asymp 1. 
$$
Since $\eta$ is compactly supported on $[c,d]$, we may find $u,v$ positive with $1 \ll u < c < d < v \ll 1$ and $x_0 \in (u,v)$ such that $\kappa = \min(v -x_0, x_0-u) \gg 1$ and
$$
\int_{\mathbb{R}}\eta(x)e(T\phi(x))dx = \int_{u}^v\eta(x)e(T\phi(x))dx.
$$ 
We may now apply Lemma \ref{lem:Huxley stat phase integrals} (B) with $f=T\phi$ and $\Theta_f = T \asymp \max(T, |h|)$, $\Omega_f = 1$, $\Omega_g = X^{-\delta}$ as before. The result follows after computing the error term and $\phi(x_0), \phi''(x_0)$.
\end{proof}

\begin{lem}[Twisted sums with squared inverses]\label{lem:twisted sums with squared inverses}
Let $q,r \geq 1$, $(q,r)=1$ and $a$ be integers with $q,r \leq X^{O(1)}$. Then with the same assumptions of Lemma \ref{lem:integrals with squared inverses} and the extra assumptions $ X^{2\delta} \ll T \ll X^{O(1)}$ and $1 \ll N \ll X^{O(1)}$, 	
\begin{align*}
&\sum_{(n,qr)=1}\eta(n/N)e_q\left(a \overline{n}^2\right)e\left(T\dfrac{N^2}{n^2}\right)\\ 
&= \dfrac{N}{\sqrt{qT}} \sum_{d \mid r} \dfrac{\mu(d)}{d} \sum_{h} \widetilde{\eta}\left(\dfrac{Nh}{dqT}\right)\K_2(a, -\overline{d}h;q)e\left(3T \left(\dfrac{Nh}{2dq T}\right)^{2/3}\right)+ O\left(\sqrt{\dfrac{q}{T}} X^{3\delta}\right),
\end{align*}
where $\widetilde{\eta}$ is as defined in Lemma \ref{lem:integrals with squared inverses}. 
\end{lem}

\begin{proof}
	By the M\"{o}bius inversion formula, 
the sum on the left equals
\begin{align*}
&\sum_{d \mid r}\mu(d) \sum_{\substack{(n,q)=1 \\ n \equiv 0 (d)}} \eta(n/N)e_q\left(a\overline{n}^2\right)e\left(T\dfrac{N^2}{n^2}\right) \\
&=  \sum_{d \mid r}\mu(d) \sum_{\substack{(n,q)=1 }} \eta(dn/N)e_q\left(a\overline{dn}^2\right)e\left(T\dfrac{N^2}{d^2n^2}\right).
\end{align*}
Splitting the $n$-sum according to the class of $n$ modulo $q$, the above is
\begin{equation}\label{splitting the n-sum according to the class of n mod q}
\sum_{d \mid r}\mu(d) \sum_{m(q)}^* e_q\left(a \overline{dm}^2\right)\sum_{\substack{n \equiv m (q) }} \eta(dn/N)e\left(T\dfrac{N^2}{d^2n^2}\right). 
\end{equation}
Letting $\phi(x) = x^{-2}$, Lemma \ref{lem:twisted Poisson summation} shows that the $n$-sum above equals
$$
\dfrac{N}{dq \sqrt{T}}\sum_{h}e_{q}(-m h) \mathfrak{J}_{\eta,\phi}\left(-\dfrac{Nh}{dq}; T\right),
$$	
where $\mathfrak{J}_{\eta,\phi}$ is defined as in Lemma \ref{lem:twisted Poisson summation}. Inserting this in (\ref{splitting the n-sum according to the class of n mod q}), switching orders of summation and substituting $m$ with $\overline{d}m$, we find that (\ref{splitting the n-sum according to the class of n mod q}) equals
\begin{equation}\label{after Poisson}
\dfrac{N}{\sqrt{qT}}\sum_{d \mid r} \dfrac{\mu(d)}{d}  \sum_{h}\K_2(a, -\overline{d}h;q)\mathfrak{J}_{\eta,\phi}\left(-\dfrac{Nh}{dq}; T\right). 
\end{equation}
Since $\|\eta^{(j)}\|_\infty \ll_j X^{\delta j}$ for each $j\geq 0$ and $T \gg X^{2\delta}$ by assumption, several integrations by parts show 
$
\mathfrak{J}_{\eta,\phi}\left(0; T\right) \ll_A X^{-A}
$
for any $A > 0$. Thus (\ref{after Poisson}) equals 
\begin{equation}\label{nonzero frequencies}
\dfrac{N}{\sqrt{qT}}\sum_{d \mid r} \dfrac{\mu(d)}{d}  \sum_{h \neq 0}\K_2(a, -\overline{d}h;q)\mathfrak{J}_{\eta,\phi}\left(-\dfrac{Nh}{dq}; T\right) + O_A\left(X^{-A}\right)
\end{equation}
for any $A > 0$. 
By Lemma \ref{lem:integrals with squared inverses}, 
\begin{align}
\mathfrak{J}_{\eta,\phi}\left(-\dfrac{Nh}{dq}; T\right) &= \widetilde{\eta}\left(\dfrac{Nh}{dqT}\right) e\left(\dfrac{3}{2^{2/3}} \left(\dfrac{\sqrt{T}Nh}{dq}\right)^{2/3}\right)\nonumber\\
&\qquad +O\left(\dfrac{X^{2\delta}}{T\max\left(1, (N|h|/dqT)^{3/2}\right)}\right).\nonumber
\end{align}
By Lemma \ref{lem: Shparlinski bound}, the contribution of the error term to (\ref{nonzero frequencies}) is
$$
\ll_\epsilon \dfrac{NX^{\epsilon + 2\delta}}{T^{3/2}\sqrt{q}}\sum_{d \mid r} \dfrac{1}{d} \sum_{h\geq 1} \sqrt{(h,q)}\min\left(1, \left(\dfrac{dqT}{Nh}\right)^{3/2}\right). 
$$
An application of Lemma \ref{lem:sums of GCDs} and summation by parts show that the sum over $h \geq 1$ is 
$
\ll dq\tau(q)T/N. 
$
Hence the expression above is
$
\ll \sqrt{q}X^{3\delta}/\sqrt{T}
$
and so (\ref{nonzero frequencies}) yields the proposition.
\end{proof}

\begin{lem}[Quadratic Kloosterman sums with square moduli]\label{lem:Quadratic Kloosterman sums with square moduli} 
	Let $c,d$ and $q \geq 1$ be integers. Then
	$$
	\K_2\left(c,d;q^2\right) = \sum_{\substack{1 \leq \ell \leq q \\ (\ell,q)=1 \\ d\ell^3 \equiv 2c (q)}} e_{q^2}\left(c \overline{\ell}^2 + d\ell\right).
	$$	
\end{lem}

\begin{proof}
We follow Sali\'e's method (see for instance Cochrane-Zheng \cite{Cochrane-Zheng}). Recall that by definition, 
$$
\K_2(c,d;q^2) = \dfrac{1}{q} \sum_{\substack{n(q^2)}}^* e_{q^2}\left(f(n)\right),
$$
where $f(n) \equiv c\overline{n}^2 + dn \ (q^2)$.  Splitting the sum according to the class of $n$ modulo $q$, we have
$$
\K_2(c,d;q^2) =  \sum_{\substack{1 \leq \ell \leq q \\ (\ell,q)=1}} \dfrac{1}{q}\sum_{m(q)}e_{q^2}\left(f(\ell + qm)\right).
$$
For any integer $n$, note
$
\overline{1 + qn} \equiv 1 - qn \left(q^2\right).
$
This follows after multiplying both sides by $1 + qn$. Thus if $(\ell,q)=1$, 
$$
\overline{\ell + qm} \equiv \overline{\ell \left(1 + qm\overline{\ell}\right)} \equiv \overline{\ell}(1 - qm\overline{\ell}) \ (q^2). 
$$
Hence 
$$
\overline{\ell + qm}^2 \equiv \overline{\ell}^2 - 2qm\overline{\ell}^3 \ (q^2)
$$
and
$$
f(\ell + qm) \equiv c\overline{\ell}^2 - 2cqm\overline{\ell}^3 + d\ell + dqm \ (q^2).  
$$
It follows
\begin{align*}
\K_2(c,d;q^2) &= \sum_{\substack{1 \leq \ell \leq q \\ (\ell,q)=1}} e_{q^2}\left(c\overline{\ell}^2 + d\ell\right)\dfrac{1}{q}\sum_{m(q) }e_q\left(\left(d - 2c\overline{\ell}^3\right)m\right) \\
&=
\sum_{\substack{1 \leq \ell \leq q \\ (\ell,q)=1 \\ d\ell^3 \equiv 2c (q)}} e_{q^2}\left(c\overline{\ell}^2 + d\ell\right)
\end{align*}
by orthogonality.
\end{proof}	

\begin{lem}[Correlations of $\K_2$ sums]\label{lem:Correlations of K_2 sums}
Let $c,d,q$ be integers with $q \geq 1$. Then
$$
\left|\sum_{n(q)}\K_2(n,c;q)\overline{\K_2(n,d;q)}\right| \leq \left(c^2-d^2,q\right)\tau(q). 
$$	
\end{lem}

\begin{proof}
	Expanding the $\K_2$ sums, switching orders of summation and using orthogonality, the left hand side above equals
	$$
	\left|\sum_{\substack{\ell_1,\ell_2 (q) \\ (\ell_2\overline{\ell_1})^2  \equiv 1 (q)}}^*e_q\left(c\ell_1 - d\ell_2\right)\right| = \left|\sum_{\ell^2 \equiv 1 (q)}\cq_q(c - d\ell)\right|
	$$
	with $\cq_q$ Ramanujan's sum.
	The last follows after substituting $\ell_2$ with $\ell_2\ell_1$ and then summing over $\ell_1$.
	Using the multiplicativity of $\cq_q$ in the modulus and the Chinese remainder theorem, the above becomes
	$$
	\prod_{p^k \mid \mid q}\left|\sum_{\ell^2\equiv 1 (p^k)}\cq_{p^k}(c - d\ell)\right|.
	$$
	If $p$ is odd or $k=2$, the solutions to $\ell^2 \equiv 1 (p^k)$ are $\pm 1$, whence the sum is $\leq (c+d,p^k) + (c-d, p^k)$ by basic bounds for $\cq_{p^k}$. If $p=2$ and $k\geq 3$, the solutions are $\pm1, \pm (1 + 2^{k-1})$. Since $\cq_{2^k}(n-d2^{k-1}) = (-1)^d\cq_{2^k}(n)$ for any integer $n$, the sum is
	$
	\leq 2\left(c - d,2^k\right) + 2\left(c+d,2^k\right)
	$
	when $p=2$. 
	The result follows from $(c\pm d, p^k) \leq ((c \pm d)(c\mp d),p^k)$. 
\end{proof}	

\begin{lem}[Twisted incomplete sums of $\K_2$ sums]\label{lem:sums of K_2 sums}
Let $N \geq 1$ and let $q\geq 1$ be an integer. Then
$$
\sup_{\substack{\alpha \in \mathbb{R} \\ (a,q)=1\\c\in \mathbb{Z} }} \left|\sum_{n \leq N}e(\alpha n)\K_2(an,c;q)\right| \ll_\epsilon   q^{\epsilon}\min\left(N, \  \sqrt{q} +  \dfrac{N}{\sqrt{q}} \right)
$$	
for any $\epsilon > 0$. 
\end{lem}

\begin{proof}
The first inequality with the bound $\ll_\epsilon q^\epsilon N$ follows from Lemma \ref{lem: Shparlinski bound}. With regards to the second inequality, expanding $\K_2(an,c;q)$ according to its definition and switching orders of summation, the $n$-sum above equals
\begin{equation}\label{trig}
\dfrac{1}{\sqrt{q}}\sum_{m(q)}^*e_q(cm)\sum_{n\leq N}e\left(\dfrac{a\overline{m}^2n}{q} + \alpha n\right) = \dfrac{1}{\sqrt{q}}\sum_{\substack{\ell(q)}}^* \beta(\ell)\sum_{n\leq N}e\left(\dfrac{a\ell n}{q} + \alpha n\right),
\end{equation}	
where 
$$
\beta(\ell) = \sum_{\substack{m(q) \\ \overline{m}^2\equiv \ell (q)}}^*e_q(c m) \ll \tau(q).
$$
Now the result follows from 
$$
\left|\sum_{n\leq N}e\left(\dfrac{a\ell n}{q} + \alpha n\right)\right| \leq \min\left(N, \dfrac{1}{\|\frac{a\ell}{q} + \alpha\|}\right)
$$
and the $1/q$-spacing mod $1$ of the points $\alpha + a\ell/q$ as $\ell$ runs over units mod $q$.
\end{proof}

\begin{lem}[Large sieve involving squares]\label{lem:Large sieve with square moduli}
	
	Let $D,N \geq 1$ and let $q$ be a natural number.	Then
	\begin{equation}\label{Delta}
	N\sup_{\alpha \in \mathbb{R}} \sum_{d \leq D} \sum_{\substack{h(qd^2) \\ \|\frac{h}{qd^2} - \alpha\|\leq 1/N}}^* 1 \ll_\epsilon \Delta_\epsilon(D,N,q)
	\end{equation}
	for any $\epsilon > 0$, where
	\begin{equation}\label{Delta def}
	\Delta_\epsilon(D,N,q) = (DN)^\epsilon\left(1 + \dfrac{q}{N}\right)\left(qD^3 + N\sqrt{D}\right) .
	\end{equation}
	 Moreover for any sequence $(\gamma_n)$ of complex numbers,
	\begin{equation}\label{L^2 version}
	\sum_{d\leq D}\sum_{h(qd^2)}^*\left|\sum_{|n| \leq N}\gamma_n e_{qd^2}(hn)\right|^2 \ll_{\epsilon}\Delta_\epsilon(D,N,q)\sum_{|n|\leq N}|\gamma_n|^2.
	\end{equation}
\end{lem}

\begin{proof}
	These are a consequence of the result of Baier-Zhao \cite{Baier-Zhao} concerning the large sieve with square moduli. See Baker's \cite{Baker(2017)} refinement in Lemma 2 there.
\end{proof}



\section{Square divisors and quasi-palindromes}\label{sec:Square divisors and quasi-palindromes}

In this section we establish Proposition \ref{prop: main term for smooth S} below. We will use the notations $\mq,\ma$ in place of $q,a$ in Theorem \ref{thm: square divisors}; the last two will be reserved instead to symbolize other (different) quantities. 

Unless otherwise stated, we view 
$b > 1$ as a fixed integer and allow implied constants to depend on $b$. This is done in order to avoid several appearances of $b$ as a subscript in the implied constants. 

Recall we are interested in bounding sums of the form
$$
S = \sum_{n\sim N}\sum_{\substack{\ell \in \Pi_b(L) \\ (\ell,b)=1 \\ n^2 \mid \ell \\ \ell \equiv \ma (\mq) }}1
$$
for $L,N$ large and integers $\ma,\mq$ with $1 \leq \mq \leq b^{O(L)}$ and $(\ma b,\mq)=1$. In what follows we use the notation
\begin{equation}\label{X = b^L definition}
X = b^L.
\end{equation}
Note we may assume $X^\delta \ll N \ll \sqrt{X}$ for some small and fixed $\delta > 0$.

Let $\lambda > 0$ be a large integer to be chosen later satisfying
$$
b^\lambda \leq \sqrt{X}.
$$
For technical reasons we also restrict $\lambda$ to be even.
We have the bound
$$
S \leq \sum_{\substack{n\sim N \\ (n,b\mq)=1}}\sum_{\substack{X \leq \ell < bX \\ d_j(\ell) = d_{L-j}(\ell)  \ \forall \ 0 \leq j < \lambda \\  (\ell ,b)=1 \\ n^2 \mid \ell \\ \ell \equiv \ma (\mq)}}1,
$$
which we rewrite as
$$
S \leq \sum_{a \in \mathcal{A}_\lambda} \sum_{\substack{n \sim N \\(n,b\mq)=1 }}\sum_{\substack{0 \leq m < bX/b^{2\lambda} \\ a + b^\lambda m \equiv 0 (n^2) \\ a + b^\lambda m \equiv \ma (\mq)}}1,
$$
where
\begin{equation}\label{def of A_lambda}
\mathcal{A}_\lambda = \left\{\sum_{0 \leq j < \lambda}c_j\left(b^j + b^{L-j}\right) \ : \ 0 \leq c_0, \ldots, c_{\lambda-1} < b, \ (c_0,b)=1 \right\}.
\end{equation}

The notation $b^{\lambda}$ is somewhat cumbersome and we set
\begin{equation}\label{q =b^lambda definition}
q = b^{\lambda}.
\end{equation}
Note the restriction $(n,b)=1$ is equivalent to $(n,q)=1$. Moreover, as $\lambda$ is even by assumption, $q$ is a square. This is helpful as then the quadratic Kloosterman sums with (square) modulus $q$ have global Sali\'e-type formulae, derivable without use of the Chinese remainder theorem (see Lemma \ref{lem:Quadratic Kloosterman sums with square moduli}). 

We fix smooth compactly supported functions $\eta,\theta : \mathbb{R} \to \mathbb{R}_0^+$ with $\eta$ supported on $[4^{-1},4]$ and $\theta$ supported on $[-2b,2b]$ satisfying 
$
\eta(t) \geq 1$ for $1/2 \leq t \leq 1$, $\theta(t) \geq 1$ for $|t| \leq b$ and 
$\|\eta^{(j)}\|_\infty, \|\theta^{(j)}\|_\infty \ll_j 1$ for each $j\geq 0$. For minor technical reasons we also restrict $\theta$ to be even. 
Then
\begin{equation}\label{eqn:smoothing S }
S \leq S_{\eta,\theta}(X,N,q,\mq,\ma),
\end{equation}
where
\begin{equation}\label{def of smooth S}
S_{\eta,\theta}(X,N,q,\mq,\ma) = \sum_{a \in \mathcal{A}_\lambda} \sum_{\substack{(n,q\mq)=1}} \eta(n/N)\sum_{\substack{m \equiv -a\overline{q} \  (n^2) \\ m \equiv (\ma-a)\overline{q} \ (\mq)}}\theta(q^2m/X).
\end{equation}

The aim in this section is to establish the following proposition. First for an integer $m \in [0, bX)$, recall we have defined
\begin{equation}\label{first def of rho}
\rho(m) = X\sum_{j\geq 0} d_j(m)b^{-j},
\end{equation}
where $d_j(m)$ denotes the $j$-th digit of $m$ in its $b$-adic expansion 
$$m = \sum_{j\geq 0}d_j(m)b^j.$$ 

\begin{prop}\label{prop: main term for smooth S}
With the assumptions and notations as before, and with the additional notations $M = X/N^2$ and $\eta_2(t) = \eta(t)/t^2$,  
\begin{align*}
&S_{\eta,\theta}(X,N,q,\mq,\ma) = \dfrac{\widehat{\eta_2}(0)\widehat{\theta}(0) \varphi^2(b) \varphi(\mq) }{b^2 \mq }\dfrac{MN}{q\mq}\\
&\qquad + X^{7\epsilon}O\left( \sqrt{\mq} + \sqrt{N} +  \min\left(\dfrac{N}{q\sqrt{\mq}}, \ \dfrac{MN}{\mq q}\right)  + \dfrac{M }{ q\mq} +  \sup_{X^{2\epsilon} \leq r \leq qX^{\epsilon}}\dfrac{r|E_>(r)|}{q}\right)
\end{align*}
for any fixed $\epsilon > 0$,
where
\begin{equation}\label{def of E>(r) in prop}
E_>(r) = \dfrac{M}{q\mq r} \sum_{\substack{h \asymp q\mq r/M \\ 1 \leq m \leq q \\ (m,q)=1 \\ (n,q\mq)=1}}f_h(n/N)e_q\left(mh\overline{\mq n^2}\right)e_{\mq}\left(\ma h\overline{q n^2}\right)e\left(-h\dfrac{\rho(m) + m}{q\mq n^2}\right)
\end{equation}
with
\begin{equation}\label{def of f_h}
f_h(t) = \eta_2(t)\widehat{\theta}\left(\dfrac{Mh}{q^2\mq t^2}\right).
\end{equation}
\end{prop}

We now proceed to prove the proposition.

\subsection{Translation to exponential sums}
For convenience, let 
$$S_{\eta,\theta} = S_{\eta,\theta}(X,N,q,\mq,\ma).$$
Combining the two congruences in (\ref{def of smooth S}) via the Chinese remainder theorem, applying Poisson summation (Lemma \ref{lem:poisson summation}) to the $m$-sum and then using B\'ezout's identity (Lemma \ref{lem:Bezout}) gives
\begin{align*}
S_{\eta,\theta} &= \dfrac{X}{q^2\mq} \sum_{\substack{a \in \mathcal{A}_\lambda \\ |h| \leq N^2q^2\mq X^{\epsilon}/X \\ (n,q\mq)=1}} \dfrac{\eta(n/N)}{n^2}\widehat{\theta}\left(\dfrac{Xh}{q^2\mq n^2}\right)e_{n^2}\left(-a h \overline{q \mq}\right)e_{\mq}\left((\ma - a)h\overline{qn^2}\right)\\
&\qquad + O\left(X^{-A}\right)
\end{align*}
for any fixed $A,\epsilon > 0$. 
By Lemma \ref{lem:coprime sums} and the fact $\#\mathcal{A}_\lambda = \varphi(b)q/b$, the term with $h=0$ contributes
$$
\dfrac{X\widehat{\theta}(0)}{q^2\mq}\sum_{a \in \mathcal{A}_\lambda}\sum_{(n,b\mq)=1}\dfrac{\eta(n/N)}{n^2} = \dfrac{\varphi^2(b)\varphi(\mq)\widehat{\eta_2}(0)\widehat{\theta}(0) X}{b^2Nq\mq^2}\left(1 + O\left(\dfrac{\tau(\mq)}{N}\right)\right) ,
$$
where
$
\eta_2(t) = \eta(t)/t^2.
$
Multiplying and dividing by $N^2$ and noticing
$N^2\eta(n/N)/n^2 = \eta_2(n/N)$, 
we have 
\begin{equation}\label{main term + error}
S_{\eta,\theta} = \dfrac{\varphi^2(b)\varphi(\mq)\widehat{\eta_2}(0)\widehat{\theta}(0) X}{b^2Nq\mq^2}\left(1 + O\left(\dfrac{\tau(\mq)}{N}\right)\right) + E ,
\end{equation}
where
\begin{equation}\label{def of E}
 E = \dfrac{X}{q^2\mq N^2} \sum_{\substack{a \in \mathcal{A}_\lambda \\ 1 \leq |h| \leq N^2q^2\mq X^{\epsilon}/X \\ (n,q\mq)=1}} \eta_2(n/N)\widehat{\theta}\left(\dfrac{Xh}{q^2\mq n^2}\right)e_{n^2}\left(-a h \overline{q \mq}\right)e_{\mq}\left((\ma - a)h\overline{qn^2}\right).
\end{equation}
By B\'{e}zout's identity, 
$$
\dfrac{\overline{q \mq}}{n^2} \equiv \dfrac{1}{q\mq n^2} -\dfrac{\overline{n}^2}{q\mq} \equiv  \dfrac{1}{q\mq n^2} - \dfrac{\overline{\mq n^2}}{q} - \dfrac{\overline{q n^2}}{\mq}\pmod{1}.  
$$
It follows
$$
E = \dfrac{M}{q^2\mq} \sum_{\substack{a \in \mathcal{A}_\lambda \\ 1 \leq |h| \leq q^2\mq X^{\epsilon}/M \\ (n,q\mq)=1}} \eta_2\left(\dfrac{n}{N}\right)\widehat{\theta}\left(\dfrac{Xh}{q^2\mq n^2}\right)e_q\left(ah\overline{\mq n^2}\right)e_{\mq}\left(\ma h\overline{q n^2}\right)e\left(-\dfrac{ah}{q\mq n^2}\right),
$$
where
$
M = X/N^2.
$

To each $a \in \mathcal{A}_\lambda$ corresponds a unique integer $1 \leq m \leq q$ with $(m,q)=1$ such that
$
a = m + \rho(m),
$
where $\rho$ is defined as in Lemma \ref{lem:Reversal-type function} with $b^L = X$ there (as assumed thus far). 
The correspondence is bijective. It is important to note $\rho(m) \equiv 0 (q)$ for any such $m$. Indeed, this follows from the constraint $m \leq q$, Property (B) in Lemma \ref{lem:Reversal-type function} and the assumptions $q \leq \sqrt{X}$ with $q,X$ powers of $b$. Thus
\begin{align*}
E &= \dfrac{M}{q^2\mq} \sum_{\substack{ 1 \leq |h| \leq q^2\mq X^{\epsilon}/M \\ 1 \leq m \leq q\\ (m,q)=1\\(n,q\mq)=1}} f_h\left(\dfrac{n}{N}\right)e_q\left(mh\overline{\mq n^2}\right)e_{\mq}\left(\ma h\overline{q n^2}\right)e\left(-h\dfrac{\rho(m) + m}{q\mq n^2}\right)
\end{align*}
with $f_h$ defined in (\ref{def of f_h}). Splitting the $h$-sum, we have
\begin{equation}\label{splitting E in terms of E< and E>}
E = E_< + E_>
\end{equation}
with $E_<, E_>$ defined as 
\begin{align}
E_< &= \dfrac{M}{q^2\mq} \sum_{\substack{ 1 \leq |h| \leq q\mq X^{2\epsilon}/M \\ 1 \leq m \leq q\\ (m,q)=1\\(n,q\mq)=1}} f_h\left(\dfrac{n}{N}\right)e_q\left(mh\overline{\mq n^2}\right)e_{\mq}\left(\ma h\overline{q n^2}\right)e\left(-h\dfrac{\rho(m) + m}{q\mq n^2}\right),\label{def of E<}\\
E_> &= \dfrac{M}{q^2\mq} \sum_{\substack{ \frac{q\mq X^{2\epsilon}}{M} \leq |h| \leq \frac{q^2\mq X^{\epsilon}}{M} \\ 1 \leq m \leq q\\ (m,q)=1\\(n,q\mq)=1}} f_h\left(\dfrac{n}{N}\right)e_q\left(mh\overline{\mq n^2}\right)e_{\mq}\left(\ma h\overline{q n^2}\right)e\left(-h\dfrac{\rho(m) + m}{q\mq n^2}\right).\label{def of E>}
\end{align}
With regards to $E_>$, a dyadic decomposition gives
\begin{equation}\label{bound in terms of E>(r)}
E_> \ll (\log X) \sup_{X^{2\epsilon} \leq r \leq qX^\epsilon}\dfrac{r|E_{>}(r)|}{q},
\end{equation}
where
\begin{equation}\label{def of E_>(r)}
E_{>}(r) = \dfrac{M}{q\mq r} \sum_{\substack{h \asymp q\mq r/M \\ 1 \leq m \leq q \\ (m,q)=1 \\ (n,q\mq)=1}}f_h(n/N)e_q\left(mh\overline{\mq n^2}\right)e_{\mq}\left(\ma h\overline{q n^2}\right)e\left(-h\dfrac{\rho(m) + m}{q\mq n^2}\right)
\end{equation}
with some suitable implied constants in the $\asymp$. We tackle $E_>(r)$ in the upcoming sections. For now we only bound $E_<$. 

\subsection{The sum $E_<$}\label{subsec:the sum E_<} 

Consider $E_<$ defined in (\ref{def of E<}) which we rewrite as  
\begin{equation}\label{eqn:start of upper bounds for E_<}
E_< = \dfrac{M}{q^2\mq} \sum_{1 \leq |h| \leq q\mq X^{2\epsilon}/M} \sum_{ \substack{1 \leq m \leq q \\ (m,q)=1}}\sum_{(n,q\mq)=1} f_{h,m}(n/N)e_q\left(hm\overline{\mq n^2}\right)e_{\mq}\left(\ma h\overline{q n^2}\right),
\end{equation}
where
\begin{equation}\label{def of f_{h,m}}
f_{h,m}(t) = f_h(t)e\left(-h\dfrac{\rho(m) + m}{q\mq N^2t^2} \right) = \eta_2(t)\widehat{\theta}\left(\dfrac{Mh}{q^2\mq t^2}\right)e\left(-h\dfrac{\rho(m) + m}{q\mq N^2t^2} \right).
\end{equation}
For $h,m$ satisfying the constraints under the sums, by the assumption $q^2 \leq X = MN^2$ and the fact that
$
X/q
$
divides $\rho(m)$
for $1 \leq m \leq q$ (see Lemma \ref{lem:Reversal-type function}) we have
$$
\dfrac{m|h|}{q\mq N^2} \leq  \dfrac{M|h|}{q^2\mq} \leq \dfrac{\rho(m)|h|}{q \mq N^2} \ll X^{2\epsilon}.
$$
Since $\eta_2$ is compactly supported on some fixed interval bounded away from zero, several differentiations show
\begin{equation}\label{size of derivatives of f_h,m}
\left\|f_{h,m}^{(j)}\right\|_\infty\ll_j X^{2\epsilon j}
\end{equation}
for any $j\geq 0$. 

Splitting the $h$-sum according to the GCD of $h,q$ and substituting variables,
$$ 
E_< = \dfrac{M}{q^2\mq} \sum_{\substack{d\mid q }}\sum_{\substack{1 \leq |h| \leq \frac{d\mq X^{2\epsilon}}{M} \\ (h,d)=1}} \sum_{ \substack{1 \leq m \leq q \\ (m,q)=1}}\sum_{(n,q\mq)=1} f_{qh/d,m}(n/N)e_d\left(hm\overline{\mq n^2}\right)e_{\mq}\left(\ma h\overline{d n^2}\right).
$$
Note the restriction $(n,q)=1$ is equivalent to $(n,cd)=1$ for some divisor $c$ of $b$ coprime to $d$. Writing $\mathbf{1}_{(n,c)=1} = \sum_{c' \mid (n,c)}\mu(c')$, substituting $n$ with $c'n$, switching orders of summation and taking absolute values, it follows
\begin{equation}\label{eqn:bound for E_< in terms of E_<(d;,u,v)}
E_< \ll   \sum_{\substack{d\mid q }}\max_{\substack{u,v\mid b^2 \\ (uv,d)=1}} |E_<(d;u,v)|,
\end{equation}
where
\begin{equation}\label{def of E<(d;u,v)}
E_<(d;u,v) = \dfrac{M}{q^2\mq}  \sum_{\substack{1 \leq |h| \leq \frac{d\mq X^{2\epsilon}}{M} \\1 \leq m \leq q\\
		(h,d)=1\\
		(m,q)=1\\
		(n,d\mq)=1
}} 
f_{qh/d,m}(un/N)e_d\left(hm\overline{\mq v n^2}\right)e_{\mq}\left(\ma h\overline{dv n^2}\right).   
\end{equation}
Combining the two exponentials into one, applying Lemma \ref{lem:completion of squared inverses} and then using the Chinese remainder theorem, we have
\begin{align*}
&\sum_{(n,d\mq)=1} f_{qh/d,m}(un/N)e_d\left(hm\overline{\mq v n^2}\right)e_{\mq}\left(\ma h\overline{dv n^2}\right)\\
&= \dfrac{N\widehat{g_{qh/d,m}}(0)\G^*(\ma h\overline{dv};\mq)\G^*(hm\overline{v\mq};d)}{\sqrt{d\mq}} +O_A\left(X^{-A}\right)
\\
&\qquad 
+ \dfrac{N}{\sqrt{d\mq}} \sum_{1 \leq |k| \leq d\mq X^{3\epsilon}/N} \widehat{g_{qh/d,m}}\left(\dfrac{Nk}{d\mq}\right)
\K_2\left(hm\overline{v\mq}, k\overline{\mq};d\right)
\K_2\left(\ma h\overline{dv}, k\overline{d}; \mq \right)
\end{align*}
for any $A > 0$, where for convenience we have set
\begin{equation}\label{def of g}
g_{qh/d,m}(t) = f_{qh/d,m}(ut).
\end{equation}

Inserting that in (\ref{def of E<(d;u,v)}) we obtain
\begin{equation}\label{eqn:Sigma + S}
E_<(d;u,v) = \dfrac{MN\Sigma}{q^2d^{1/2}\mq^{3/2}}  + \dfrac{MN}{q^2d^{1/2}\mq^{3/2}}\sum_{ \substack{1 \leq m \leq q \\ (m,q)=1}}\mathcal{S}(m) + O_A\left(X^{-A}\right),
\end{equation}
where
\begin{align}\label{def of Sigma}
 \Sigma &= \sum_{\substack{1 \leq |h| \leq \frac{d\mq X^{2\epsilon}}{M} \\ 1 \leq m \leq q \\  (h,d)=(m,q)=1}} \sum_{1 \leq |k| \leq \frac{d\mq X^{3\epsilon}}{N}} \widehat{g_{qh/d,m}}\left(\dfrac{Nk}{d\mq}\right)
 \K_2\left(hm\overline{v\mq}, k\overline{\mq};d\right)
 \K_2\left(\ma h\overline{vd}, k\overline{d}; \mq \right)
 \end{align}
and
\begin{equation}\label{def of S(m,d)}
\mathcal{S}(m) =\sum_{\substack{1 \leq |h| \leq \frac{d\mq X^{2\epsilon}}{M} \\ (h,d)=1}}\widehat{g_{qh/d,m}}(0)\G^*(\ma h\overline{vd};\mq)\G^*(hm\overline{v\mq};d)
\end{equation}
with $(m,q)=1$ by assumption. 

\subsection{Upper bound for $\Sigma$}\label{subsec:upper bound for Sigma}
Since $g_{qh/d,m}(t)$ (defined in (\ref{def of g})) is supported on some interval $t \asymp 1$, 
$$
\widehat{g_{qh/d,m}}\left(\dfrac{Nk}{d\mq}\right) = \int_{t \asymp 1} g_{qh/d,m}(t)e(-Nkt/d\mq)\der t.
$$
Switching orders of summation/integration and defining
\begin{equation}\label{def of H,K}
H := \dfrac{d\mq X^{2\epsilon}}{M}, \hspace{2em} K := \dfrac{d\mq X^{3\epsilon}}{N},
\end{equation}
we may majorize $\Sigma$ as
\begin{align}\label{Sigma in terms of Sigma_gamma}
\Sigma &\ll \sup_{\substack{\gamma  }}\sum_{\substack{1 \leq |h| \leq H \\ (h,d)=1}}\sum_{m(q)} \left|\sum_{1 \leq |k| \leq K}\gamma_k\K_2\left(hm\overline{v\mq}, k\overline{\mq};d\right)
\K_2\left(\ma h\overline{vd}, k\overline{d}; \mq \right)\right| \nonumber\\
&= q\sup_{\substack{\gamma  }} \sum_{\substack{1 \leq |h| \leq H \\ (h,d)=1}}\dfrac{1}{d}\sum_{m(d)} \left|\sum_{1 \leq |k| \leq K}\gamma_k\K_2\left(m, k\overline{\mq};d\right)
\K_2\left(\ma h\overline{vd}, k\overline{d}; \mq \right)\right|,
\end{align}
where, in the supremum, $\gamma$ runs over sequences of $1$-bounded complex numbers $\gamma_k$.
The last equality holds since each $(hv\mq,d)=1$ and the summand is $d$-periodic in $m$ with $d\mid q$ by assumption. 

Denote by $\Sigma_\gamma$ the expression to the right of $\sup_{\gamma}$ on the last equality. 
Applying Cauchy-Schwarz, expanding the square and switching orders of summation, 
$$
\Sigma^2_\gamma \ll H \sum_{\substack{1 \leq |k_1|, |k_2| \leq K \\ 1 \leq |h| \leq H \\ (h,d)=1}}\gamma_{k_1}\overline{\gamma_{k_2}}  \K_2\left(\ma h\overline{vd}, k_1\overline{d}; \mq \right)\overline{\K_2\left(\ma h\overline{vd}, k_2\overline{d}; \mq \right)}\mathscr{Z}(k_1,k_2),
$$
where
$$
\mathscr{Z}(k_1,k_2) = 
\dfrac{1}{d}\sum_{m(d)}\K_2\left(m, k_1\overline{\mq};d\right)\overline{\K_2\left(m, k_2\overline{\mq};d\right)}.
$$
By Lemma \ref{lem: Shparlinski bound}, the diagonal terms with $k_1 = k_2$ contribute
$$
\ll X^{\epsilon/2} H\sum_{1 \leq h \leq H}(h,\mq)\sum_{1 \leq k \leq K}(k,d) \ll X^\epsilon H^2K.
$$
By Lemma \ref{lem:Correlations of K_2 sums}, the off-diagonal terms contribute
\begin{align*}
&\ll \dfrac{HX^{\epsilon/2}}{d}\sum_{1 \leq h \leq H}(h,\mq)\sum_{1 \leq k_1 \neq k_2 \leq K}(k_1^2-k_2^2,d)\\
& \ll \dfrac{H^2X^{\epsilon/2}\tau(\mq)}{d} \sum_{1 \leq k_1 \neq k_2 \leq K}(k_1-k_2,d)(k_1 + k_2,d)\\
& \ll \dfrac{H^2X^{\epsilon/2}\tau(\mq)}{d} \sum_{1 \leq \ell \leq K}(\ell,d)\sum_{1 \leq k \ll K}(k,d)\\
&\ll 
\dfrac{H^2K^2X^{\epsilon}}{d}.
\end{align*}
The inequality before the last follows from Weyl-differencing and a slight increase in the length of one sum. It follows $\Sigma_\gamma X^{-\epsilon} \ll H\sqrt{K} + HK/\sqrt{d}$. By (\ref{Sigma in terms of Sigma_gamma}) and the definitions of $H,K$ in (\ref{def of H,K})
\begin{equation}\label{Sigma final bound}
\Sigma \ll \dfrac{qd^2\mq^2X^{6\epsilon}}{MN}\left( \sqrt{\dfrac{N}{d\mq}} +  \dfrac{1}{\sqrt{d}}\right).
\end{equation}

\subsection{Upper bound for $\mathcal{S}(m)$}\label{subsec:Upper bound for S(m)}
Recall that by definition,
$$
\mathcal{S}(m) =\sum_{\substack{1 \leq |h| \leq \frac{d\mq X^{2\epsilon}}{M} \\ (h,d)=1}}\widehat{g_{qh/d,m}}(0)\G^*(\ma h\overline{vd};\mq)\G^*(hm\overline{v\mq};d)
$$
for $1 \leq m \leq q$ with $(m,q)=1$. 
By the definitions of $f_{qh/d}, g_{qh/d}$ in (\ref{def of f_{h,m}}), (\ref{def of g}) with $u \mid b^2$ there, 
$$
g_{qh/d}(t) = \eta_2(ut)\widehat{\theta}\left(\dfrac{Mh}{qd \mq u^2t^2}\right)e\left(-h \dfrac{\rho(m) + m}{d\mq N^2 u^2t^2}\right) .
$$
Since $\eta_2(t)$ is supported on $t \asymp 1$, so is $g_{qh/d}(t)$. For any such $t$ and by the assumption $X = MN^2$, each
$$
ut, \ \dfrac{Mh}{qd \mq u^2t^2}, \ h \dfrac{\rho(m) + m}{d\mq N^2 u^2t^2} \ll X^{2\epsilon}.  
$$
Thus if we expand the Fourier transform $\widehat{g_{qh/d,m}}(0)$, switch orders of summation and integration, sum by parts and conjugate the sums over $h < 0$, 
$$
\mathcal{S}(m) \ll X^{2\epsilon}\sup_{1 \leq H \leq d\mq X^{2\epsilon}/M}\left|\sum_{\substack{1 \leq h \leq H \\ (h,d)=1}}\G^*(\ma h\overline{vd};\mq)\G^*(hm\overline{v\mq};d)\right|.
$$
By Lemma \ref{lem: Gauss sum variant}, if $(a,d)=1$ and $\G^*(a;d) \neq 0$, then $\odd(d)$ is square-free and $16 \nmid d$. Since $d \mid q$ and $q$ is a power of $b$ by assumption, the last implies $d \mid 4b$. Recalling $(hmv\mq,d)=1$ by assumptions, 
\begin{align}
\mathcal{S}(m) &\ll \mathbf{1}_{d \mid 4b}X^{2\epsilon}\sup_{1 \leq H \leq d\mq X^{2\epsilon}/M}\left|\sum_{\substack{1 \leq h \leq H \\ (h,d)=1}}\G^*(\ma h\overline{vd};\mq)\G^*(hm\overline{v\mq};d)\right| \nonumber\\
&\ll
\mathbf{1}_{d \mid 4b}X^{2\epsilon}\sup_{\substack{1 \leq H \leq d\mq X^{2\epsilon}/M \\ c \in \mathbb{Z}}}\left|\sum_{\substack{1 \leq h \leq H }}\G^*(\ma h\overline{vd};\mq)e_d(ch)\right|\nonumber\\
&\ll
\mathbf{1}_{d \mid 4b}\min\left(\dfrac{\mq }{M}, \ \sqrt{\mq}  \right)X^{5\epsilon}. \label{S(m) final bound}
\end{align}
The inequality before the last follows after a Fourier expansion modulo $d$ of 
$$h \mapsto \mathbf{1}_{(h,d)=1}\G^*(hm\overline{v\mq};d).$$ The last inequality follows from Lemma \ref{lem:sums of K_2 sums}.

\subsection{Concluding}\label{subsec:upper bound for E_<} 

By (\ref{Sigma final bound}), (\ref{S(m) final bound}) and (\ref{eqn:Sigma + S}),
$$
E_{<}(d;u,v)X^{-6\epsilon} \ll \dfrac{d\sqrt{N} + d\sqrt{\mq}}{q} + \mathbf{1}_{d\mid 4b}\min\left(\dfrac{N}{q\sqrt{\mq}},  \ \dfrac{MN}{q\mq}\right).
$$
By (\ref{eqn:bound for E_< in terms of E_<(d;,u,v)}),
$$
E_<X^{-7\epsilon} \ll \sqrt{N} + \sqrt{\mq} + \min\left(\dfrac{N}{q\sqrt{\mq}},  \ \dfrac{MN}{q\mq}\right). 
$$
Proposition \ref{prop: main term for smooth S} now follows from this, (\ref{main term + error}), (\ref{splitting E in terms of E< and E>}), (\ref{bound in terms of E>(r)}) and (\ref{def of E_>(r)}).

\section{Large $N$ relative to conductors}\label{sec:BAB-process}

Here we prove the following bound, applicable in a regime with $N$ sufficiently large relative to $q,\mq$. The argument in the proof goes along the lines of one $B$-process on the $n$-sum, followed by an $A$-type process on the $m$-sum involving the reversal function, followed by a further $B$-process on the dual sum emanating from the first $B$-process. Following the proof, we show, in Proposition \ref{prop:Factorizing the d-sum}, that a relevant sum appearing in the proof of Proposition \ref{prop: N large for squares n^2 with n of size N}, can be factorized into a product of two large sums. This facilitates a more efficient application of Cauchy-Schwarz, compared to that employed in the proof of Proposition \ref{prop: N large for squares n^2 with n of size N}, from which stronger bounds may be derived. Note however that Proposition \ref{prop:Factorizing the d-sum} is an  addenda incorporated after the establishment of our main results and will not be employed in their proofs. See also the discussion following the proof of Proposition \ref{prop: N large for squares n^2 with n of size N} and Remark \ref{Remark on factorization consequences} at the end.

\begin{prop}\label{prop: N large for squares n^2 with n of size N}
	With the same assumptions of \textnormal{Proposition \ref{prop: main term for smooth S}} and $X^{2\epsilon} \ll r \ll qX^\epsilon$,
	$$
	E_>(r) \ll  X^{4\epsilon}\left(q^{3/2}\mq^{1/4}N^{1/4} + q^{7/4}\mq^{1/2} + q\sqrt{N}\right).
	$$
	\end{prop}

\begin{proof}
By definition,
$$
E_>(r) = \dfrac{M}{q\mq r}  \sum_{ \substack{h \asymp q\mq r/M \\ 1 \leq m \leq q \\ (m,q)=1 \\ (n,q\mq)=1} }f_{h}(n/N)e_q\left(h m\overline{\mq n^2}\right)e_{\mq}\left(\ma h \overline{qn^2}\right)e\left(-\dfrac{h\rho(m) + hm}{q\mq n^2} \right),
$$
where
$$
f_{h}(t) = \eta_2(t)\widehat{\theta}\left(\dfrac{Mh}{q^2\mq t^2}\right). 
$$
Note also that since $\theta$ is real-valued and even by assumption, so is $ \widehat{\theta}$. Hence after conjugating the sum over $h > 0$, 
\begin{equation}\label{bound for E> in terms of sigma}
E_>(r) \ll \max_{\substack{\sigma \in \{\pm 1\}}}\left|E_{>}(r;\sigma)\right|,
\end{equation}
where
$$
E_{>}(r;\sigma) =  \dfrac{M}{q\mq r}  \sum_{ \substack{1 \leq h \asymp q\mq r/M \\ 1 \leq m \leq q \\ (m,q)=1 \\ (n,q\mq)=1} }f_{h}(n/N)e_q\left(\sigma h m\overline{\mq n^2}\right)e_{\mq}\left(\sigma \ma h \overline{qn^2}\right)e\left(T(m,h)\dfrac{N^2}{n^2} \right) 
$$
with
\begin{equation}\label{def of T(m,h)}
T(m,h) = T(m,h;q,\mq,N) = \dfrac{h \rho(m) + hm}{q\mq N^2}.
\end{equation}

For $h,m$ satisfying the constraints under the sum and by the assumption $M = X/N^2$, 
\begin{equation}\label{size of T}
T(m,h) \asymp \dfrac{h\rho(m)}{q\mq N^2} \asymp \dfrac{r\rho(m)}{X} \asymp r \gg X^{2\epsilon}
\end{equation}
by Property (E) in Lemma \ref{lem:Reversal-type function} and the assumption $r \gg X^{2\epsilon}$. We also have
\begin{equation}\label{bound for derivative of f_h}
\left\|f_{h}^{(j)}\right\|_\infty \ll_j 1 + \left(\dfrac{Mh}{q^2\mq}\right)^j \ll_j 1 + \dfrac{r^j}{q^j} \ll_j X^{\epsilon j}
\end{equation}
for each $j\geq 0$, by the assumption $r \ll qX^\epsilon$. Define $\widetilde{f}_{h}$ similarly as $\widetilde{\eta}$ was defined in Lemma \ref{lem:integrals with squared inverses} and let $F_{h}(t) = \sqrt{t} \widetilde{f}_{h}(t)$. Note $t \asymp 1$ positive in the support of $t \mapsto F_h(t)$ and
\begin{equation}\label{derivative of F_h}
\left\|F_{h}^{(j)}\right\|_\infty \ll_j X^{\epsilon j}
\end{equation}
for each $j \geq 0$. This follows from (\ref{bound for derivative of f_h}). 
Setting $T = T(m,h)$, combining the two arithmetic exponentials into one,  applying Lemma \ref{lem:twisted sums with squared inverses} and then using the Chinese remainder theorem,
\begin{align*}
&\sum_{(n,q\mq)=1}f_{h}\left(\dfrac{n}{N}\right)e_q\left(\sigma h m\overline{\mq n^2}\right)e_{\mq}\left(\sigma \ma h \overline{qn^2}\right)e\left( \dfrac{ TN^2}{n^2} \right) \\
&= \sqrt{N}\sum_{n\geq 1} F_{h}\left(\dfrac{Nn}{q\mq T}\right)\dfrac{\K_2\left(\sigma hm\overline{\mq}, - n\overline{\mq}; q\right)\K_2\left(\sigma \ma h\overline{q}, - n\overline{q}; \mq\right)}{\sqrt{n}}e\left(3T \left(\dfrac{Nn}{2q\mq T}\right)^{2/3}\right)\\
&\qquad
+ O\left(\sqrt{\dfrac{q\mq }{T}} X^{3\epsilon}\right).
\end{align*}
Since $T \asymp r$, the contribution of the error term to $E_>(r;\sigma)$ is 
$\ll q^{3/2} \sqrt{\mq} X^{3\epsilon}/\sqrt{r} $.
Recall that $q$ is a square by assumption and so $\sqrt{q}$ is an integer. Then by Lemma \ref{lem:Quadratic Kloosterman sums with square moduli}, 
$$
\K_2\left(\sigma hm\overline{\mq}, - n\overline{\mq}; q\right) = \sum_{\substack{1 \leq \ell \leq \sqrt{q} \\ (\ell,q)=1 \\ n\ell^3 \equiv -2\sigma hm (\sqrt{q})}} e_{q}\left(\sigma h m \overline{\mq \ell^2} -n \ell\overline{\mq}\right).
$$
Thus
\begin{align*}
&E_{>}(r;\sigma) = O\left(\dfrac{q^{3/2}\sqrt{\mq} X^{3\epsilon}}{\sqrt{r}}\right)+\\
&\dfrac{M }{q\mq r} \sum_{\substack{1 \leq h \asymp \frac{q\mq r}{M} \\ 1 \leq \ell \leq \sqrt{q}\\1 \leq m \leq q \\ n\geq 1 \\ (\ell m,q)=1 \\ n\ell^3 \equiv -2\sigma h m (\sqrt{q})}}
\sqrt{\dfrac{N}{n}}F_{h}\left(\dfrac{Nn}{q\mq T}\right)\
\gamma_{h,n}e_{q}\left(\sigma h m \overline{\mq \ell^2} -n \ell\overline{\mq}\right)e\left(3T^{\frac{1}{3}} \left(\dfrac{Nn}{2q\mq}\right)^{\frac{2}{3}}\right)
,
\end{align*}
where
\begin{equation}\label{gamma_h,n}
\gamma_{h,n} = \K_2\left(\sigma \ma h\overline{q}, - n\overline{q}; \mq\right).
\end{equation}

Let $b \leq s \mid \sqrt{q}$, to be chosen later, be an integral power of $b$. Then so is $q/s$ and $ \sqrt{q} \mid q/s$. We may write $m = a + qd/s$ with $1 \leq a \leq q/s$, $(a,q)=1$, and $0 \leq d < s$. Note the fact $\sqrt{q} \mid q/s$ implies $-2\sigma h(a + qd/s)  \equiv -2\sigma a h (\sqrt{q})$ uniformly in $d$. Moreover in the support of $n \mapsto F_h(nN/q\mq T)$, we have $n \asymp q\mq T/N \asymp q\mq r/N$. The last holds by (\ref{size of T}). Thus $\sqrt{N/n} \asymp N/\sqrt{q\mq r}$. Then after moving the $d$-sum to the inside and taking absolute values, 
\begin{align}\label{moving the d-sum to the inside}
 &E_{>}(r;\sigma)\ll 
\dfrac{q^{3/2} \sqrt{\mq}X^{3\epsilon}}{\sqrt{r}} \nonumber\\
&+
\dfrac{M N}{(q\mq r)^{\frac{3}{2}}} \sum_{\substack{1 \leq a \leq q/s \\1 \leq h \asymp \frac{q\mq r}{M} \\ \ell (\sqrt{q})\\ 1 \leq n \asymp q\mq r/N \\ (a \ell ,q)=1 \\ n\ell^3 \equiv -2\sigma ah  (\sqrt{q})}} \left| \gamma_{h,n}\sum_{0\leq d < s}
F_{h}\left(\dfrac{Nn}{q\mq T}\right)
e_{s}\left(\sigma h d \overline{\mq \ell^2} \right)e\left(3T^{\frac{1}{3}} \left(\dfrac{Nn}{2q\mq}\right)^{\frac{2}{3}}\right)\right|.
\end{align}
 Recall also that $T = T(m,h)$
so that now 
\begin{equation}\label{T rewritten}
T = T\left(a + \dfrac{dq}{s},h\right).
\end{equation}

Consider the $\ell$-sum, which we move to the inside. It equals
$$
\sum_{\substack{\ell (\sqrt{q})\\ n\ell^3 \equiv -2\sigma ah (\sqrt{q})}}^* \alpha(a,h,\ell,n),
$$
where $\alpha(a,h,\ell,n) \geq 0$ denotes the absolute value of the $d$-sum in (\ref{moving the d-sum to the inside}). Note $\alpha(a,h,\ell,n)$ is $s$-periodic in the variable $\ell$ and recall also that $s \mid \sqrt{q}$. Thus if we split the sum according to the class of $\ell$ modulo $s$, 
\begin{align}\label{reduction to mod s}
\sum_{\substack{\ell (\sqrt{q})  \\ n\ell^3 \equiv -2\sigma ah (\sqrt{q})}}^* \alpha(a,h,\ell,n)
= 
\sum_{\substack{m(s) \\ nm^3 \equiv -2\sigma a h (s)}}^* \alpha(a,h,m,n) \beta(-2\sigma ah,m,n)
,
\end{align}
where for integers $k,m,n$,
\begin{equation}\label{beta}
\beta(k,m,n) = \sum_{\substack{\ell (\sqrt{q}) \\ \ell\equiv m (s)\\ n\ell^3 \equiv k (\sqrt{q})  }}^* 1.
\end{equation}
Inserting that in (\ref{moving the d-sum to the inside}) and applying Cauchy-Schwarz, we have
\begin{equation}\label{bound in terms of Y,Z}
E_{>}(r;\sigma) \ll \dfrac{MN\sqrt{\mathscr{Y}\mathscr{Z}}}{(q\mq r)^{3/2}}+ \dfrac{q^{3/2}\sqrt{\mq}X^{3\epsilon}}{\sqrt{r}} ,
\end{equation}
where
\begin{equation}\label{def of Y and Z}
\mathscr{Y} = \sum_{\substack{1 \leq a \leq q/s \\1 \leq h \asymp \frac{q\mq r}{M} \\ m (s)\\ 1 \leq n \asymp q\mq r/N \\ (a m ,s)=1 \\ m^3n \equiv -2\sigma ah  (s)}}\alpha^2(a,h,m,n), \hspace{2em}
\mathscr{Z} = \sum_{\substack{1 \leq a \leq q/s \\1 \leq h \asymp \frac{q\mq r}{M} \\ m (s)\\ 1 \leq n \asymp q\mq r/N \\ (a m ,s)=1 \\ m^3n \equiv -2\sigma ah  (s)}}\left| \beta(-2\sigma ah,m,n)\gamma_{h,n}\right|^2. 
\end{equation}
Note we may assume $q\mq r\gg M,N$, as otherwise the sums are zero. 

Let us first consider $\mathscr{Z}$. By the definition of $\gamma_{h,n}$ in (\ref{gamma_h,n}) and Lemma \ref{lem: Shparlinski bound}, we have $|\gamma_{h,n}|^2 \ll_{\delta}\mq^{\delta}(n,\mq)$ for any $\delta > 0$. We can majorize $\mathscr{Z}$ as
$$
\mathscr{Z} \ll_\delta X^{\delta}\sum_{\substack{k \ll q^2\mq r/Ms \\ m (s)\\ n \asymp q\mq r/N}}(n,\mq)\beta^2(k,m,n).
$$
Recalling the definition of $\beta$ in (\ref{beta}), expanding the square and switching orders of summation, we have
$$
\mathscr{Z} \ll_\delta  X^{\delta}\sum_{n\asymp q\mq r/N}(n,\mq)\sum_{\substack{\ell,\ell' (\sqrt{q}) \\ \ell \equiv \ell' (s)\\ n\ell^3 \equiv n\ell'^3 (\sqrt{q})}}^*\sum_{\substack{m(s) \\ m\equiv \ell (s)}}\sum_{\substack{k \ll q^2\mq r/Ms \\ k\equiv n\ell^3 (\sqrt{q})}}1.
$$
Since $q\mq r \gg M$ and $s \mid \sqrt{q}$ by assumption, we have $q^2\mq r/Ms \gg \sqrt{q}$. Then the inner sum is $\ll q^{3/2}\mq r/Ms$. The $m$-sum equals $1$. Moreover for any integer $\ell'$ coprime to $\sqrt{q}$, 
$$
\sum_{\substack{\ell (\sqrt{q}) \\ n\ell^3 \equiv n \ell'^3 (\sqrt{q})}}^*1 = (n,\sqrt{q})\sum_{\substack{\ell (\sqrt{q}/(n,\sqrt{q})) \\ \ell^3 \equiv 1 (\sqrt{q}/(n,\sqrt{q}))}}^*1 \ll (n,\sqrt{q}) 
$$
and 
$$\sum_{n \asymp q\mq r/N} (n,\sqrt{q})(n,\mq)  = \sum_{n \asymp q\mq r/N} (n,\sqrt{q}\mq)\ll \dfrac{q\mq r \tau(q\mq)}{N}$$ 
by Lemma \ref{lem:sums of GCDs}. It follows
\begin{equation}\label{bound for Z}
\mathscr{Z} \ll  \dfrac{q^3\mq^2 r^2 X^{\epsilon}}{MNs}.
\end{equation}

Let us now consider $\mathscr{Y}$ defined in (\ref{def of Y and Z}). Extending the summation over $1 \leq n \asymp qr/N$ to all of $\mathbb{Z}$, substituting $\overline{m}$ with $m$, expanding the square and switching orders of summation, we have
\begin{align}\label{expanding the square of Y}
&\mathscr{Y} \leq \nonumber \\
& \sum_{\substack{1 \leq a \leq \frac{q}{s} \\ 0 \leq d,d' < s \\h \asymp \frac{q\mq r}{M} \\ m (s)\\ (a m ,s)=1 }} e_s\left(\sigma h (d-d')\overline{\mq}m^2\right)\sum_{n \equiv -2\sigma ahm^3 (s)} G_{a,d,d',h}\left(\dfrac{Nn}{q\mq r}\right)e\left(\Delta_{a,d,d',h}\left(\dfrac{Nn}{q\mq r}\right)^{\frac{2}{3}}\right),
\end{align}
where
\begin{equation}\label{def of G}
G_{a,d,d',h}(t) = F_h\left(\dfrac{rt}{T(a + qd/s, h)}\right)\overline{F_h\left(\dfrac{rt}{T(a + qd'/s, h)}\right)}
\end{equation}
and
\begin{equation}\label{def of Delta}
\Delta_{a,d,d',h} = \dfrac{3 r^{2/3}}{2^{2/3}}\left(T^{1/3}\left(a + \dfrac{qd}{s}, h\right) - T^{1/3}\left(a + \dfrac{qd'}{s}, h\right)\right).
\end{equation}

Recall $0 < T \asymp r \gg X^{2\epsilon}$ by (\ref{size of T}) and $F_h(t)$ is smooth compactly supported on some positive interval $t \asymp 1$. Then so is $G_{a,d,d',h}$ and  
\begin{equation}\label{derivative of G}
\left\|G_{a,d,d',h}^{(j)}\right\|_\infty \ll_j X^{\epsilon j}
\end{equation} for each $j \geq 0$, 
by (\ref{derivative of F_h}). 

With regards to the archimedean conductor $\Delta_{a,d,d',h}$ defined above, let us first recall that, by the definition in (\ref{def of T(m,h)}) and the bound in (\ref{size of T}),
$$
T\left(a + \dfrac{qd}{s}, h\right) = \dfrac{h\rho(a + qd/s) + ha + hqd/s}{q \mq N^2} \asymp r \gg X^{2\epsilon}.
$$
Since $q/s$ is an integral power of $b$ (as so are $s\mid \sqrt{q}$ and $q$ by assumptions) and $1 \leq a \leq q/s$ with $(a,s)=1$ (hence $(a,b)=1$) then Property (D) of Lemma \ref{lem:Reversal-type function} implies $\rho(a + qd/s) = \rho(a) + \rho(qd/s)$. For any $x,y > 0$ with $x \asymp y$, the fundamental theorem of calculus gives
$$
x^{1/3} - y^{1/3} = \dfrac{1}{3}\int_y^x \dfrac{dz}{z^{2/3}} \ll \dfrac{x-y}{x^{2/3}}.
$$
Since $T \asymp r$, it follows
\begin{equation}\label{size of U part 1}
\Delta_{a,d,d',h} \ll \dfrac{h|\rho(qd/s) - \rho(qd'/s) + q(d-d')/s|}{q\mq N^2}.
\end{equation}
For $0 \leq d,d' < s$, we have $qd/s,qd'/s < q$. By Property (B) of Lemma \ref{lem:Reversal-type function}, $X/q$ divides both $\rho(qd/s), \rho(qd'/s)$. Since $q \leq X/q$ (as $q \leq \sqrt{X}$ by assumption) and $\rho$ is injective on $[0, bX)$, then $|d-d'|q/s \leq  |\rho(qd/s) - \rho(qd'/s)|$. Thus for any $h \asymp q\mq r/M$ and by the assumption $X = MN^2$,  
\begin{equation}\label{Delta(a,d,d',h) in terms of Delta(d,d')}
\Delta_{a,d,d',h} \ll \dfrac{r|\rho(qd/s) - \rho(qd'/s)|}{X}.
\end{equation}
Note that since $q/s$ is an integral power of $b$, then Property (C) of Lemma \ref{lem:Reversal-type function} implies
\begin{equation}\label{bound for Delta(d,d')}
\rho(qd/s) - \rho(qd'/s) \ll \dfrac{Xs}{q},
\end{equation}
whence
\begin{equation}\label{Delta(a,d,d',h) in terms of Delta(d,d')}
\Delta_{a,d,d',h} \ll \dfrac{rs}{q} \ll sX^{\epsilon}.
\end{equation}
The last inequality holds since $r \ll qX^\epsilon$ by assumption. 

Let us now return to (\ref{expanding the square of Y}) which we split and rewrite as 
\begin{equation}\label{split Y}
\mathscr{Y} \leq \mathscr{Y}_< + \mathscr{Y}_>,
\end{equation}
where
\begin{equation}\label{def of Y1 and Y2}
\mathscr{Y}_< = \sum_{\substack{1 \leq a \leq q/s \\ 0 \leq d,d' < s \\ h \asymp q\mq r/M \\  (a,s)=1 \\ \Delta_{a,d,d',h} \ll X^{2\epsilon}}}\Phi\left(a,d,d',h\right), \hspace{2em} 
\mathscr{Y}_> = \sum_{\substack{1 \leq a \leq q/s \\ 0 \leq d,d' < s \\ h \asymp q\mq r/M \\  (a,s)=1 \\ \Delta_{a,d,d',h} \gg X^{2\epsilon}}}\Phi\left(a,d,d',h\right), 
\end{equation}
with
\begin{align}\label{def of Phi}
&\Phi\left(a,d,d',h\right) \nonumber\\
& = \sum_{m(s)}^* e_s\left(\sigma h(d-d')\overline{\mq}m^2\right) \sum_{n\equiv -2\sigma a h m^3 (s)}G_{a,d,d',h}\left(\dfrac{Nn}{q\mq r}\right)e\left(\Delta_{a,d,d',h}\left(\dfrac{Nn}{q\mq r}\right)^{2/3}\right).
\end{align}

Let us start by considering $\mathscr{Y}_<$. Note we may write the summand of the $n$-sum as $H_{a,d,d',h}(Nn/q\mq r), 
$
where 
\begin{equation}\label{def of H}
H_{a,d,d',h}(t) = G_{a,d,d',h}(t)e(\Delta_{a,d,d',h}t^{2/3}).\end{equation}
 By (\ref{derivative of G}) and the constraint $\Delta_{a,d,d',h} \ll X^{2\epsilon}$ under the sum in (\ref{def of Y1 and Y2}) corresponding to $\mathscr{Y}_<$, we have $\|H_{a,d,d',h}^{(j)}\|_\infty \ll_j X^{2\epsilon j}$ for each $j\geq 0$. Applying Lemma \ref{lem:poisson summation} and switching orders of summation, we have
\begin{align}\label{poisson for small amplitude}
\Phi(a,d,d',h) &= \dfrac{q\mq r\widehat{H_{a,d,d',h}}(0)\G^*(\sigma h (d-d')\overline{\mq};s)}{N\sqrt{s}} +O_A\left(X^{-A}\right)\nonumber\\
& + \dfrac{q\mq r}{\sqrt{s}N}\sum_{1 \leq |k| \leq \frac{sNX^{3\epsilon}}{q\mq r}}\widehat{H_{a,d,d',h}}\left(\dfrac{q\mq rk}{Ns}\right)\Ku_2\left( -2\sigma ahk, \sigma h (d-d')\overline{\mq};  s\right)
\end{align}
for any $A > 0$, where $\G^*$ is as defined in (\ref{def of Gauss type sum}) and $\Ku_2$ is the Kummer-type sum
\begin{equation}\label{Kummer def}
\Ku_2(c,d;s) = \dfrac{1}{\sqrt{s}}\sum_{m(s)}^*e_s\left(cm^3 + dm^2\right) \ll_\delta s^{\delta}\sqrt{(c,d,s)}.
\end{equation}
The bound above holds for any $\delta > 0$, by Lemma \ref{lem: Shparlinski bound}. For any integer $c$, Lemma \ref{lem: Gauss sum variant} gives 
\begin{align*}
\G^*(c;s) &\ll \mu^2\left(\odd\left(\dfrac{s}{(c,s)}\right)\right)\mathbf{1}_{16 \nmid s/(c,s)} \tau(s)\sqrt{(c,s)}\\
&\ll
\mathbf{1}_{s\mid 4bc}\tau(s)\sqrt{s}.
\end{align*}
The last inequality holds since $s$ is a power of $b$ by assumption. 
Thus by (\ref{poisson for small amplitude}), the assumption $(a,s)=1$ and the two bounds above for $\Ku_2, \G^*$, 
\begin{align*}
\Phi(a,d,d',h) &\ll_{A,\delta} \dfrac{q\mq r\tau(s)\mathbf{1}_{s\mid 4bh(d-d')} }{N} 
+ X^{-A} + \dfrac{q\mq r s^{\delta}}{N\sqrt{s}}\sum_{1 \leq k \leq NsX^{3\epsilon}/q\mq r}\sqrt{(2hk,s)}.
\end{align*}
Hence by the definition of $\mathscr{Y}_<$ in (\ref{def of Y1 and Y2}),
\begin{align}
\mathscr{Y}_< &\ll_{\delta} \dfrac{q^2\mq r\tau(s)}{Ns} \sum_{\substack{0 \leq d,d' < s \\ h \asymp q\mq r/M \\ s \mid 4bh(d-d')}}1 + \dfrac{q^2\mq rs^{1/2 + \delta}}{N}
\sum_{\substack{h \asymp q\mq r/M \\ 1 \leq k \leq NsX^{3\epsilon}/q\mq r}}\sqrt{(2hk,s)} \nonumber\\
 &\ll 
 \dfrac{q^2\mq  r X^{o(1)}}{Ns}\left(\dfrac{q\mq rs}{M} + s \sum_{\substack{1 \leq n \ll \frac{q\mq rs}{M} \\ s \mid n}}1\right) 
 + \dfrac{q^2\mq r\sqrt{s}X^{\delta + o(1)}}{N} \sum_{1 \leq n \ll \frac{NsX^{3\epsilon}}{M}}\sqrt{(n,s)} \nonumber .
\end{align}
It follows 
\begin{equation}\label{bound for Y<}
\mathscr{Y}_< \ll \dfrac{q^3 \mq^2 r^2 X^{o(1)}}{MN} + \dfrac{q^2 \mq r s^{3/2} X^{4\epsilon}}{M}.
\end{equation}

Consider now $\mathscr{Y}_>$ defined in (\ref{def of Y1 and Y2}). Applying Lemma \ref{lem:twisted Poisson summation} to the $n$-sum in the definition of $\Phi$ in (\ref{def of Phi}) and switching orders of summation, we have
$$
\Phi(a,d,d',h) = \dfrac{q\mq r}{N\sqrt{s \Delta_{a,d,d',h}}} \sum_{k} \Ku_2\left(-2\sigma a h k, \sigma h(d-d')\overline{\mq};s\right)J_{a,d,d',h}\left(\dfrac{q\mq rk}{Ns}\right),
$$
where
\begin{align}\label{def of J}
J_{a,d,d',h}\left(\dfrac{q\mq rk}{Ns}\right) &= \sqrt{\Delta_{a,d,d',h}} \widehat{H_{a,d,d',h}}\left(\dfrac{q\mq rk}{Ns}\right) \nonumber\\ &=\sqrt{\Delta_{a,d,d',h}}\int_{\mathbb{R}}G_{a,d,d',h}(t) e\left(\Delta_{a,d,d',h} t^{2/3} - \dfrac{q\mq rk }{Ns}t \right)dt
\end{align}
with $G_{a,d,d',h}, \Delta_{a,d,d',h}, \Ku_2$ defined in (\ref{def of G}), (\ref{def of Delta}) and (\ref{Kummer def}) respectively. Recall that $0 < t \asymp 1$ in the support of $G_{a,d,d',h}(t)$,  $\|G_{a,d,d',h}^{(j)}\|_\infty \ll_j X^{\epsilon j}$ for each $j\geq 0$ (by (\ref{derivative of G})) and $\Delta_{a,d,d',h} \gg X^{2\epsilon}$ by the constraint under the defining sum of $\mathscr{Y}_>$ in (\ref{def of Y1 and Y2}). As such, several integrations by parts, on the defining integral of $J_{a,d,d'h}(q\mq rk/Ns)$, show we may localize $k$ to some interval $k \asymp Ns\Delta_{a,d,d',h}/q\mq r$ with suitable implied constants in the $\asymp$ (here we make use of the non-stationary nature of the phase $\Delta_{a,d,d',h} t^{2/3} - q\mq rk t/Ns$ for $k$ outside of this interval and $t \asymp 1$ in the support of $G_{a,d,d',h}$). Thus
\begin{align}\label{localizing k}
&\Phi(a,d,d',h)  = O_{A}\left(X^{-A}\right)\nonumber\\
& + \dfrac{q\mq r}{N\sqrt{s \Delta_{a,d,d',h}}} \sum_{k \asymp Ns\Delta_{a,d,d',h}/q\mq r} \Ku_2\left(-2\sigma a h k, \sigma h(d-d')\overline{\mq};s\right)J_{a,d,d',h}\left(\dfrac{q\mq rk}{Ns}\right)
\end{align}
for any $A > 0$. Let $k \asymp Ns\Delta_{a,d,d',h}/q\mq r$. We may write the exponential's phase in (\ref{def of J}) as
$
\Delta_{a,d,d',h} t^{2/3} - q\mq rk t/Ns = \Delta_{a,d,d',h}\phi(t),
$
where 
$$
\phi(t) = t^{2/3} - \dfrac{q\mq rkt}{Ns \Delta_{a,d,d',h}}.
$$
Clearly $\phi''(t) \gg 1$ for $t \asymp 1$ in the support of $G_{a,d,d',h}$. Hence by 
van der Corput's lemma (see for instance Lemma 2.5 in the lecture notes by Tao \cite{Tao}) 
$$
\int_{0 < t \asymp 1} e\left(\Delta_{a,d,d',h}\phi(t)\right)dt \ll \dfrac{1}{\sqrt{\Delta_{a,d,d',h}}}.
$$
Thus if we integrate by parts once and recall $\|G_{a,d,d',h}'\|_\infty \ll X^\epsilon$, we find that
$$
J_{a,d,d',h}\left(\dfrac{q\mq rk}{Ns}\right) \ll X^{\epsilon}.
$$
By this, (\ref{localizing k}), (\ref{Kummer def}) and the assumption $(a,s)=1$,
\begin{equation}\label{final bound for Phi}
\Phi(a,d,d',h) \ll_A X^{-A} + \dfrac{q\mq r X^{2\epsilon}}{N\sqrt{s |\Delta_{a,d,d',h}|}} \sum_{k \asymp Ns\Delta_{a,d,d',h}/q\mq r}\sqrt{(2hk,s)}.
\end{equation}

By the definition of $\mathscr{Y}_>$ in (\ref{def of Y1 and Y2}), a dyadic decomposition, the bound $\Delta_{a,d,d',h} \ll sX^{\epsilon}$ in (\ref{Delta(a,d,d',h) in terms of Delta(d,d')}), and (\ref{final bound for Phi}),
\begin{align*}
\mathscr{Y}_> &\ll(\log X) \sup_{X^{2\epsilon} \ll D \ll sX^{\epsilon}} \sum_{\substack{1 \leq a \leq q/s \\ 0 \leq d,d' < s \\ h \asymp q\mq r/M \\ \Delta_{a,d,d',h} \asymp D}}\left| \Phi(a,d,d',h)\right|\\
&\ll_A X^{-A} + \sup_{X^{2\epsilon} \ll D \ll sX^{\epsilon}}\dfrac{q^2\mq rs^{1/2} X^{2\epsilon + o(1)}}{N|D|^{1/2}} \sum_{1 \leq n \ll NDs/M}(n,s)^{1/2}\\
&\ll \dfrac{q^2 \mq r s^2 X^{3\epsilon}}{M}.
\end{align*}

By the bound above, (\ref{bound for Y<}) and (\ref{split Y}),
\begin{equation}\label{bound for Y}
\mathscr{Y}X^{-4\epsilon} \ll \dfrac{q^3 \mq^2 r^2}{MN} + \dfrac{q^2 \mq r s^2}{M}.
\end{equation}

By (\ref{bound for Z}), (\ref{bound in terms of Y,Z}), (\ref{bound for E> in terms of sigma}) and the assumption $r \ll qX^{\epsilon}$,  
\begin{equation}\label{combining bounds}
E_>(r) X^{-4\epsilon}\ll \dfrac{q^{2}\mq^{1/2}}{s^{1/2}} + qN^{1/2}s^{1/2}
\end{equation}
for any proper power $s \leq \sqrt{q}$ of $b$.

If $q\sqrt{\mq/N} \ll 1$, then choosing $s = b$ gives
\begin{equation}\label{small s}
E_>(r) X^{-4\epsilon}\ll q^2\mq^{1/2} + q\sqrt{N} \ll q\sqrt{N}. 
\end{equation}
If $1 \ll q\sqrt{\mq/N} \ll \sqrt{q}$, a choice of $s \asymp q\sqrt{\mq/N}$ yields
\begin{equation}\label{medium s}
E_>(r) X^{-4\epsilon}\ll q^{3/2}\mq^{1/4}N^{1/4}. 
\end{equation}
Finally if $q\sqrt{\mq/N} \gg \sqrt{q}$, we select $s = \sqrt{q}$ obtaining
\begin{equation}\label{large s}
E_>(r) X^{-4\epsilon}\ll q^{7/4}\mq^{1/2} + q^{5/4}\sqrt{N} \ll q^{7/4}\mq^{1/2}.  
\end{equation}
In any case, 
$$
E_>(r) X^{-4\epsilon} \ll q^{3/2}\mq^{1/4}N^{1/4} + q^{7/4}\mq^{1/2} + q\sqrt{N}
$$
as required.
\end{proof}

\subsection*{Factorizing the $d$-sum in \textnormal{(\ref{moving the d-sum to the inside})}}\label{subsec:Factorizing the $d$-sum}

Since $\beta$ defined in (\ref{beta}) is essentially $\asymp 1$	(at least on average) we have (essentially) applied Cauchy-Schwarz with only one large sum on the inside, namely, the $d$-sum in (\ref{moving the d-sum to the inside}). As we know, Cauchy-Schwarz is more effective when there is, instead, a product of two large sums on the inside. Here we show that it is indeed possible (on average) to factorize that sum as a product of two large sums. For simplicity we assume $\mq=1$ here, but the proof easily adapts to the general case. See Proposition \ref{prop:Factorizing the d-sum} below. As mentioned earlier, the material in this subsection (noticed and added some time after attaining the main results of the article) will not be employed in the proof of our main results. This would otherwise imply lengthy revisions of the work. Nonetheless, we have decided to incorporate it as it may be beneficial to the researcher seeking to improve our results and/or adapt ideas in another context involving the reversal function. See also Remark \ref{Remark on factorization consequences} at the end.

\begin{prop}\label{prop:Factorizing the d-sum}
	Let $f(k) = \rho(k) + k$ and let 
	$$\rho_s(k) = \sum_{0 \leq j < \log_b(s)}d_j(k)b^{\log_b(s) - 1 - j}.$$ Then for any powers $b \mid U \mid s \mid \sqrt{q}$ of $b$, the normalized sum in \textnormal{(\ref{moving the d-sum to the inside})}, with $\mq = 1$ there, is
	\begin{equation}\label{factoring_d-sum}
	\ll q\sqrt{r}X^{2\epsilon }+
	\dfrac{MNX^{3\epsilon}}{(qr)^{\frac{3}{2}}} \sup_{\substack{\gamma \in [0,1] \\ \xi \ll X^{2\epsilon}/X}}\sum_{\substack{1 \leq a \leq q/s \\ 1 \leq h \asymp qr/M \\ \ell(\sqrt{q}) \\ 1 \leq n \asymp qr/N\\(a\ell,q)=1 \\ n\ell^3 \equiv -2\sigma a h (\sqrt{q})}}\mathscr{U}_U(a,h,\ell,n;\gamma,\xi)\mathscr{V}_U(a,h,\ell,n;\gamma,\xi),
	\end{equation}
	where $\mathscr{U}_U = \mathscr{U}_U(a,h,\ell,n;\gamma,\xi)$ and  	$\mathscr{V}_U = \mathscr{V}_U(a,h,\ell,n;\gamma,\xi)$ are defined as
	\begin{equation}\label{Usum}
	\mathscr{U}_U = \left|\sum_{0 \leq m < \frac{s}{U}}e_s\left(\sigma h m\overline{\ell}^2\right)e\left(\dfrac{h^{\frac{1}{3}}n^{\frac{2}{3}}f(qm/s)}{2^{\frac{2}{3}}q f^{\frac{2}{3}}(a)}\right) e\left(\xi f\left(\dfrac{qm}{s}\right) + \gamma \rho_s(m)\right)\right|
	\end{equation}
	and 
	\begin{equation}\label{Vsum}
	\mathscr{V}_{U} = \left|\sum_{0 \leq c < U}e_U\left(\sigma h c\overline{\ell}^2\right)e\left(\dfrac{h^{\frac{1}{3}}n^{\frac{2}{3}}f(qc/U)}{2^{\frac{2}{3}}q f^{\frac{2}{3}}(a)}\right) e\left(\xi f\left(\dfrac{qc}{U}\right) + \gamma \rho_s\left(\dfrac{cs}{U}\right)\right)\right|,
	\end{equation}
	respectively.
\end{prop}

We now prove the proposition. First note the
arithmetic component 
$e_s(\sigma h d\overline{\ell}^2)$
 in (\ref{moving the d-sum to the inside}) is linear in $d$ and it easily factorizes if we express $d = m + cs/U$ for $0 \leq m < s/U$, $0 \leq c < U$ with $U \mid s$ some power of $b$ (which we are free to choose). 

Consider now the analytic component $e(3T^{1/3}(Nn/2q)^{2/3})$ there. By the definitions in (\ref{T rewritten}) and (\ref{def of T(m,h)}), the phase equals
\begin{align*}
3T^{1/3}\left(\dfrac{Nn}{2q}\right)^{2/3} &=  \dfrac{3h^{1/3}n^{2/3}}{2^{2/3}q}\left(\rho\left(a + \dfrac{dq}{s}\right) +  a + \dfrac{dq}{s}\right)^{1/3}\\
&=
\dfrac{3h^{1/3}n^{2/3}}{2^{2/3}q}\left(\rho\left(a\right) + a + \rho\left(\dfrac{dq}{s}\right) +  \dfrac{dq}{s}\right)^{1/3}.
\end{align*}
The last holds by the fact $1\leq a < q/s$ and the additive property of $\rho$ in the absence of carry. Letting 
$$f(k) = \rho(k) + k$$
 and recalling $(a,b)=1$, Lemma \ref{lem:Reversal-type function} implies $f(a) \asymp X$ and $f(qd/s) \ll Xs/q$. Since $s\leq \sqrt{q}$ by assumption, the binomial theorem gives
$$
\left(f(a) + f\left(\dfrac{dq}{s}\right)\right)^{1/3} = f(a)^{1/3} + \dfrac{f(qd/s)}{3f(a)^{2/3}} -\dfrac{f(qd/s)^2}{9f(a)^{5/3}} + O\left(\dfrac{f(qd/s)^3}{f(a)^{8/3}} \ll \dfrac{X^{1/3}s^3}{q^3}\right).
$$
Since $f(qd/s) \ll Xs/q$, $0\leq d < s$ and $f(a) \asymp X$, we have
$$
f^2\left(\dfrac{qd}{s}\right) = \rho^2\left(\dfrac{qd}{s}\right) + O\left(Xs\right)
$$
and 
$$
\left(f(a) + f\left(\dfrac{dq}{s}\right)\right)^{1/3} = f(a)^{1/3} + \dfrac{f(qd/s)}{3f(a)^{2/3}} -\dfrac{\rho(qd/s)^2}{9f(a)^{5/3}} + O\left( \dfrac{X^{1/3}s^3}{q^3} + \dfrac{s}{X^{2/3}}\right).
$$
Since $h \asymp qr/M$, $n\asymp qr/N$ and $X \asymp MN^2$ by assumptions, it follows
$$
3T^{1/3}\left(\dfrac{Nn}{2q}\right)^{2/3} = \dfrac{3h^{1/3}n^{2/3}}{2^{2/3}q}\left(f(a)^{1/3} + \dfrac{f(qd/s)}{3f(a)^{2/3}} -\dfrac{\rho(qd/s)^2}{9f(a)^{5/3}}\right) + O\left(\dfrac{rs^3}{q^3} + \dfrac{rs}{X}\right).
$$
Since $s \leq \sqrt{q}$, $r\ll qX^\epsilon$ and $q \leq \sqrt{X}$ by assumption, the error term is $\ll X^{\epsilon}/\sqrt{q}$. Hence
$$
e\left(3T^{\frac{1}{3}}\left(\dfrac{Nn}{2q}\right)^{\frac{2}{3}}\right) = e\left(\dfrac{3h^{\frac{1}{3}}n^{\frac{2}{3}}}{2^{\frac{2}{3}}q}\left(f(a)^{1/3} + \dfrac{f(qd/s)}{3f(a)^{2/3}} -\dfrac{\rho(qd/s)^2}{9f(a)^{5/3}}\right)\right) + O\left(\dfrac{X^\epsilon}{\sqrt{q}}\right).
$$
Denoting by $\alpha(a,h,\ell,n)$ the absolute value of the $d$-sum in (\ref{moving the d-sum to the inside}) with $\mq =1$ there, it follows
\begin{align}\label{an approximation of alpha}
\alpha(a,h,\ell,n) &= \left|\sum_{0\leq d < s}F_h\left(\dfrac{Nn}{qT}\right)e_s\left(\sigma h d \overline{\ell}^2\right)e\left(\dfrac{h^{\frac{1}{3}}n^{\frac{2}{3}}}{2^{\frac{2}{3}}q}\left( \dfrac{f(qd/s)}{f(a)^{2/3}} -\dfrac{\rho(qd/s)^2}{3f(a)^{5/3}}\right)\right)\right|\\
&\qquad +O\left(\dfrac{sX^\epsilon}{\sqrt{q}}\right).\nonumber
\end{align}
Note that since $s\leq \sqrt{q}$, the error term is more than acceptable for our purposes. The term in the phase containing $f(qd/s)$ is nice since it is $d$-additive in the absence of carry (by the additive property of $\rho$). However we still need to handle the term with $\rho^2(qd/s)$. From previous facts and assumptions on the size of our variables,
$$
-\dfrac{h^{1/3}n^{2/3}\rho(qd/s)^2}{3\times 2^{2/3}qf(a)^{5/3}} \ll \dfrac{s^2 r}{q^2} \ll X^\epsilon. 
$$
If the arithmetic function $\rho$ were differentiable with suitable derivative bounds, we could easily remove this term from the phase by partial summation (at a negligible cost). It is not, but we can use an involutive property of $\rho$ to our advantage. Since $q/s$ is a power of $b$ by assumption, say $q/s = b^k$, the definition of $\rho$ implies
\begin{align*}
\rho\left(\dfrac{qd}{s}\right) &= X\sum_{j\geq 0}d_{j}\left(b^k d\right)b^{-j} = X\sum_{j\geq 0}d_{j-k}\left(d\right)b^{-j} = X\sum_{j\geq k}d_{j-k}\left(d\right)b^{-j}\\
&=
X\sum_{j\geq 0}d_{j}\left(d\right)b^{-j - k}\\
&=
\dfrac{bX\rho_s(d)}{q},
\end{align*}
where
$$
\rho_s(d) = \sum_{0\leq j < \log_b s} d_j(d)b^{\log_b(s) - 1 - j}.
$$
We note $\rho_s$ is an involution on $[0,s) \cap \mathbb{Z}$ and 
$$
\rho\left(\dfrac{q\rho_s(d)}{s}\right) = \dfrac{bXd}{q}.
$$
Thus if we substitute $d$ with $\rho_s(d)$, use the identity above and recall 
$$T = T\left(a + \dfrac{qd}{s},h\right),$$ 
(\ref{an approximation of alpha}) becomes
\begin{align*}
\alpha(a,h,\ell,n) = \left| \sum_{0 \leq d < s}\psi(a,\rho_s(d),h,\ell,n)e\left(-\dfrac{h^{1/3}n^{2/3}X^2b^2d^2}{2^{2/3}3q^3f^{5/3}(a)}\right)\right| + O\left(\dfrac{sX^\epsilon}{\sqrt{q}}\right),
\end{align*}
where
\begin{equation}\label{def of psi(a,d,h,ell,n)}
\psi(a,d,h,\ell,n) = F_h\left(\dfrac{Nn}{qT(a + qd/s, h)}\right)e_s\left(\sigma h d\overline{\ell}^2\right)e\left(\dfrac{h^{\frac{1}{3}}n^{\frac{2}{3}}f(qd/s)}{2^{\frac{2}{3}}q f^{\frac{2}{3}}(a)}\right).
\end{equation}
Applying Abel's summation formula, and differentiating the exponential involving $d^2$, gives
\begin{align*}
&\alpha(a,h,\ell,n) + O\left(\dfrac{sX^\epsilon}{\sqrt{q}}\right)\\
&\ll \left|\sum_{0 \leq d < s}\psi(a,\rho_s(d),h,\ell,n)\right| + \int_0^s \dfrac{h^{\frac{1}{3}}n^{\frac{2}{3}}X^2t}{q^3f^{\frac{5}{3}}(a)} \left|\sum_{0 \leq d < t}\psi(a,\rho_s(d),h,\ell,n)\right|\der t\\
&\ll 
\left|\sum_{0 \leq d < s}\psi(a,\rho_s(d),h,\ell,n)\right| + \dfrac{X^\epsilon}{s}\int_0^s  \left|\sum_{0 \leq d < t}\psi(a,\rho_s(d),h,\ell,n)\right|\der t.
\end{align*}
Note in the last line we used the assumptions $h \asymp qr/M$, $n \asymp qr/N$, $MN^2 = X$, $r \ll qX^\epsilon$ and $s \leq \sqrt{q}$. 

It may be beneficial to undo our earlier substitution of $d$ with $\rho_s(d)$. However the $d$-sum in the integral is incomplete. We may complete it to the original range $0 \leq d < s$ via a well-known trick: 

An application of the identity $$\mathbf{1}_{k=0} = \int_0^1e( k \gamma) \der \gamma
$$
valid for integers $k$, gives 
\begin{align*}
\sum_{0 \leq d < t}\psi(a,\rho_s(d),h,\ell,n) &= \sum_{0 \leq d < s}\psi(a,\rho_s(d),h,\ell,n) \sum_{\substack{0 \leq m < t \\ m = d}}1\\
 &= \int_0^1 \sum_{0 \leq m < t}e(-\gamma m) \sum_{0 \leq d < s}\psi(a,\rho_s(d),h,\ell,n)e(\gamma d)\der \gamma\\
 &=
 \int_0^1 \sum_{0 \leq m < t}e(-\gamma m) \sum_{0 \leq d < s}\psi(a,d,h,\ell,n)e(\gamma \rho_s(d))\der \gamma.
\end{align*}
In the last line we used the fact that $\rho_s$ is an involution on $[0,s) \cap \mathbb{Z}$. It follows
\begin{align}\label{bound for alpha in terms of psi}
\alpha(a,h,\ell,n)  &\ll \left|\sum_{0 \leq d < s}\psi(a,d,h,\ell,n)\right| +  \dfrac{sX^\epsilon}{\sqrt{q}} \nonumber\\
&\qquad + \dfrac{X^\epsilon}{s}\int_{0}^s \int_0^1\left|\sum_{0 \leq m < t}e(\gamma m)\right|\left|\sum_{0 \leq d < s}\psi(a,d,h,\ell,n)e(\gamma \rho_s(d))\right|\der \gamma \der t.
\end{align}

Let us now consider the factor
$$
F_h\left(\dfrac{Nn}{qT(a + qd/s,h)}\right)
$$
in the definition of $\psi(a,d,h,\ell,n)$ in (\ref{def of psi(a,d,h,ell,n)}). Recall $F_h(t)$ is smooth compactly supported on some positive interval $t\asymp 1$ satisfying $\|F_h^{(j)}\|_\infty \ll_j X^{\epsilon j}$ for each $j\geq 0$ (see (\ref{derivative of F_h})). Then so is the function $G_h(t) = F_h(1/t)$ and $\|G_h^{(j)}\|_\infty \ll_j X^{\epsilon j}$ for each $j\geq 0$. By definition and the Fourier inversion theorem, 
\begin{align*}
F_h\left(\dfrac{Nn}{qT(a + qd/s,h)}\right) &= G_h\left(\dfrac{qT(a + qd/s,h)}{Nn}\right)\\
& = \int_{\mathbb{R}}\widehat{G_h}(\xi)e\left(\dfrac{qT(a + qd/s,h) \xi}{Nn}\right) d\xi\\
&=
\int_{\mathbb{R}}\widehat{G_h}(\xi)e\left(\dfrac{h(f(a) + f(qd/s))\xi}{N^3n}\right) d\xi
\end{align*}
by the definition of $T(a + qd/s,h)$ in (\ref{def of T(m,h)}) and the additive property of $\rho$ and $f$. Substituting $h\xi/N^3n$ with $\xi$ gives
\begin{align*}
F_h\left(\dfrac{Nn}{qT(a + qd/s,h)}\right) = \dfrac{N^3 n}{h} \int_{\mathbb{R}}\widehat{G_h}\left(\dfrac{N^3n}{h}\xi\right)e\left(f(a)\xi + f\left(\dfrac{qd}{s}\right)\xi\right) d\xi.
\end{align*}
Since $h \asymp qr/M$, $n \asymp qr/N$ and $X = MN^2$ by assumptions, $N^3n/h \asymp X$. Then by the rapid decay of $\widehat{G_h}$ at $\pm \infty$,
\begin{align*}
F_h\left(\dfrac{Nn}{qT(a + qd/s,h)}\right) &= \dfrac{N^3 n}{h} \int_{\xi \ll X^{\epsilon + \delta}/X}\widehat{G_h}\left(\dfrac{N^3n}{h}\xi\right)e\left(f(a)\xi + f\left(\dfrac{qd}{s}\right)\xi\right) d\xi\\
&\qquad + O_{A,\delta}\left(X^{-A}\right)
\end{align*}
for any $A,\delta > 0$. By the definition of $\psi$ in (\ref{def of psi(a,d,h,ell,n)}), the bound in (\ref{bound for alpha in terms of psi}) and the size $N^3n/h \asymp X$, 
\begin{align*}
& \alpha(a,h,\ell,n)\\
& \ll_\delta 
X \int_{\xi \ll X^{\epsilon + \delta}/X} \left|\sum_{0 \leq d < s}e_s\left(\sigma h d\overline{\ell}^2\right)e\left(\dfrac{h^{\frac{1}{3}}n^{\frac{2}{3}}f(qd/s)}{2^{\frac{2}{3}}q f^{\frac{2}{3}}(a)}\right) e\left(\xi f\left(\dfrac{qd}{s}\right)\right)\right|d\xi + \dfrac{sX^\epsilon}{\sqrt{q}}\\
& \qquad +
\dfrac{X^{1 + \epsilon}}{s} \int_0^s\int_0^1\int_{\xi \ll X^{\epsilon + \delta}/X}\left|\sum_{0\leq m < t}e(\gamma m)\right|\\
&\qquad 
\times \left|\sum_{0 \leq d < s}e_s\left(\sigma h d\overline{\ell}^2\right)e\left(\dfrac{h^{\frac{1}{3}}n^{\frac{2}{3}}f(qd/s)}{2^{\frac{2}{3}}q f^{\frac{2}{3}}(a)}\right) e\left(\xi f\left(\dfrac{qd}{s}\right) + \gamma \rho_s(d)\right)\right|\der \xi \der \gamma \der t.
\end{align*}
Switching orders of summation/integration, taking suprema and using the fact 
$$
\sup_{0\leq t < s}\int_{0}^1\left| \sum_{0 \leq m < t}e(\gamma m)\right|\der \gamma 
\ll \int_{0}^{1/2}\min\left(s, \dfrac{1}{\gamma}\right)\der \gamma\ll \log s \ll \log X,
$$
it follows
\begin{align*}
\sum_{\substack{1 \leq a \leq q/s \\ 1 \leq h \asymp qr/M \\ \ell(\sqrt{q}) \\ 1 \leq n \asymp qr/N\\(a\ell,q)=1 \\ n\ell^3 \equiv -2\sigma a h (\sqrt{q})}} \alpha(a,h,\ell,n) 
&\ll X^{3\epsilon}
\sup_{\substack{\gamma \in [0,1] \\ \xi \ll X^{2\epsilon}/X}} \sum_{\substack{1 \leq a \leq q/s \\ 1 \leq h \asymp qr/M \\ \ell(\sqrt{q}) \\ 1 \leq n \asymp qr/N\\(a\ell,q)=1 \\ n\ell^3 \equiv -2\sigma a h (\sqrt{q})}}\widetilde{\alpha}(a,h,\ell,n;\gamma,\xi) + E,
\end{align*}
where
$$
\widetilde{\alpha}(a,h,\ell,n;\gamma,\xi) = \left|\sum_{0 \leq d < s}e_s\left(\sigma h d\overline{\ell}^2\right)e\left(\dfrac{h^{\frac{1}{3}}n^{\frac{2}{3}}f(qd/s)}{2^{\frac{2}{3}}q f^{\frac{2}{3}}(a)}\right) e\left(\xi f\left(\dfrac{qd}{s}\right) + \gamma \rho_s(d)\right)\right|
$$
and
\begin{align*}
E &= \dfrac{sX^\epsilon }{\sqrt{q}}\sum_{\substack{1 \leq a \leq q/s \\ 1 \leq h \asymp qr/M \\ \ell(\sqrt{q}) \\ 1 \leq n \asymp qr/N\\(a\ell,q)=1 \\ n\ell^3 \equiv -2\sigma a h (\sqrt{q})}}1 \ll
\dfrac{sX^{\epsilon + o(1)}\mathbf{1}_{qr \gg M}}{\sqrt{q}}\sum_{\substack{ \ell(\sqrt{q}) \\ 1 \leq n \asymp qr/N }} \sum_{\substack{k \ll q^2r/sM
		\\k \equiv n\ell^3 (\sqrt{q})} }  1\\
&\ll \dfrac{q^{5/2} r^2 X^{\epsilon + o(1)}}{MN}.
\end{align*}

We can now use the additive property (in the absence of carry) to factorize the $d$-sum as a product of two large sums. Writing $d = m + cs/U$ with $0 \leq m < s/U$, $0 \leq c < U$ for some power $U \mid s$ of $b$ (which we are free to choose) we have
$$
\widetilde{\alpha}(a,h,\ell,n;\gamma,\xi) = \mathscr{U}_U(a,h,\ell,n;\gamma,\xi)\mathscr{V}_U(a,h,\ell,n;\gamma,\xi),
$$
where $\mathscr{U}_U(a,h,\ell,n;\gamma,\xi),\mathscr{V}_U(a,h,\ell,n;\gamma,\xi)$ are defined in (\ref{Usum}), (\ref{Vsum}). The proof is concluded after multiplying both sides by the normalizing factor $MN/(qr)^{3/2}$.

\begin{rmk}\label{Remark on factorization consequences}
Applying Cauchy-Schwarz on the sum in \textnormal{(\ref{factoring_d-sum})} and proceeding with arguments in the style of those in the proof of \textnormal{Proposition \ref{prop: N large for squares n^2 with n of size N}}, it seems it should be feasible (unverified personal notes) to show that the sum in \textnormal{(\ref{factoring_d-sum})} is (after optimizing the choice of $U$) essentially
$$
\ll \left(\textnormal{Trivial Bound}\right)\times \left(\dfrac{1}{\sqrt{s}} + \dfrac{1}{(qr/N)^{3/8}}\right).
$$
Similar arguments can and should be applied to one of the sums in the upcoming section (after a similar $B$-process, the analogue of the $d$-sum there is essentially a shifted version of the $d$-sum here). Further improvements are expected. For instance, we have not exploited cancellation coming from the $h$-sum. It should also be possible to choose $s$ larger than $\sqrt{q}$ in order to improve the diagonal portion after Cauchy-Schwarz. Here one may have to describe $\K_2$ in terms of cubic roots and work over $p$-adics on the arithmetic side. See for instance \textnormal{Mangerel} \cite{Mangerel} and the related content in \textnormal{Mili\'cevi\'c} \cite{Milicevic(2015)},  \textnormal{Mili\'cevi\'c-Zhang} \cite{Milicevic-Zhang(2023)}. 
\end{rmk}

\section{Large conductors relative to $N$}\label{sec:N small relative to q}

Here we establish the following bound, primarily intended for a regime with $q,\mq$ rather large relative to $N$. In this regime, the total conductor of the harmonics is too large, rendering completion of the $n$-sum unsatisfactory. However, as discussed in the introduction, we can first apply an $A$-type process to the $m$-sum involving the reversal function, thereby lowering the total conductor to manageable levels where completion methods ($B$-process) start to become effective. After applying Cauchy-Schwarz and expanding the squares, it should be possible here to adapt the arguments, of the previous section, in order to improve upper bounds. Nevertheless the upper bound below suffices for our purposes.

\begin{prop}\label{prop:first estimate for E_>(r)}
With the same assumptions of \textnormal{Proposition \ref{prop: main term for smooth S}} and $X^{2\epsilon} \ll r \ll qX^\epsilon$, 
	$$
	E_>(r) \ll X^{4\epsilon} \left(qN^{5/8} + \sqrt{q} N \right) \left(1 + \dfrac{q}{r}\right).
	$$
	\end{prop}

We now proceed to prove the proposition.

Moving the $m$-sum to the inside, taking absolute values and noting $f_h(t) \ll \eta(t)$, we have
$$
E_>(r) \ll 
\dfrac{M}{q\mq r}\sum_{h \asymp q\mq r/M}\sum_{\substack{(n,q)=1 }}\eta\left(\dfrac{n}{N}\right) \left|\sum_{\substack{1 \leq m \leq q \\ (m,q)=1}} e_q\left(hm\overline{\mq n^2}\right)e\left( - \dfrac{h\rho(m) + hm}{q\mq n^2}\right)\right|.
$$

Recall Property (D) of Lemma \ref{lem:Reversal-type function}, which states that $\rho$ satisfies the additive property 
$$\rho(m + n) = \rho(m) + \rho(n)$$ 
for integers $0 \leq m,n < bX$ with no carry in their $b$-adic addition. Let us now introduce two integer variables $\iota, \kappa$, to be chosen later, satisfying $1 \leq \iota < \kappa < \lambda$. Writing $$m = a + b^\iota c + b^\kappa d$$ 
with $1 \leq a < b^\iota$, $(a,b)=1$, $0 \leq c < b^{\kappa - \iota}$, $0 \leq d < b^{\lambda - \kappa}$, and then using the additive property of $\rho$ and the exponentials, we can rewrite the sum as
$$
\dfrac{M}{q\mq r}\sum_{h \asymp q\mq r/M}\sum_{\substack{(n,q)=1 }}\eta\left(\dfrac{n}{N}\right)\left|\Sigma_{3}\Sigma_{2}\Sigma_{1}(n,h)\right|,
$$
where
\begin{align*}
\Sigma_{3}(n,h) &= \sum_{\substack{1 \leq a < b^\iota \\ (a,b)=1}} e_q\left(ha\overline{\mq n^2}\right)e\left( - \dfrac{h\rho(a) + ha}{q\mq n^2}\right),\\
\Sigma_{2}(n,h) &= \sum_{0 \leq c < b^{\kappa-\iota}} e_q\left(hcb^{\iota}\overline{\mq n^2}\right)e\left( - \dfrac{h\rho(cb^{\iota}) + hcb^\iota}{q\mq n^2}\right),\\
\Sigma_{1}(n,h) 
&=\sum_{0 \leq d < b^{\lambda - \kappa}} e_q\left(hdb^{\kappa}\overline{\mq n^2}\right)e\left(- \dfrac{h\rho(db^{\kappa}) + hdb^\kappa}{q\mq n^2}\right).
\end{align*}
Bounding $\Sigma_3$ trivially by $b^{\iota}$ and applying Cauchy-Schwarz, it follows
$$
E_>(r) \ll  b^{\iota}\sqrt{\mathscr{Z}_1\mathscr{Z}_2},
$$
where
\begin{align*}
\mathscr{Z}_j &= \dfrac{M}{q\mq r}\sum_{h \asymp q\mq r/M}\sum_{\substack{(n,q)=1 }}\eta(n/N)\left|\Sigma_j(n,h)\right|^2
\end{align*}
for $j \in \{1,2\}$.

For an integer $s \mid q$, define
\begin{equation}\label{eqn: def of rho_ell}
\rho_s(d) = \rho\left(\dfrac{q}{s} d\right).
\end{equation}
Substituting $q/b^\iota$ with $s$ in the first sum above, or substituting $q/b^\kappa$ with $s$ in the second sum, 
we note each of the $\Sigma_j$ sums above may be written as
$$
\Sigma(n,h,D,s) = \sum_{0 \leq d < D} e_s\left(h\overline{\mq n^2} d\right)e\left(- \dfrac{h \rho_s(d)}{q\mq n^2} - \dfrac{h  d}{s\mq n^2}  \right)
$$
for some powers $b \leq D \leq s <  q$ of $b$. Since the restriction $(n,q)=1$ is equivalent to $(n,s)=1$ for proper powers $q,s$ of $b$, it then suffices to consider a generic sum such as
\begin{equation}\label{generic}
\mathscr{Z}(D,s) = \dfrac{M}{q\mq r}\sum_{h \asymp q\mq r/M}\sum_{(n,s)=1}\eta(n/N)\left|\Sigma(n,h,D,s)\right|^2.
\end{equation}
With our notation,
\begin{equation}\label{eqn: bound for F in terms of Z}
E_>(r) \ll b^\iota  \sqrt{\mathscr{Z}\left(b^{\lambda-\kappa},b^{\lambda-\kappa}\right)
	\mathscr{Z}\left(b^{\kappa-\iota},b^{\lambda-\iota}\right)}.
\end{equation}

It is convenient to remove the term 
$- hd/s\mq n^2$ from the phase above. To this end, we note that for $1 \leq D \leq s$, $h \asymp q\mq r/M $, $n \asymp N$ and our assumptions $MN^2 = X$, $q \leq \sqrt{X}$, $r \ll qX^\epsilon$, we have
$$
\dfrac{h  }{s\mq n^2} \ll \dfrac{X^\epsilon}{s} \leq \dfrac{X^\epsilon}{D}.
$$ 
Letting 
\begin{equation}\label{Sigma}
\widetilde{\Sigma}(n,h,t,s) = \sum_{0 \leq d < t} e_s\left(h\overline{\mq n^2} d\right)e\left(- \dfrac{h \rho_s(d)}{q\mq n^2}  \right),
\end{equation}
Abel's summation formula and Cauchy-Schwarz yield
\begin{align*}
\Sigma(n,h,D,s) &\ll \left|\widetilde{\Sigma}(n,h,D,s)\right| + \dfrac{X^\epsilon}{D} \int_1^D\left|\widetilde{\Sigma}(n,h,t,s)\right| dt\\
&\leq
\left|\widetilde{\Sigma}(n,h,D,s)\right| + \left(\dfrac{X^{2\epsilon}}{D}\int_1^D\left|\widetilde{\Sigma}(n,h,t,s)\right|^2 dt\right)^{1/2}.
\end{align*}
Since $|A + B|^2 \ll |A|^2 + |B|^2$ for any $A,B$, it follows
\begin{align}
\mathscr{Z}(D,s)  
&\ll X^{2\epsilon}\sup_{1 \leq t \leq D}\widetilde{\mathscr{Z}}(t,s), \label{eqn: tilde Z}
\end{align}
where
\begin{equation}\label{tilde Z}
\widetilde{\mathscr{Z}}(t,s) = \dfrac{M}{q\mq r}\sum_{h\asymp q\mq r/M}\sum_{\substack{(n,s)=1}}\eta(n/N)\left|\widetilde{\Sigma}(n,h,t,s)\right|^2.
\end{equation}

We turn our attention to $\widetilde{\mathscr{Z}}(t,s)$.
Expanding the square via Weyl-differencing and switching orders of summation, the above becomes
\begin{align*}
\widetilde{\mathscr{Z}}(t,s) 
&=
\sum_{\substack{0\leq d',d+d' < t}}\dfrac{M}{q\mq r}\sum_{h \asymp  q\mq r/M}\sum_{\substack{(n,s)=1 }}\eta(n/N)e_s\left(dh\overline{\mq n^2}\right)e\left(-\dfrac{h\Delta_{d,s}(d')}{q\mq n^2}\right),
\end{align*}
where
$$
\Delta_{d,s}(d') = \rho_s(d'+d)-\rho_s(d').
$$

By Property (B) of Lemma \ref{lem:Reversal-type function}, for any integer $0 \leq d \leq t\leq D$, we have that $\rho_s(d)$ (which we recall was defined as $\rho_s(d) = \rho(qd/s)$) is divisible by $Xs/qD$ (this is an integer as $D\mid s$ and $q \mid X$ by assumption; recall also that $b \leq D \leq s < q$ are powers of $b$). Hence 
\begin{equation}\label{eqn: Delta is divisible by large integer}
\dfrac{Xs}{qD} \mid \Delta_{d,s}(d').
\end{equation}
Then given any interval $I$ of $\mathbb{R}$ of finite length $|I|$, there are at most 
$$\ll \dfrac{qD|I|}{Xs} + 1$$
possible values in $I$ that $\Delta_{d,s}(d') $ may attain. For any one such of these values, say $m$, 
note
$$
\sum_{\substack{0\leq d',d+d' < t \\ \Delta_{d,s}(d') = m}}1 = \sum_{0 \leq d' < t} \sum_{\substack{-d' \leq d < t-d' \\ \rho_s(d + d') = m + \rho_s(d')}}1 \leq \sum_{0 \leq d' < t}1 \leq t.
$$
The second to last inequality holds since $\rho$ is injective on $[0, bX) \cap \mathbb{Z}$. In particular, for any given $\delta > 0$ small, the terms with $\Delta_{d,s}(d') \ll X^{1+\delta}/r$ contribute 
$$
\ll \left(\dfrac{X^{\delta}qD}{rs} + 1\right)DN
$$
to $\widetilde{\mathscr{Z}}(t,s) $. 
Thus
\begin{align}\label{asymptotic for Z}
\widetilde{\mathscr{Z}}(t,s) 
&=
\sum_{\substack{0\leq d',d+d' < t \\ r\Delta_{d,s}(d') \gg X^{1+\delta}}}\dfrac{M}{q\mq r}\sum_{h \asymp q\mq r/M} \mathscr{S}_\eta \left(dh\overline{\mq},N,s, -\dfrac{h\Delta_{d,s}(d')}{q\mq N^2}\right)\\
&\qquad
+O\left(\dfrac{X^{\delta}qD^2N}{rs} + DN \right),\nonumber
\end{align}
where
\begin{equation}\label{def of mathscr S}
\mathscr{S}_\eta(k,N,s,T) = \sum_{(n,s)=1}\eta(n/N)e_s\left(k\overline{n}^2\right)e\left(T\dfrac{N^2}{n^2}\right) .
\end{equation}
The diagonal term, with $d =0$, does not appear in the sum since $\Delta_{0,s}=0$. 

Note for any such $d,d',h$ obeying the restrictions under the sums and with the assumption $X = MN^2$,
\begin{equation}\label{eqn: lower bound for amplitude}
\dfrac{h\Delta_{d,s}(d')}{q\mq N^2} \gg X^\delta. 
\end{equation}
By Property (C) of Lemma \ref{lem:Reversal-type function}, $\rho_s(d) := \rho(qd/s) < bXs/q$ for any $1 \leq d \leq s$ and thus $$\Delta_{d,s}(d') \ll \dfrac{Xs}{q}. $$ Hence we also have
\begin{equation}\label{eqn: upper bound for amplitude}
\dfrac{h\Delta_{d,s}(d')}{q\mq N^2} \ll \dfrac{rs}{q} \ll sX^\epsilon. 
\end{equation}
The last inequality holds since $r \ll qX^\epsilon$ by assumption. Then by (\ref{eqn: lower bound for amplitude}) and (\ref{eqn: upper bound for amplitude}), we are thus in the situation of Lemma \ref{lem:twisted sums with squared inverses} with $X^\delta \ll T \ll sX^\epsilon$, $\|\eta^{(j)}\|_\infty \ll_j 1$ for each $j\geq 0$ (by assumption) and $r = 1$. Conjugating if necessary so that $T > 0$, Lemmas \ref{lem:twisted sums with squared inverses}, \ref{lem: Shparlinski bound} and \ref{lem:sums of GCDs} imply
\begin{align*}
\mathscr{S}_\eta(k,N,s,T) &\ll \dfrac{NX^{\delta/2}}{\sqrt{sT}} \sum_{1 \leq n \asymp sT/N} \sqrt{(k,n,q)} + X^{\delta} \sqrt{\dfrac{s}{T}}\\
&\ll \sqrt{sT}X^{\delta}\\
&\ll s X^{\epsilon + \delta} 
\end{align*}
for $T \ll s X^{\epsilon}$. Inserting the bound in (\ref{asymptotic for Z}) (after conjugating if necessary on some terms so that $h \Delta_{d,s}(d') > 0$) and using the bound in (\ref{eqn: upper bound for amplitude}) gives
\begin{align}
\widetilde{\mathscr{Z}}(t,s) &\ll DN + \dfrac{X^{\delta}qD^2N}{rs} +  X^{\epsilon+\delta} sD^2 \nonumber \\
&\leq  DN + \dfrac{X^{\delta}qDN}{r} +  X^{\epsilon+\delta} sD^2 \label{first bound for Z tilde}
\end{align}
for $t \leq D \leq s$ and any fixed $\delta > 0$.
We choose $\delta = \epsilon$. Recalling (\ref{eqn: tilde Z}) and inserting this in (\ref{eqn: bound for F in terms of Z}) yields (after moving the $b^\iota$ factor, on the outside there, inside of a square root)
\begin{align}\label{eqn: square roots}
E_>(r)X^{-4\epsilon} 
&\ll 
\sqrt{b^{\lambda - \kappa}N + \dfrac{qb^{\lambda - \kappa}N}{r} +   b^{3(\lambda - \kappa)} }
\sqrt{b^{\kappa + \iota}N + \dfrac{qb^{\kappa + \iota}N}{r} +  b^{\lambda - \iota + 2\kappa } } \nonumber\\
&= 
\sqrt{b^{\lambda }N + \dfrac{qb^{\lambda }N}{r} +   b^{3\lambda - 2\kappa} }
\sqrt{b^{\iota}N + \dfrac{qb^{\iota}N}{r} +  b^{\lambda - \iota + \kappa } } \nonumber\\
&=
\sqrt{qN + \dfrac{q^2N}{r} +  \dfrac{q^3}{V^2}  }
\sqrt{UN + \dfrac{qUN}{r} + \dfrac{q  V}{U} } \nonumber
\end{align}
for any $1 \leq \iota < \kappa < \lambda$, where $U = b^\iota$, $V = b^\kappa$, and (as we have assumed thus far) $q = b^\lambda$. 
We choose $U,V$ (or equivalently $\iota, \kappa$) as follows:
$$
(U,V)\asymp
\begin{cases}
(q/N^{3/4}, q/\sqrt{N}) &\mbox{ if } q \gg N^{3/4}\\
(1, q/\sqrt{N})&\mbox{ if } \sqrt{N} \ll q \ll N^{3/4}\\
(1,1) &\mbox{ if } q \ll \sqrt{N}.
\end{cases}
$$
The notation $(U,V) \asymp (A, B)$ means that $U \asymp A$ and $V \asymp B$. In other words, we choose $\iota,\kappa$ so that
\begin{align*}
U \asymp \max\left(1, \dfrac{q}{N^{3/4}}\right), \hspace{2em}
V \asymp \max\left(1, \dfrac{q}{\sqrt{N}}\right).
\end{align*}
With this choice, the product of the square roots above is
\begin{align*}
&\ll\sqrt{qN\left(1 + \dfrac{q}{r}\right) + \dfrac{q^3}{(q/\sqrt{N})^2}}\sqrt{\left(1 + \dfrac{q}{N^{3/4}}\right)N\left(1 + \dfrac{q}{r}\right) + q\dfrac{1 + q/\sqrt{N}}{q/N^{3/4}}}
\\
&\asymp \sqrt{qN\left(1 + \dfrac{q}{r}\right)} \sqrt{\left(N + qN^{1/4}\right)\left(1 + \dfrac{q}{r}\right)}\\
&\asymp \left(qN^{5/8} + \sqrt{q}N\right)\left(1 + \dfrac{q}{r}\right). 
\end{align*}
Thus
$$
E_>(r) \ll  X^{4\epsilon}\left(qN^{5/8} + \sqrt{q}N\right)\left(1 + \dfrac{q}{r}\right).
$$
This establishes Proposition \ref{prop:first estimate for E_>(r)}.

\section{Proof of Theorem \ref{thm: square divisors}}\label{sec:Proof of Theorem {thm:square divisors}}

	Combining the two bounds in Propositions \ref{prop: N large for squares n^2 with n of size N}	and \ref{prop:first estimate for E_>(r)} gives, for any $X^{2\epsilon} \ll r\ll qX^{\epsilon}$, 
	\begin{align*}
	\dfrac{rE_{>}(r)X^{-5\epsilon}}{q} &\ll \min\left(qN^{5/8} + \sqrt{q}N, \ q^{3/2}\mq^{1/4}N^{1/4} + q^{7/4}\mq^{1/2} + q \sqrt{N}\right)\\
	&\ll
	qN^{5/8} + \min\left(\sqrt{q} N , \ q^{3/2}\mq^{1/4}N^{1/4} + q^{7/4}\mq^{1/2}\right)\\
	&\asymp
	qN^{5/8} + q\min\left(\dfrac{N}{\sqrt{q}}, \ \max\left(\sqrt{q}\mq^{1/4}N^{1/4}, \ q^{3/4}\mq^{1/2}\right)\right).
	\end{align*}
	We have
	$$
	\min\left(\dfrac{N}{\sqrt{q}}, \ \sqrt{q}\mq^{1/4}N^{1/4}\right) \leq N^{5/8} \mq^{1/8}
	$$
	and
	$$
	\min\left(\dfrac{N}{\sqrt{q}}, \ q^{3/4}\mq^{1/2}\right) \leq N^{3/5}\mq^{1/5}. 
	$$
	Then
	$$
	\min\left(\dfrac{N}{\sqrt{q}}, \ \max\left(\sqrt{q}\mq^{1/4}N^{1/4}, \ q^{3/4}\mq^{1/2}\right)\right) \leq N^{5/8}\mq^{1/8} + N^{3/5}\mq^{1/5} 
	$$
	and
	$$
	\dfrac{rE_{>}(r)X^{-5\epsilon}}{q} \ll q\left(N^{5/8}\mq^{1/8} + N^{3/5}\mq^{1/5} \right).
	$$	
	Proposition \ref{prop: main term for smooth S} now gives
	\begin{align}
	&S_{\eta,\theta}(X,N,q,\mq,\ma) = \dfrac{ \widehat{\eta_2}(0)\widehat{\theta}(0)\varphi^2(b)\varphi(\mq)}{b^2\mq} \dfrac{MN}{q\mq}\\
	&\qquad + X^{\epsilon}O\left( \sqrt{\mq} + \sqrt{N} + \min\left(\dfrac{N}{q\sqrt{\mq}}, \dfrac{MN}{q\mq}\right) + \dfrac{M}{q\mq} + qN^{5/8}\mq^{1/8} + qN^{3/5}\mq^{1/5} \right) \nonumber\\
	&\ll \dfrac{MNX^{\epsilon}}{q\mq} + \left(\sqrt{\mq} +qN^{5/8}\mq^{1/8} + qN^{3/5}\mq^{1/5} \right)X^{\epsilon} \nonumber\\
	&\ll \dfrac{X^{1 + \epsilon}}{q\mq N} + q\left( N^{5/8}\mq^{1/8} + N^{3/5}\mq^{1/5} \right)X^{\epsilon} + \sqrt{\mq}X^{\epsilon}\label{a bound for smooth S}
	\end{align}
	for any fixed $\epsilon > 0$. The last holds since  $M = X/N^2$ by assumption. 
	
	If 
	$
	X/\mq N \ll N^{5/8}\mq^{1/8} + N^{3/5}\mq^{1/5},
	$
	we may choose $q = b^{2}$ (recalling $q = b^\lambda$ by assumption, equivalently $\lambda = 2$). This gives
	$$
	S_{\eta,\theta}(X,N,q,\mq,\ma)X^{-\epsilon} \ll  N^{5/8}\mq^{1/8} + N^{3/5}\mq^{1/5} + \sqrt{\mq}
	$$
	for any fixed $\epsilon > 0$. On the other hand, if $
	X/\mq N \gg N^{5/8}\mq^{1/8} + N^{3/5}\mq^{1/5},
	$ we may choose $\lambda > 0$ even so that
	$$
	q \asymp \sqrt{\dfrac{X/\mq N}{N^{5/8}\mq^{1/8} + N^{3/5}\mq^{1/5}}}.
	$$
	In this case, (\ref{a bound for smooth S}) gives
	$$
	S_{\eta,\theta}(X,N,q,\mq,\ma)X^{-\epsilon} \ll \dfrac{\sqrt{X}}{\mq^{7/16}N^{3/16}} + \dfrac{\sqrt{X}}{\mq^{2/5}N^{1/5}} + \sqrt{\mq}. 
	$$
	Regardless of the case,
	$$
	S_{\eta,\theta}(X,N,q,\mq,\ma)X^{-\epsilon} \ll \dfrac{\sqrt{X}}{\mq^{7/16}N^{3/16}} + \dfrac{\sqrt{X}}{\mq^{2/5}N^{1/5}} + \mq^{1/8}N^{5/8} + \mq^{1/5}N^{3/5} + \sqrt{\mq}.
	$$
	We can rewrite the above as 
	\begin{align*}
	S_{\eta,\theta} 
	\ll
	\dfrac{\sqrt{X}^{1 + \epsilon}}{\sqrt{\mq}}\left(\left(\dfrac{\mq^{1/2}}{N}\right)^{1/5} + \left(\dfrac{\mq^{1/3}}{N}\right)^{3/16} + \dfrac{N^{3/5}\mq^{7/10} + (N\mq)^{5/8} + \mq}{\sqrt{X}}  \right)
	\end{align*}
	with the shorthand notation $S_{\eta,\theta} = S_{\eta,\theta}(X,N,q,\mq,\ma)$. Now (\ref{square divisors of Pi(L,q,a)}) follows from (\ref{eqn:smoothing S }) after substituting the notations $\mq,\ma$ here with $q,a$ and using $|\Pi_b(L)| \asymp_b b^{L/2} = \sqrt{X}$. The inequality in (\ref{square divisors of Pi}) follows after we set $\mq = \ma = 1$ here and note $N^{5/8}/\sqrt{X} \ll N^{-3/16}$ (since $N\ll \sqrt{X}$ by assumption).

\section{Medium and small square divisors}\label{sec:medium square divisors}

Theorem \ref{thm: square divisors} is most effective when the square divisors $d^2$ there (after replacing the notation $n$ with $d$) are sufficiently large relative to $q$. In particular, letting $X = b^{L}$, it controls moduli $q \ll X^{1/16-\epsilon}$ as soon as $d \gg X^{3/16}$. On the other hand, since palindromes are known to have level of distribution $X^{1/5 }$ \cite{T-Panario} we then also have control of the cases when $d \ll \sqrt{X^{1/5-\epsilon}/X^{1/16-\epsilon}} \approx X^{0.07}$. 
In order to handle the medium cases, we adapt some of the arguments in \cite{T-Panario} while using the estimate of Baier-Zhao \cite{Baier-Zhao} (see also Baker \cite{Baker(2017)}) for the large sieve with square moduli (in Lemma \ref{lem:Large sieve with square moduli} here). We also employ the $L^1$ and moment bounds in \cite{T-Panario} (see Lemmas \ref{lem:T-Panario L^1 bound} and \ref{lem:Moment bounds} below). The main conclusion of this section is the following proposition.

\begin{prop}\label{prop:equidistribution involving squares}
For any $D,Q,x\geq 1$ and any fixed $\epsilon > 0$, we have the following two results.

If $DQ \ll x^{1/5-\epsilon}$, then
\begin{equation}\label{DQ <= x^{1/5}}
\sum_{\substack{d \sim D \\ q \sim Q \\ (dq,b^3-b)=1}}\sup_{y \leq x}\max_{a \in \mathbb{Z}}\left|\sum_{n\in \mathscr{P}_b^*(y)}\left(\mathbf{1}_{n\equiv a (qd^2)} - \dfrac{1}{qd^2}\right)\right|
\ll \sqrt{x}\exp\left(-\sigma \sqrt{\log x}\right)
\end{equation}
for some $\sigma > 0$ depending only on $b,\epsilon$.

If $DQ \ll x^{1/4-\epsilon}$ and $Q \ll x^{1/15-\epsilon}$, then
\begin{equation}\label{DQ <= x^{1/4} and Q <= x^{1/15}}
\sum_{\substack{d \sim D \\ q \sim Q \\ (dq,b^3-b)=1}}\sup_{y \leq x}\max_{a \in \mathbb{Z}}\left|\sum_{n\in \mathscr{P}_b^*(y)}\left(\mathbf{1}_{n\equiv a (qd^2)} - \dfrac{1}{qd^2}\right)\right|
\ll \sqrt{x}\exp\left(-\sigma \sqrt{\log x}\right)
\end{equation}
for some $\sigma > 0$ depending only on $b,\epsilon$.
\end{prop}

We need a few lemmas. 
First for a real variable $\alpha$ and integer $N\geq 0$, define 
\begin{equation}\label{little phi}
\phi_b(\alpha) = \left|\sum_{0 \leq m < b}e(\alpha m)\right|
\end{equation}
and
\begin{equation}\label{Big Phi}
\Phi_N\left(\alpha\right) = \prod_{1 \leq n < N}\phi_b\left(\alpha\left(b^n + b^{2N -n}\right)\right)
\end{equation}
with the convention $\Phi_N\left(\alpha\right) = 1$ if $N\leq 1$.
Note from the definition of $\Pi_b(2N)$ and the additivity of exponentials, 
\begin{equation}\label{sum-to-product}
\left|\sum_{n \in \Pi_b(2N)}e(\alpha n) \right| \leq b^2\Phi_N(\alpha).
\end{equation}
We also have the following lemma. First define the set $\mathscr{P}_b^0(x)$ by 
\begin{equation}\label{P^0(x)}
\mathscr{P}_b^0(x) = \left\{n \in \mathscr{P}_b(x) \ : \ \lfloor \log_b n\rfloor \equiv 0 (2)\right\}.
\end{equation}

\begin{lem}[Incomplete sums]\label{lem:Sum-to-product}
Let $x\geq 1$ and $\alpha$ real. Then
$$
\left|\sum_{n \in \mathscr{P}_b^0(x)}e(\alpha n)\right| \leq b^2\sum_{0 \leq N \leq \frac{1}{2}\log_b x}\sum_{0 \leq M \leq N}\Phi_M\left(\alpha b^{N-M}\right).
$$	
\end{lem}

\begin{proof}
This is implicit in the earlier works of  Banks-Hart-Sakata \cite{Banks-Hart-Sakata(2004)} and Col \cite{Col(2009)}. See for instance Lemma 6.1 in \cite{T-Panario} for a proof. 	
\end{proof}

\begin{lem}[$L^1$-bound]\label{lem:T-Panario L^1 bound}
Let $N,Q \geq 1$. Then
$$
\max_{k\in \mathbb{Z}}\sum_{\substack{1 < q \leq Q \\ (q,b^3-b)=1}}\sum_{h(q)}^*\Phi_N\left(\dfrac{h}{q} + \dfrac{k}{b^3-b}\right)\ll_\epsilon Q^2b^{3N/5 + \epsilon N} + b^N Q^{1 - \sigma_1}\exp\left(-\dfrac{\sigma_\infty N}{\log Q}\right)
$$
for any $\epsilon > 0$,
where $\sigma_1, \sigma_\infty > 0$ are some values depending only on $b,\epsilon$. 
\end{lem}

\begin{proof}
See Proposition 8.1 in \cite{T-Panario}.
\end{proof}

\begin{lem}[Moment bounds]\label{lem:Moment bounds}
For any integers $N,K\geq 2$,
$$
\int_0^1\Phi_N^{2K}(\alpha)\der \alpha \leq b^{2(K-1)N + 2}\left(1 + O\left(\dfrac{1}{\sqrt{K}} + \dfrac{b^2}{K}\right)\right)^{2N}.
$$
The implied constant in the error term is absolute (and does not depend on $b$). 
\end{lem}

\begin{proof}
See Proposition 7.1 in \cite{T-Panario}. In a recent preprint, Dartyge-Rivat-Swaenepoel \cite{Dartyge-Rivat-Swaenepoel(2025)} improved it (see Proposition 7.1 there) by saving a factor of $b^{2K}$ and gave explicit constants for the error term. Their bounds hold for any $K \in \mathbb{R^+}$. For our applications here, $K$ is bounded in terms of $b$.
\end{proof}	

\begin{lem}[Moments involving squares]\label{lem:Moments involving square moduli}
Let $D,q \geq 1$ with $q \in \mathbb{N}$ and let $N,K \geq 2$. Then
\begin{align*}
&\sup_{\beta \in \mathbb{R}}\sum_{d\leq D}\sum_{h(qd^2)}^*\Phi_N^{2K}\left(\dfrac{h}{qd^2} + \beta\right)\\
& \ll_{\epsilon} \left(qD^3 + Kb^{2N}\sqrt{D}\right)b^{2(K-1)N}\left(1 + \dfrac{q}{Kb^{2N}}\right)\left(1 + O\left(\dfrac{1}{\sqrt{K}} + \dfrac{b^2}{K}\right)\right)^{2N}\left(b^NDK\right)^{\epsilon}
\end{align*}	
for any $\epsilon > 0$. 
\end{lem}

\begin{proof}
Let $\alpha$ be a real variable. Recalling the definition of $\Phi_N$ in (\ref{Big Phi}), one observes
$$
\Phi_N^K(\alpha + \beta) = \left|\sum_{0 \leq n \leq Kb^{2N}}\gamma_n e(\alpha n)\right|
$$	
for some complex numbers $\gamma_n$ independent of $\alpha$. Then by Lemma \ref{lem:Large sieve with square moduli}, 
\begin{align*}
&\sum_{d\leq D}\sum_{h(qd^2)}^*\Phi_N^{2K}\left(\dfrac{h}{qd^2} + \beta\right)\\ &\ll_{\epsilon}\left(b^NDK\right)^{\epsilon}\left(qD^3 + Kb^{2N}\sqrt{D}\right)\left(1 + \dfrac{q}{Kb^{2N}}\right)\sum_{0 \leq n \leq Kb^{2N}}|\gamma_n|^2
\end{align*}
for any $\epsilon > 0$. 
By Parseval's identity and the $1$-periodicity of $\Phi_N$,
\begin{align*}
\sum_{0 \leq n \leq Kb^{2N}}|\gamma_n|^2 &= \int_0^1\left|\sum_{0 \leq n \leq Kb^{2N}}\gamma_n e(\alpha n)\right|^2 \der \alpha = \int_0^1\Phi_N^{2K}(\alpha + \beta)\der \alpha\\
& = \int_0^1\Phi_N^{2K}(\alpha) \der \alpha.
\end{align*}
The result now follows from Lemma \ref{lem:Moment bounds}. 
\end{proof}	

\begin{lem}[$L^2$-bound]\label{lem:L^2 bound}
Let $L,M,N\geq  1$ be integers with $L+M < N$ and let $D,Q\geq 1$. Then
$$
\sup_{\beta \in \mathbb{R}}\sum_{\substack{ d\leq D \\ q \leq Q \\ h(qd^2)}}\prod_{L < n \leq L+M}\phi_b^2\left(\left(\dfrac{h}{qd^2} + \beta\right)\left(b^n + b^{2N - n}\right)\right) \ll Q^2D^3b^M + QD^2b^{2M + o(N)}.
$$
Moreover
$$
\sup_{\beta \in \mathbb{R}}\sum_{d \leq D}\sum_{q \leq Q}\sum_{h(qd^2)}\Phi_N\left(\dfrac{h}{qd^2} + \beta\right) \ll D^2Qb^{N + o(N)}\left(\dfrac{DQ}{b^{N/2}} + 1\right).
$$	
\end{lem}

\begin{proof}
	First note the second inequality follows from the first after factoring $\Phi_N$ as $(\prod_{1 \leq n \leq N/2})(\prod_{N/2 < n < N})$ and applying Cauchy-Schwarz. With regards to the first inequality,
expanding the square, switching orders of summation, using the orthogonality of the additive characters modulo $qd^2$ and taking absolute values, the left hand side is at most
$$
QD^2 \sum_{\substack{0 \leq u_{L+1}, \ldots, u_{L+M} < b \\ 0\leq  v_{L+1}, \ldots, v_{L+M} < b}} \sum_{\substack{d\leq D \\ q \leq Q \\ qd^2 \mid S(\mathbf{u}, \mathbf{v})}}1,
$$	
where
$$
S(\mathbf{u}, \mathbf{v}) = \sum_{L < n \leq L + M}(u_n - v_n)\left(b^n + b^{2N-n}\right).
$$
Note $S(\mathbf{u}, \mathbf{v}) =0$ if and only if $u_n = v_n$ for each $L < n \leq L+M$ (by uniqueness of the $b$-ary representation of integers). Such diagonal terms contribute at most $Q^2D^3b^M$ to the sum. For the off-diagonal terms with $u_n \neq v_n$ for some $n$, one notes the inner sum is at most
$
b^{o(N)}\tau(S(\mathbf{u}, \mathbf{v})) \ll b^{o(N)} 
$
by the bounds for the divisor function and the fact $S(\mathbf{u}, \mathbf{v}) \ll b^{2N}$.
\end{proof}

\begin{lem}[Algebraic property]\label{lem:Algebraic property}
Suppose $(q,b)=1$ and let $\beta,\delta \in \mathbb{R}$ with $\delta > 0$. For integers $0 \leq M \leq N$,
$$
\sum_{h(q)}^*\prod_{M < n < N}	\phi_b^\delta\left(\left(\dfrac{h}{q} + \beta\right)\left(b^n + b^{2N-n}\right)\right) = \sum_{h(q)}^*\Phi_{N-M}^\delta\left(\dfrac{h}{q} + b^M\beta\right). 
$$
\end{lem}	

\begin{proof}
	The product on the left equals
\begin{align*}
&\prod_{0 < n < N-M}\phi_b^\delta\left(\left(\dfrac{h}{q} + \beta\right)\left(b^{n+M} + b^{2N-(n + M)}\right)\right) \\
&= \prod_{0 < n < N-M}\phi_b^\delta\left(\left(\dfrac{hb^M}{q} + \beta b^M\right)\left(b^{n} + b^{2(N-M)-n}\right)\right)\\
&= \Phi_{N-M}^\delta\left(\dfrac{b^M h}{q} + b^M \beta\right)
\end{align*}
by definition. Since $(q,b)=1$ and $\Phi_{N-M}$ is $1$-periodic,
$$
\sum_{h(q)}^*\Phi_{N-M}^\delta\left(\dfrac{b^M h}{q} + b^M \beta\right) = \sum_{h(q)}^*\Phi_{N-M}^\delta\left(\dfrac{h}{q} + b^M \beta\right)
$$ 
as required.	
\end{proof}	

\begin{lem}[$L^1$-bound involving squares]\label{lem:L^1 bound}
 Assume $D\gg b^{\epsilon N}$ and $DQ \ll b^{(2/5-\epsilon)N}$ for some small fixed $ \epsilon > 0$. Then	
$$
\sup_{\beta \in \mathbb{R}}\sum_{\substack{d \leq D \\ q \leq Q \\ (dq,b)=1}}\sum_{h(qd^2)}^*    \Phi_N\left(\dfrac{h}{qd^2} + \beta\right) \ll D^2Qb^{\delta N }
$$
for some $\delta  < 1$ depending only on $b,\epsilon$. 
\end{lem}

\begin{proof}
Let $M \geq 1$ be an integer satisfying $b^M \asymp DQ$ and factorize $\Phi_N$ as $\Phi_N = P_1P_2P_3$, where $P_1,P_2,P_3$ are the products over $1 \leq n \leq M$, $M < n \leq 2M$ and $2M < n < N$, respectively. Let $b^2 < K \ll 1$ be an integer to be chosen later and let $\ell > 2$ be defined by the equation $2/\ell + 1/2K=1$. By H\"older's inequality and positivity, 
$$
\sum_{\substack{d \leq D \\ q \leq Q \\ (dq,b)=1}}\sum_{h(qd^2)}^*    \Phi_N\left(\dfrac{h}{qd^2} + \beta\right) \leq \Sigma_1^{1/\ell} \Sigma_2^{1/\ell}\Sigma_3^{1/2K},
$$
where 
\begin{align*}
\Sigma_3 &= \sum_{\substack{d \leq D \\ q \leq Q \\ (dq,b)=1}}\sum_{h(qd^2)}^*    P_3^{2K}\left(\dfrac{h}{qd^2} + \beta\right)
\end{align*}
and for $j \in \{1,2\}$, 
$$
\Sigma_j = \sum_{\substack{d \leq D \\ q \leq Q }}\sum_{h(qd^2)}P_j^{\ell}\left(\dfrac{h}{qd^2} + \beta\right). 
$$	
For $1 \leq j \leq 2$, note $P_j^\ell = P_j^{\ell-2}P_j^2 \leq b^{(\ell-2)M}P_j^2$. Thus
$$
\Sigma_j \leq b^{(\ell-2)M}\sum_{\substack{d \leq D \\ q \leq Q }}\sum_{h(qd^2)}P_j^{2}\left(\dfrac{h}{qd^2} + \beta\right).
$$
By Lemma \ref{lem:L^2 bound} and the assumptions $b^M \asymp DQ$, $D \gg b^{\epsilon N}$, we have
$$
\Sigma_j \ll b^{o(N)}(DQ)^{\ell-2}Q^3D^4 \ll (QD)^\ell QD^{2 + o(1)} \hspace{3em}(j \in \{1,2\}).
$$
Hence 
\begin{equation}\label{product of sigmas raised to the 1/l}
\Sigma_1^{1/\ell}\Sigma_2^{1/\ell} \ll (QD)^2Q^{2/\ell}D^{4/\ell + o(1)} = Q^{3-1/2K}D^{4-1/K + o(1)} 
\end{equation}
since $2/\ell + 1/2K = 1$ by assumption. 

With regards to $\Sigma_3$, Lemmas \ref{lem:Algebraic property} and \ref{lem:Moments involving square moduli} imply
\begin{align*}
\Sigma_3 &= \sum_{\substack{d \leq D \\ q \leq Q \\ (dq,b)=1}}\sum_{h(qd^2)}^*\Phi_{N-2M}^{2K}\left(\dfrac{h}{qd^2} + \beta b^{2M}\right)\\
&\ll_\gamma 
Q\left(QD^3 + b^{2(N-2M)}\sqrt{D}\right)b^{2(K-1)(N-2M)}D^\gamma\left(1 + \dfrac{bc}{\sqrt{K}} \right)^{2(N-2M)}
\end{align*}
for some absolute constant $c > 0$ and for any $\gamma > 0$. Note we used the assumptions $b^2 < K \ll 1$ and $D \gg b^{\epsilon N}$, $DQ \ll b^{N}$ with $\epsilon > 0$ fixed. For convenience, set 
$$
F = b^{N-2M} \asymp \dfrac{b^N}{D^2Q^2}.
$$
The last holds since $b^M \asymp DQ$ by assumption. The previous inequality implies
$$
\Sigma_3 \ll_\gamma \left(Q^2D^3F^{2(K-1)} + Q\sqrt{D}F^{2K}\right) F^{2\log_b(1 + bc/\sqrt{K})} D^{\gamma}.
$$
Hence 
$$
\Sigma_3^{1/2K} \ll_\gamma \left(Q^{1/K}D^{3/2K}F^{1-1/K} + Q^{1/2K}D^{1/4K}F \right)F^{\frac{1}{K}\log_b(1 + bc/\sqrt{K})} D^{\gamma}
$$
for any $\gamma > 0$. Multiplying by $(\Sigma_1 \Sigma_2)^{1/\ell}$ and recalling (\ref{product of sigmas raised to the 1/l}) gives
\begin{align*}
(\Sigma_1 \Sigma_2)^{1/\ell}\Sigma_3^{1/2K} 
&\ll_\gamma
FQ^3D^4\left(\left(\dfrac{DQ}{F^2}\right)^{1/2K} + \dfrac{1}{D^{3/4K}}\right)F^{\frac{1}{K}\log_b(1 + bc/\sqrt{K})} D^{\gamma}\\
&=
D^2Qb^N\left(\left(\dfrac{DQ}{b^{2N/5}}\right)^{5/2K} + \dfrac{1}{D^{3/4K}} \right)F^{\frac{1}{K}\log_b(1 + bc/\sqrt{K})} D^{\gamma}.
\end{align*}
The last holds since $F \asymp b^N/D^2Q^2$. By assumption, $DQ \ll b^{(2/5-\epsilon)N}$ and $D \gg b^{\epsilon N}$. Hence the above implies
$$
(\Sigma_1 \Sigma_2)^{1/\ell}\Sigma_3^{1/2K}  \ll_\gamma
D^2Qb^Nb^{-LN/4K},
$$
where
$$
L = 3\epsilon - 4\log_b\left(1 + \frac{bc}{\sqrt{K}}\right) - 4\gamma K.
$$
Remains to note we may choose $K = K(b,\epsilon)$ large enough and $0 < \gamma = \gamma(b,\epsilon)$ small enough so that $L \geq \epsilon$. 
\end{proof}

\begin{lem}[$L^1/L^2$-bound]\label{lem:L^1/L^2-bound}
For any $D,Q,N \geq 1$ with $DQ \ll b^N$,
$$
\sup_{\beta \in \mathbb{R}}\sum_{\substack{d\leq D \\ q \leq Q \\ (dq,b)=1}}\sum_{h(qd^2)}^*\Phi_N\left(\dfrac{h}{qd^2} + \beta\right) \ll_\epsilon b^{N + \epsilon N}D^2Q\left(\dfrac{DQ}{b^{N/2}} + \left(\dfrac{Q^2}{D}\right)^{1/8}\right)
$$
for any $\epsilon > 0$. 
\end{lem}

\begin{proof}
Let $1 \leq M < N$ be an integer to be chosen later and split the product as $\Phi_N = P_1P_2$, where $P_1,P_2$ are the products over $1 \leq n \leq M$ and $M < n < N$, respectively. Set $U = b^M$. Applying Cauchy-Schwarz, the sum on the left above is at most
$
\sqrt{\Sigma_1\Sigma_2},
$
say. Lemma \ref{lem:L^2 bound} gives
$$
\Sigma_1 \ll Q^2D^3U + QD^2U^2b^{o(N)}.
$$
With regards to $\Sigma_2$, Lemmas \ref{lem:Algebraic property} and \ref{lem:Large sieve with square moduli} imply 
$$
\Sigma_2\ll_\epsilon Q\left(QD^3 + \dfrac{b^{2N}\sqrt{D}}{U^2}\right)\dfrac{b^{N + \epsilon N}}{U}
$$
for any $\epsilon > 0$.
Hence
\[
\sqrt{\Sigma_1\Sigma_2}\ \ll_\epsilon\ 
Q D^2 b^{N + \epsilon N}\Biggl(
\frac{QD}{b^{N/2}}
\;+\;
\frac{\sqrt{Qb^N}}{D^{1/4}U}
\;+\;
\frac{\sqrt {UDQ}}{b^{N/2}}
\;+\;
\frac{b^{N/2}}{D^{3/4}\sqrt U}
\Biggr)\
\]
 for any $\epsilon > 0$. If $b^N \gg D^{5/4}Q^{1/2}$, we may choose $U \asymp  b^N/D^{5/4}Q^{1/2}$. In this case,
$$
\sqrt{\Sigma_1\Sigma_2}\ \ll_\epsilon Q D^2 b^{N + \epsilon N}\left(\frac{QD}{b^{N/2}} + \left(\dfrac{Q^2}{D}\right)^{1/8}\right) .
$$
If $b^N \ll D^{5/4}Q^{1/2}$, a choice of $U \asymp DQ$ gives
$$
\sqrt{\Sigma_1\Sigma_2} \ll_\epsilon Q D^2 b^{N + \epsilon N}\dfrac{QD}{b^{N/2}}.
$$
The result now follows.
\end{proof}	

\begin{proof}[Proof of Proposition \ref{prop:equidistribution involving squares}]
For $q \sim Q$ and $d \sim D$, a Fourier expansion gives
$$
\sum_{n\in \mathscr{P}_b^*(y)}\left(\mathbf{1}_{n\equiv a(qd^2)} - \dfrac{1}{qd^2}\right) \ll \dfrac{1}{QD^2}\sum_{1 \leq h < qd^2}\left|\sum_{n \in \mathscr{P}_b^*(y)}e_{qd^2}(hn)\right|
$$
uniformly in $a$. Since every $n \in \mathscr{P}_b^*(y)$ is coprime to $b^3-b$ by definition, and $b + 1$ divides $b^3-b$, then $(n,b+1)=1$. Since every $b$-palindrome $n$ with $\lfloor \log_b n\rfloor$ odd is divisible by $b+1$, it follows $\mathbf{1}_{n \in \mathscr{P}_b^*(y)} = \mathbf{1}_{n \in \mathscr{P}_b^0(y)}\mathbf{1}_{(n,b^3-b)=1}$, where $\mathscr{P}_b^0(y)$ (defined in (\ref{P^0(x)})) is the set of all $b$-palindromes $1 \leq n \leq y$ with $\lfloor \log_b n\rfloor$ even. By the M\"obius inversion formula $\mathbf{1}_{(n,b^3-b)=1} = \sum_{r \mid (b^3-b,n)}\mu(r)$, a Fourier expansion of $\mathbf{1}_{r \mid n}$ and Lemma \ref{lem:Sum-to-product}, it follows that the left hand sides of (\ref{DQ <= x^{1/5}}) and (\ref{DQ <= x^{1/4} and Q <= x^{1/15}}) are
\begin{align*}
\ll 
(\log x)^2\max_{\substack{0 \leq N \leq \frac{1}{2}\log_b x \\ k \in \mathbb{Z}}}\mathscr{S}(D,Q,N,k),
\end{align*}
where
$$
\mathscr{S}(D,Q,N,k) = \dfrac{1}{QD^2}\sum_{\substack{d \sim D \\ q \sim Q \\ (dq,b^3-b)=1}}\sum_{1 \leq h < qd^2}\Phi_N\left(\dfrac{h}{qd^2} + \dfrac{k}{b^3-b}\right).
$$
Note we also used the constraint $(dq,b)=1$. 
The cases when $N \leq (\frac{1}{2} - \epsilon)\log_b x$ are unsubstantial. Indeed, by Lemma \ref{lem:L^2 bound} and the assumption $DQ \ll x^{1/4 - \epsilon}$, $\mathscr{S}(D,Q,N,k) \ll x^{\frac{1}{2} - \epsilon + o(1)}$ for $N \leq (\frac{1}{2}-\epsilon)\log_b x$. Thus the inequality above becomes
\begin{equation}\label{N large close to 1/2log_b(x)}
\ll x^{\frac{1}{2} - \epsilon +o(1)}+
(\log x)^2\max_{\substack{(\frac{1}{2}-\epsilon)\log_b x \leq N \leq \frac{1}{2}\log_b x \\ k \in \mathbb{Z}}}\mathscr{S}(D,Q,N,k).
\end{equation}
We may thus assume $(\frac{1}{2}-\epsilon)\log_b x \leq N \leq \frac{1}{2}\log_b x$ in what follows.

Splitting the sum according to the GCD of $h,d$ and substituting variables gives
\begin{align*}
\mathscr{S}(D,Q,N,k) &= \dfrac{1}{QD^2}\sum_{\substack{dr \sim D \\ q \sim Q \\ (dqr,b^3-b)=1}}\sum_{\substack{1 \leq h < qrd^2 \\ (h,d)=1}}\Phi_N\left(\dfrac{h}{qrd^2} + \dfrac{k}{b^3-b}\right)\\
&\ll 
\dfrac{1}{QD^2}\sum_{\substack{d \leq D \\ qr \asymp QD/d \\ (dqr,b^3-b)=1}}\sum_{\substack{1 \leq h < qrd^2 \\ (h,d)=1}}\Phi_N\left(\dfrac{h}{qrd^2} + \dfrac{k}{b^3-b}\right)\\
&\leq
\dfrac{1}{QD^2}\sum_{\substack{d \leq D \\ q \asymp QD/d \\ (dq,b^3-b)=1}}\tau(q)\sum_{\substack{1 \leq h < qd^2 \\ (h,d)=1}}\Phi_N\left(\dfrac{h}{qd^2} + \dfrac{k}{b^3-b}\right).
\end{align*}
We now split the last sum according to the GCD of $h,q$ and substitute variables as before. This shows the above equals
\begin{align*}
& \dfrac{1}{QD^2}\sum_{\substack{d \leq D \\ qr \asymp QD/d  \\ dq > 1\\ (dqr,b^3-b)=1}}\tau(qr)\sum_{\substack{h( qd^2) }}^*\Phi_N\left(\dfrac{h}{qd^2} + \dfrac{k}{b^3-b}\right)\\
&\ll 
\dfrac{(\log x)}{D}\sum_{\substack{d \leq D \\ q \ll DQ/d  \\ dq > 1\\ (dq,b^3-b)=1}}\dfrac{\tau(q)}{dq}\sum_{\substack{h( qd^2) }}^*\Phi_N\left(\dfrac{h}{qd^2} + \dfrac{k}{b^3-b}\right)
\end{align*}
after summing over $r \asymp DQ/dq$ while using the inequalities $\tau(qr) \leq \tau(q)\tau(r)$ and $\sum_{r \asymp DQ/dq}\tau(r) \ll DQ(\log x)/dq $ for $DQ \ll x$. Thus after dyadic decompositions,
\begin{equation}\label{sup involving S and T}
\mathscr{S}(D,Q,N,k) \ll (\log x)^3\sup_{\substack{1 \leq C \leq D \\ 1 \leq R \ll DQ/C}}\mathscr{T}(C,D,R,N,k),
\end{equation}
where
$$
\mathscr{T}(C,D,R,N,k) = \dfrac{R^{o(1)}}{CDR}\sum_{\substack{c \sim C \\ r \sim R \\ cr > 1 \\ (cr,b^3-b)=1}}\sum_{h(rc^2)}^* \Phi_N\left(\dfrac{h}{rc^2} + \dfrac{k}{b^3-b}\right).
$$

We bound $\mathscr{T}(C,D,R,N,k)$ according to the sizes of $C,D,Q,R$. 

{\bf Case 1 ($CR \ll x^{1/5-\epsilon}$):} We split this case into two subcases.

{\bf Subcase 1.1 ($C \ll x^{\epsilon/2}$):} Since $C \leq D$ as well by the constraint under the $\sup$ above, necessarily $C \ll \min(D, x^{\epsilon/2})$. By bounds for the divisor function and Lemma \ref{lem:T-Panario L^1 bound},
\begin{align*}
\mathscr{T}(C,D,R,N,k) &\ll \dfrac{(CR)^{o(1)}}{CDR}\sum_{\substack{1 < q \leq C^2R \\  (q,b^3-b)=1}}\sum_{h(q)}^*\Phi_N\left(\dfrac{h}{q} + \dfrac{k}{b^3-b}\right)\\
&\ll_\delta
\dfrac{(CR)^{o(1)}}{CDR}\left(C^4R^2b^{3N/5+\delta N} + b^N\left(C^2R\right)^{1 - \sigma_1}\exp\left(-\dfrac{\sigma_\infty N}{\log CR}\right)\right)\\
&\ll_\delta x^{1/2-\epsilon/2 + \delta} + x^{1/2}\left(C^2R\right)^{-\sigma_1 + o(1)}\exp\left(-\dfrac{\sigma_\infty' \log x}{\log C^2 R}\right)\\
&\ll_\delta 
x^{1/2-\epsilon/2 + \delta} + \sqrt{x}\exp\left(-\sigma \sqrt{\log x}\right)\\
& \ll \sqrt{x}\exp\left(-\sigma \sqrt{\log x}\right)
\end{align*}
for any $\delta > 0$ and some $\sigma_1,\sigma_\infty,\sigma_\infty',\sigma > 0$ depending only on $b,\delta,\epsilon$.
Note we used the assumptions $CR \leq x^{1/5-\epsilon}$, $C \ll\min(D, x^{\epsilon/2})$, $(\frac{1}{2}-\epsilon)\log_b x \leq N \leq \frac{1}{2}\log_b x$ together with the inequality $ \exp(-\kappa K-\sigma_\infty' \log (x)/K) \ll \exp(-\kappa_0 \sqrt{\log x})$ for any $K\geq 0$ and some $\kappa_0$ depending only on $\kappa,\sigma_\infty' >0$. 

{\bf Subcase 1.2 ($C \gg x^{\epsilon/2}$):} Since $x \geq b^{2N}$ by assumption, then $C \gg b^{\epsilon N}$. Since $CR \ll x^{1/5-\epsilon}$ and $ (\frac{1}{2} - \epsilon)\log_b x \leq N$, one can show $CR \ll b^{2N/5 - \epsilon N}$. Now Lemma \ref{lem:L^1 bound} implies $\mathscr{T}(C,D,R,N,k) \ll x^{\delta/2}$ for some $\delta < 1$ depending only on $b,\epsilon$. Combining this with the conclusion of the previous subcase, we obtain that if $CR \ll x^{1/5-\epsilon}$, then $\mathscr{T}(C,D,R,N,k) \ll \sqrt{x}\exp\left(-\sigma \sqrt{\log x}\right)$. This also implies (\ref{DQ <= x^{1/5}}) since if $DQ \ll x^{1/5-\epsilon}$, then $CR \ll x^{1/5-\epsilon}$. Indeed, by the constraints under the $\sup$ in (\ref{sup involving S and T}), $CR \ll DQ$. With regards to (\ref{DQ <= x^{1/4} and Q <= x^{1/15}}), note it follows from (\ref{DQ <= x^{1/5}}) in the case when $DQ \ll x^{1/5-\epsilon}$. Thus to conclude the proof of Proposition \ref{prop:equidistribution involving squares}, it only remains to consider the case when $x^{1/5-\epsilon} \ll DQ \ll x^{1/4-\epsilon}$ with $Q \ll x^{1/15-\epsilon}$.

{\bf Case 2} ($x^{1/5-\epsilon} \ll DQ \ll x^{1/4-\epsilon}$ with $Q \ll x^{1/15-\epsilon}$): By Lemma \ref{lem:L^1/L^2-bound} and the assumptions $C \leq D$, $R \ll DQ/C$, $DQ \ll x^{1/4-\epsilon}$ and $D\gg x^{1/5-\epsilon}/Q$,
\begin{align*}
\mathscr{T}(C,D,R,N,k) 
&\ll_\delta 
x^{1/2-\epsilon + \delta} + x^{1/2 + \delta}\dfrac{C}{D}\left(\dfrac{R^2}{C}\right)^{1/8}\\
& \ll_\delta 
x^{1/2-\epsilon + \delta} + x^{1/2 + \delta}\left(\dfrac{Q^2}{D}\right)^{1/8}\\
&\ll_\delta
x^{1/2-\epsilon + \delta} + x^{1/2 + \delta}\left(\dfrac{Q^3}{x^{1/5-\epsilon}}\right)^{1/8}
\end{align*}
for any $\delta > 0$. Now the result follows from the assumption $Q\ll x^{1/15-\epsilon}$. 
\end{proof}



\section{Proof of Theorem \ref{thm:equidistribution of square-free palindromes}}\label{sec:Proof of Theorem equidistribution of square-free palindromes}

We now proceed with a proof of Theorem \ref{thm:equidistribution of square-free palindromes}.
Let $1 \leq Q \leq x^{1/16-\epsilon}$ and consider the sum
\begin{equation}\label{E(Q)}
E(Q) = \sum_{\substack{q \sim  Q \\ (q,b^3-b)=1}}\sup_{y\leq x}\max_{(a,q)=1}\left|\sum_{\substack{n\in \mathscr{P}_b^*(y) \\ n\equiv a (q)}}\mu^2(n) - \dfrac{6\mathfrak{S}\left(b^3q-bq\right)|\mathscr{P}_b^*(y)|}{\pi^2q}\right|,
\end{equation}
where, for a natural number $k$, 
$$
\mathfrak{S}(k) = \prod_{p \mid k}\left(1 - \dfrac{1}{p^2}\right)^{-1}.
$$
Clearly $\mathfrak{S}$ is multiplicative; particularly $\mathfrak{S}(b^3q-bq) = \mathfrak{S}(b^3-b)\mathfrak{S}(q)$ for $(q,b^3-b)=1$.  
By the M\"obius inversion formula, $\mu^2(n) = \sum_{d^2 \mid n}\mu(d)$. Hence the $n$-sum equals
$$
\sum_{d \leq \sqrt{x}}\mu(d)\sum_{\substack{n \in \mathscr{P}_b^*(y) \\ n\equiv a (q) \\ n \equiv 0(d^2)}}1 = \sum_{\substack{d \leq \sqrt{x} \\ (d,q(b^3-b))=1}}\mu(d)\sum_{\substack{n \in \mathscr{P}_b^*(y) \\ n\equiv a (q) \\ n \equiv 0(d^2)}}1.
$$
To see why the last equality holds, note that since $(a,q)=1$ and $n\equiv a (q)$, then $(n,q)=1$. Moreover $(n,b^3-b)=1$ for $n \in \mathscr{P}_b^*(y)$. Since $d \mid n$, consequently $(d,(b^3-b)q)=1$ as needed. 

Note
$$
\dfrac{6}{\pi^2}\mathfrak{S}\left(b^3q-bq\right) 
 = \sum_{\substack{d \geq 1 \\ (d,q(b^3-b))=1}} \dfrac{\mu(d)}{d^2}  
 = 
 \sum_{\substack{d \leq \sqrt{x} \\ (d,q(b^3-b))=1}} \dfrac{\mu(d)}{d^2} + O\left(\dfrac{1}{\sqrt{x}}\right).
$$
The first equality holds by an Euler product expansion of the sum on the right. Since $|\mathscr{P}_b^*(y)| \ll \sqrt{y} \leq \sqrt{x}$ for $y \leq x$, it follows
\begin{align}
E(Q) &\ll 1 +  \sum_{\substack{ q \sim Q \\ (q,b^3-b)=1 }}\sup_{y \leq x}\max_{(a,q)=1}\sum_{\substack{d \leq \sqrt{x} \\ (d,q(b^3-b))=1}}\left|\sum_{\substack{n \in \mathscr{P}_b^*(y) \\ n\equiv a (q) \\ n \equiv 0 (d^2)}}1 - \dfrac{|\mathscr{P}_b^*(y)|}{qd^2}\right| \nonumber\\
 &\ll 1 + (\log x)\sup_{1 \leq D \leq \sqrt{x}}E(Q,D), \label{E(Q) bound in terms of E(D,Q)}
\end{align}
where
\begin{equation}\label{E(Q,D)}
E(Q,D) = \sum_{\substack{q \sim Q \\ (q,b^3-b)=1 \\ }}\sup_{y \leq x}\max_{(a,q)=1}\sum_{\substack{d\sim D \\ (d,q(b^3-b))=1}}\left|\sum_{\substack{n \in \mathscr{P}_b^*(y) \\ n\equiv a (q) \\ n \equiv 0 (d^2)}}1 - \dfrac{|\mathscr{P}_b^*(y)|}{qd^2}\right|.
\end{equation}
Since $(d,q)=1$, we may combine the two 
congruences via the Chinese remainder theorem; thus 
$$
E(Q,D) \leq \sum_{\substack{d \sim D \\ q \sim Q \\ (dq,b^3-b)=1 \\ (d,q)=1}}\sup_{y \leq x}\max_{\substack{(a,q)=1 \\ d^2 \mid a}}\left|\sum_{\substack{n \in \mathscr{P}_b^*(y) \\ n\equiv a (qd^2)}}1 - \dfrac{|\mathscr{P}_b^*(y)|}{qd^2}\right|.
$$
If $DQ \ll x^{1/4-\epsilon}$, then, as $Q \leq x^{1/16-\epsilon}$ by assumption, the inequality in (\ref{DQ <= x^{1/4} and Q <= x^{1/15}}) of Proposition \ref{prop:equidistribution involving squares} implies $E(Q,D) \ll \sqrt{x}\exp(-\sigma \sqrt{\log x})$ for some $\sigma > 0$ depending only on $b,\epsilon$. We may then assume $DQ \gg x^{1/4-\epsilon}$. In this case, the definition of $E(Q,D)$ in (\ref{E(Q,D)}) implies
\begin{align*}
&E(Q,D) \ll  \sum_{\substack{q \sim Q \\ (q,b)=1 }}\max_{\substack{(a,q)=1 }}\sum_{d \sim D}\sum_{\substack{n \in \mathscr{P}_b^*(x) \\ d^2 \mid n\equiv a (q)}}1 + \sqrt{x}\sum_{\substack{d \sim D \\ q \sim Q }}\dfrac{1}{qd^2}\\
&\ll
\dfrac{\sqrt{x}}{D} + (\log x)\sum_{\substack{ q \sim Q \\ (q,b)=1 }}\max_{\substack{(a,q)=1  \\0 \leq L \leq \log_b x}}\sum_{d \sim D}\sum_{\substack{ \ell \in \Pi_b(L) \\ (\ell,b)=1\\d^2 \mid \ell \\ \ell \equiv a(q)}}1\\
&\ll_\delta
\dfrac{\sqrt{x}}{D} + x^{1/2 + \delta}\sqrt{Q}
\left( \left(\dfrac{Q^{1/2}}{D}\right)^{1/5} + \left(\dfrac{Q^{1/3}}{D}\right)^{3/16} + \dfrac{D^{3/5}Q^{7/10} + (DQ)^{5/8} + Q}{\sqrt{x}}   \right)
\end{align*}
for any $\delta > 0$, by Theorem \ref{thm: square divisors}. The right hand side above is
\begin{align*}
&\ll_\delta x^{1/2 + \delta}\left(\left(\dfrac{Q^{3}}{D}\right)^{1/5} + \left(\dfrac{Q^3}{D}\right)^{3/16} + \dfrac{D^{3/5}Q^{6/5} + D^{5/8}Q^{9/8} + Q^{3/2}}{\sqrt{x}}\right)\\
&\ll_\delta
x^{1/2 + \delta}\left(\left(\dfrac{Q^{4}}{x^{1/4-\epsilon}}\right)^{1/5} + \left(\dfrac{Q^4}{x^{1/4-\epsilon}}\right)^{3/16} + \dfrac{x^{3/10}Q^{6/5} + x^{5/16}Q^{9/8} + Q^{3/2}}{\sqrt{x}}\right).
\end{align*}
The last holds by the assumption $x^{1/4-\epsilon}/Q \ll D \ll x^{1/2}$. Now Theorem \ref{thm:equidistribution of square-free palindromes} follows from the assumption $Q \leq x^{1/16 - \epsilon}$.

\section{Acknowledgements}\label{sec:Acknowledgements}

We thank Universidad de Cantabria for its warm hospitality and for making our time in Santander so memorable, both personally and mathematically. We are especially grateful to Ana G\'omez-P\'erez, Domingo G\'omez-P\'erez and Isabel Pirsic for their exceptional kindness. We thank Igor E.\ Shparlinski for bringing the problem to our attention and for reading an earlier draft of this work. We thank Daniel Panario for earlier discussions and reading a preliminary draft. We thank Daniel R.\ Johnston and Bryce Kerr for collegiate discussions and encouragement. Part of the work was supported by
C\'atedra Universidad de Cantabria - INCIBE de nuevos retos en ciberseguridad, financed by European Union NextGeneration-EU, the Recovery Plan, Transformation and Resilience, through INCIBE. A preliminary stage of this project was supported in part by the Natural Sciences and Engineering Research Council of Canada (NSERC).

\appendix

\section{Large square divisors and digital $A$-process}\label{sec: large squares}

Here we prove the following proposition. The argument in the proof is motivated in part by Mauduit-Rivat's \cites{Mauduit-Rivat(2009), Mauduit-Rivat(2010)} use of van der Corput's inequality to truncate the sum-of-digits function. See Lemme 16 in \cite{Mauduit-Rivat(2009)} and Lemme 5 in \cite{Mauduit-Rivat(2010)}. Rather than apply the concept to exponential sums as they do there, we can exploit the defining structure of the palindromes and apply the idea on the set directly.
The proposition shows that, for any fixed $\delta > 0$, there are few palindromes in $\Pi_b(2L)$ divisible by squares $n^2$ with $n \gg x^{3/8 + \delta}$, where $x = b^{2L}$. Note from the proof that the statement of the proposition, and argument in its proof, can be easily adapted for $\Pi_b(L)$ with $L$ odd as well. 

\begin{prop}\label{prop: tail bound}
	For any positive integer $L$ with $x \asymp b^{2L}$ and any $N \geq 1$,
	\begin{equation}\label{eqn: bound for T}
	\sum_{n \sim N}\sum_{\substack{\ell \in \Pi_b(2L) \\ (\ell,b)=1\\ n^2 \mid \ell}}1 \ll \dfrac{x^{3/10+o(1)}}{N^{4/5}}\#\Pi_b(2L).
	\end{equation}
\end{prop}

We need the following three lemmas.

\begin{lem}[Banks-Shparlinski \cite{Banks-Shparlinski}, Theorem 7]\label{lem: Banks-Shparlinski}
	For any positive integers $L,q$, 
	$$
	\sum_{\substack{\ell \in \Pi_b(L) \\ q \mid \ell}}1 \ll \dfrac{\#\Pi_b(L)}{\sqrt{q}}.
	$$
\end{lem}

\begin{lem}[van der Corput]\label{lem: van der Corput}
	Let $z_1, \ldots, z_N$ be complex numbers. Then for any integer $H \geq 1$, 
	$$
	\left|\sum_{1 \leq n \leq N} z_n\right|^2 \leq \dfrac{N+H-1}{H} \sum_{|h| < H} \left(1 - \dfrac{|h|}{H} \right) \sum_{\substack{1 \leq n\leq  N \\ 1 \leq n+h \leq N }} z_{n+h}\overline{z_n}. 
	$$
\end{lem}

\begin{proof}
	See for instance Lemma 8.17 of Iwaniec-Kowalski \cite{Iwaniec-Kowalski}.	
\end{proof}

In the following, $\tau(q) = \sum_{d \mid q}1$ is the divisor function.  

\begin{lem}\label{lem: cong bound}
	Let $M,N,q $ be positive integers. Then
	$$
	\sum_{m \leq M } \max_{a \in \mathbb{Z}}\sum_{\substack{n \leq N \\ mn \equiv a (q)}}1 \leq \dfrac{MN}{q}\tau(q) + M\tau(q).
	$$
\end{lem}

\begin{proof}
	Splitting the $m$-sum according to the GCD of $m,q$ and substituting variables, the left hand side equals
	$$
	\sum_{\substack{d\mid q \\ m \leq dM/q \\ (m,d)=1}}\max_{a \in \mathbb{Z}} \sum_{\substack{n \leq N \\ n \equiv a(d)}}1 .
	$$ 
	The inner sum is at most $N/d + 1$ 
	and the result follows.
\end{proof}

\begin{proof}[{Proof of Proposition \ref{prop: tail bound}}]
	We may assume $x^{3/8} \ll N \ll \sqrt{x}$ as otherwise the statement is trivial. Note 
	\begin{equation}\label{eqn: def of T}
	\mathscr{T} := \sum_{n \sim N}\sum_{\substack{\ell \in \Pi_b(2L) \\ (\ell,b)=1\\ n^2 \mid \ell }}1 \ll \sup_{M \asymp x/N^2} S(L,M,N),
	\end{equation}
	where
	$$
	S(L,M,N) = \sum_{\substack{m \sim M \\ n \sim N \\ mn^2 \in \Pi_b(2L)\\ (mn,b)=1}}1.
	$$
	
	By the Cauchy-Schwarz inequality, 
	$$
	S^2(L,M,N) \leq  M \sum_{\substack{m \sim M \\ (m,b)=1}}\left(\sum_{\substack{n \sim N \\ (n,b)=1\\mn^2 \in \Pi_b(2L)}}1 \right)^2.
	$$
	Fixing a $\delta > 0$ sufficiently small and
	applying Lemma \ref{lem: van der Corput}, we obtain
	$$
	S^2(L,M,N) \ll \dfrac{MN }{H} \sum_{\substack{m \sim M \\ |h| \leq H \\ n, n + h \sim N \\ (mn,b)=1\\mn^2, mn^2 + 2mnh + mh^2 \in \Pi_b(2L) }} 1
	$$
	for any $1 \ll H \ll  Nx^{-\delta}$. The term with $h=0$ contributes
	$
	\ll MN \mathscr{T}/H  
	$
	to $S^2(L,M,N)$. Thus
	\begin{equation}\label{eqn: bound for S^2}
	S^2(L,M,N) \ll \dfrac{MN \mathscr{T}}{H}   + \dfrac{MN \mathscr{S}}{H},
	\end{equation}
	where
	$$
	\mathscr{S} = \sum_{\substack{m \sim M \\ 1 \leq |h| \leq H \\ n, n + h \sim N \\ (mn,b)=1\\mn^2, mn^2 + 2mnh + mh^2 \in \Pi_b(2L) }} 1.
	$$
	
	Let us first consider the contribution of the terms with $h > 0$; that is, consider
	$$
	\mathscr{S}_+ = \sum_{\substack{m \sim M \\ 1 \leq h \leq H \\ n, n + h \sim N \\ (mn,b)=1 \\ mn^2, mn^2 + 2mnh + mh^2 \in \Pi_b(2L) }} 1.
	$$
	Note that for any $h,m,n$ as above with $H \ll N x^{-\delta}$, we have $2mnh + mh^2 \leq 3MNH \asymp xH/N \ll x^{1-\delta}$.
	For an integer $\lambda$ satisfying
	\begin{equation}\label{eqn: assumptions on lambda}
	\max\left(L,  \log_b(3MNH)\right) < \lambda < 2L,
	\end{equation}
	define the set
	$$
	\mathcal{G}_\lambda = \left\{ (k,\ell) \in \mathbb{N}^2 \ : \ \exists j > \lambda \ \text{s.t.} \ d_j(k ) \neq d_j(\ell)\right\}. 
	$$
	We will choose $\lambda$ later.
	
	Let us now split $\mathscr{S}_+$ as $\mathscr{S}_+ = \mathscr{S}_{+1} + \mathscr{S}_{+2}$,
	where
	\begin{align*}
	\mathscr{S}_{+1} &= \sum_{\substack{m \sim M \\ 1 \leq h \leq H \\ n, n + h \sim N \\ (mn,b)=1\\ (mn^2, mn^2 + 2mnh + mh^2) \in \Pi_b(2L)^2 \cap \mathcal{G}_\lambda}} 1,\\
	\mathscr{S}_{+2} &= \sum_{\substack{m \sim M \\ 1 \leq h \leq H \\ n, n + h \sim N \\ (mn,b)=1\\(mn^2, mn^2 + 2mnh + mh^2) \in \Pi_b(2L)^2 \setminus \mathcal{G}_\lambda}} 1.
	\end{align*}
	When we compute the sum $mn^2 + 2mnh + mh^2$, the $b$-adic digits of $mn^2$ with index $j > \lambda$ are unchanged unless there occurs a carry propagation. In order for the latter to happen, we must have $d_j(mn^2) = b-1$ for each $\log_b(3MNH) < j \leq \lambda$.
	Consequently,
	\begin{align*}
	\mathscr{S}_{+1} &\leq H \sum_{\substack{\ell \in \Pi_b(2L) \\  n\sim N  \\  d_j(\ell) = b-1 \ \forall \  \log_b(3MNH) < j \leq \lambda   \\   n^2 \mid \ell \\}}1.
	\end{align*}
	Note $MNH \gg MN \asymp x/N \gg x^{1/2} \asymp b^L$ by assumption. Then 
	the number of $b$-palindromes in $\Pi_b(2L)$, with such prescribed digits, is $\ll b^{L-(\lambda - \log_b(MNH))} \asymp \sqrt{x}MNH  b^{-\lambda}$. Any such has at most $x^{o(1)}$ square divisors. Thus, since $MN^2 \asymp x$,
	\begin{equation}\label{bound for S+1}
	\mathscr{S}_{+1} \ll \dfrac{x^{3/2 + o(1)}H^2 }{N b^{\lambda}}. 
	\end{equation}
	
	Consider now $\mathscr{S}_{+2}$. Let $h,m,n$, obeying the size restrictions above with $(m,b)=1$, be such that 
	$(mn^2, mn^2 + 2mnh + mh^2) \in \Pi_b(2L )^2 \setminus \mathscr{G}_\lambda$. Since the pair is not in $\mathscr{G}_\lambda$, we must have $d_j(mn^2 + 2mnh + mh^2) = d_j(mn^2)$ for each $j > \lambda$. Since both $mn^2, mn^2 + 2mnh + mh^2 \in \Pi_b(2L )$, this implies $d_j(mn^2 + 2mnh + mh^2) = d_j(mn^2)$ for each $0 \leq j \leq 2L - \lambda-1$, whence $d_j(2mnh + mh^2) = 0$ for each $0 \leq j \leq 2L-\lambda-1$. Consequently, $2mnh + mh^2$ is divisible by $b^{2L-\lambda}$. Since $(m,b)=1$ by assumption, we may then bound $\mathscr{S}_{+2}$ as
	\begin{align}\label{first bound for S+2}
	\mathscr{S}_{+2} 
	&\leq 
	\sum_{1 \leq h \leq H} \sum_{\substack{n\sim N \\ 2hn \equiv -h^2 (b^{2L-\lambda})}}\sum_{\substack{m \sim M \\ (m,b)=1 \\ mn^2 \in \Pi_b(2L)}}1.
	\end{align}
	Lemma \ref{lem: Banks-Shparlinski} gives
	$$
	\sum_{\substack{m \sim M \\ (m,b)=1 \\ mn^2 \in \Pi_b(2L)}}1 \ll \dfrac{\sqrt{x}}{N}
	$$
	for any $n \sim N$. We also have
	\begin{align*}
	\sum_{1 \leq h \leq H } \sum_{\substack{n \sim N \\ 2hn \equiv -h^2 (b^{2L-\lambda})}} 1 &\leq \sum_{1 \leq m \leq 2H } \max_{a \in \mathbb{Z}}\sum_{\substack{n \sim N \\ mn \equiv a (b^{2L-\lambda})}} 1 \\
	&\ll \dfrac{HNb^{\lambda}L^{\omega(b)}}{x} + HL^{\omega(b)}
	\end{align*}
	by Lemma \ref{lem: cong bound},
	where $\omega(b)$ is the number of distinct prime divisors of $b$.
	Then (\ref{first bound for S+2}) yields
	\begin{equation}\label{eqn: bound for S+2}
	\mathscr{S}_{+2} \ll \dfrac{b^{\lambda}HL^{\omega(b)}}{\sqrt{x}} + \dfrac{\sqrt{x}H L^{\omega(b)}}{N}.
	\end{equation}
	Thus by (\ref{bound for S+1}) and (\ref{eqn: bound for S+2}),
	\begin{equation}\label{bound for S+}
	\mathscr{S}_+ \ll \dfrac{x^{3/2 +o(1)} H^2}{Nb^{\lambda}} + \dfrac{b^{\lambda}HL^{\omega(b)}}{\sqrt{x}} + \dfrac{\sqrt{x}H L^{\omega(b)}}{N}.
	\end{equation}
	The bound is minimized when $b^{\lambda} \asymp x \sqrt{H/N}$. We may set $\lambda = \lfloor \log_b x \sqrt{H/N} \rfloor$ since this does not contradict the assumption (\ref{eqn: assumptions on lambda}) on the size of $\lambda$. Indeed, this follows from the assumption $ 1 \ll H \ll Nx^{-\delta}$. With this choice of $\lambda$, (\ref{bound for S+}) becomes
	\begin{equation}\label{bound for S+ part 2}
	\mathscr{S}_+ \ll \sqrt{\dfrac{x^{1+o(1)} H^3}{N}} + \dfrac{\sqrt{x}H L^{\omega(b)}}{N} \ll \sqrt{\dfrac{x^{1+o(1)} H^3}{N}}.
	\end{equation}

	It remains to consider the contribution of the terms with $h < 0$ to $\mathscr{S}$; that is, to consider
	$$
	\mathscr{S}_- = \sum_{\substack{m\sim M \\ 1 \leq h \leq H \\ n,n-h \sim N \\ (mn,b)=1 \\ mn^2, mn^2 - 2mnh + mh^2 \in \Pi_b(2L)}}1.
	$$
	Here similar arguments and bounds apply, the main difference being that the analogue $\mathscr{S}_{-1}$ of $\mathscr{S}_{+1}$ is bounded above by
	$$
	H \sum_{\substack{\ell \in \Pi_b(2L) \\  n\sim N \\  d_j(\ell) = 0 \ \forall \  \log_b(3MNH) < j \leq \lambda  \\ n^2 \mid \ell \\}}1.
	$$
	This is bounded by the same quantity there. In conclusion, 
	$$
	\mathscr{S} \ll \sqrt{\dfrac{x^{1+o(1)} H^3}{N}}.
	$$
	Inserting this in (\ref{eqn: bound for S^2}) gives
	\begin{align*}
	S^2(L,M,N) &\ll \dfrac{MN \mathscr{T}}{H} + M \sqrt{x^{1 + o(1) }NH}\\
	&\asymp 
	\dfrac{x \mathscr{T}}{NH} + \dfrac{x^{3/2 + o(1)} H^{1/2}}{N^{3/2}}
	\end{align*}
	for any $M \asymp x/N^2$, where $\mathscr{T}$ is defined as in (\ref{eqn: def of T}).
	The bound is minimized when $H = \mathscr{T}^{2/3}N^{1/3}x^{-1/3}$, but some care is needed to ensure that this does not contradict our earlier assumption $1 \ll H \ll Nx^{-\delta}$. Note that by Lemma \ref{lem: Banks-Shparlinski}, $\mathscr{T} \ll \sqrt{x}$; hence $\mathscr{T}^{2/3}N^{1/3}x^{-1/3} \ll N^{1/3} \ll Nx^{-\delta}$. The condition $\mathscr{T}^{2/3}N^{1/3}x^{-1/3} \gg 1$ is satisfied as long as $\mathscr{T} \gg \sqrt{x/N}$. One may indeed assume this is the case, as otherwise (\ref{eqn: bound for T}) holds trivially. We may thus set $H = \mathscr{T}^{2/3}N^{1/3}x^{-1/3}$ above. Doing so and recalling (\ref{eqn: def of T}) we obtain
	$$
	\mathscr{T} \ll \dfrac{x^{2/3 + o(1)}\mathscr{T}^{1/6}}{N^{2/3}};
	$$
	hence
	$
	\mathscr{T} \ll x^{4/5 + o(1)}N^{-4/5}. 
	$
	Now the result follows from $\sqrt{x} \asymp b^{L} \asymp \#\Pi_b(2L)$. 
\end{proof}

\end{document}